\documentclass[11pt,reqno]{amsart}

\newtheorem{thm}{Theorem}[section]
\newtheorem*{thm*}{Theorem}
\newtheorem{lem}[thm]{Lemma}
\newtheorem*{lem*}{Lemma}

\newtheorem{cor}[thm]{Corollary}
\newtheorem{claim}[thm]{Claim}
\newtheorem{prop}[thm]{Proposition}

\theoremstyle{definition}

\newtheorem{assump}[thm]{Assumption}
\newtheorem{case}{Case}\renewcommand{\thecase}{}
\newtheorem*{case*}{Case}

\newtheorem{defn}[thm]{Definition}
\newtheorem*{defn*}{Definition}

\newtheorem*{exmp*}{Example}

\newtheorem{rmk}[thm]{Remark}
\newtheorem*{rmk*}{Remark}
\newtheorem{step}{Step}\renewcommand{\thestep}{}

\theoremstyle{remark}

\makeatletter
\def\alphenumi{
  \def\theenumi{\alph{enumi}}
  \def\p@enumi{\theenumi}
  \def\labelenumi{(\@alph\c@enumi)}}
\makeatother

\makeatletter
\def\thecase{\@arabic\c@case}
\makeatother

\makeatletter
\def\thestep{\@arabic\c@step}
\makeatother

\newcount\hh
\newcount\mm
\mm=\time
\hh=\time
\divide\hh by 60
\divide\mm by 60
\multiply\mm by 60
\mm=-\mm
\advance\mm by \time
\def\hhmm{\number\hh:\ifnum\mm<10{}0\fi\number\mm}

\let\oldmarginpar\marginpar
\renewcommand\marginpar[1]{\-\oldmarginpar[\raggedleft\footnotesize #1]%
{\raggedright\footnotesize #1}}

\renewcommand\emptyset{\varnothing}

\newcommand\DD{\mathbb{D}}
\newcommand\EE{\mathbb{E}}

\newcommand\NN{\mathbb{N}}

\newcommand\PP{\mathbb{P}}
\newcommand\QQ{\mathbb{Q}}
\newcommand\RR{\mathbb{R}}

\newcommand\cB{{\mathcal{B}}}
\newcommand\cC{{\mathcal{C}}}

\newcommand\cF{{\mathcal{F}}}
\newcommand\cG{{\mathcal{G}}}
\newcommand\cH{{\mathcal{H}}}

\newcommand\cN{{\mathcal{N}}}

\newcommand\cX{{\mathcal{X}}}

\newcommand\cZ{{\mathcal{Z}}}

\newcommand\eps{\varepsilon}

\newcommand\diag{\operatorname{diag}}

\newcommand\dist{\operatorname{dist}}

\newcommand{\essinf}{\operatornamewithlimits{ess\ inf}}
\newcommand{\esssup}{\operatornamewithlimits{ess\ sup}}

\numberwithin{equation}{section}

\usepackage{amssymb}
\usepackage{hyperref}
\usepackage{mathrsfs}
\usepackage{slashed}
\usepackage[usenames]{color}
\usepackage{url}
\usepackage{verbatim}

\hypersetup{pdftitle={Harnack Inequalities for Degenerate Diffusions}}

\hypersetup{pdfauthor={Charles L. Epstein and Camelia A. Pop}}

\begin{document}

\title[Harnack Inequalities]{Harnack Inequalities for Degenerate Diffusions}
\author[C. Epstein]{Charles L. Epstein}
\address[CLE]{Department of Mathematics, University of Pennsylvania, 209 South 33rd Street, Philadelphia, PA 19104-6395}
\email{cle@math.upenn.edu}

\author[C. Pop]{Camelia A. Pop}
\address[CP]{Department of Mathematics, University of Pennsylvania, 209 South 33rd Street, Philadelphia, PA 19104-6395}
\email{cpop@math.upenn.edu}

\date{\today{ }\hhmm}

\begin{abstract}
We study various probabilistic and analytical properties of a class of degenerate diffusion operators arising in Population Genetics, the so-called generalized Kimura diffusion operators \cite{Epstein_Mazzeo_2010, Epstein_Mazzeo_annmathstudies, Epstein_Mazzeo_cont_est}. Our main results is a stochastic representation of weak solutions to a degenerate parabolic equation with singular lower-order coefficients, and the proof of the scale-invariant Harnack inequality for nonnegative solutions to the Kimura parabolic equation. The stochastic representation of solutions that we establish is a considerable generalization of the classical results on Feynman-Kac formulas concerning the assumptions on the degeneracy of the diffusion matrix, the boundedness of the drift coefficients, and on the a priori regularity of the weak solutions.
\end{abstract}

%


\thanks{CLE's Research partially supported by NSF grant DMS12-05851, and  ARO grant W911NF-12-1-0552.}

\maketitle

\tableofcontents

\section{Introduction}
\label{sec:Introduction}
Generalized Kimura diffusion operators are a class of degenerate elliptic operators arising in Population Genetics as the infinitesimal generators of continuous limits of Markov chains \cite{Kimura_1957, Kimura_1964, Shimakura_1981, Ethier_Kurtz, KarlinTaylor2}. A thorough study of the parabolic equations defined by generalized Kimura operators was initiated by C. Epstein and R. Mazzeo in \cite{Epstein_Mazzeo_2010, Epstein_Mazzeo_annmathstudies}, where the authors construct anisotropic H\"older spaces to prove existence, uniqueness and optimal regularity of solutions to the parabolic Kimura equation. In general, Kimura operators act on functions defined on manifolds with corners \cite[\S 2.1]{Epstein_Mazzeo_annmathstudies}. Let $S_{n,m}:=\RR^n_+\times\RR^m$, where $\RR_+:=(0,\infty)$, $n$ and $m$ are nonnegative integers such that $n+m\geq 1$. In adapted local coordinates, $z=(x,y)\in S_{n,m}$, the generalized Kimura operator takes the form
\begin{equation}
\label{eq:Generator}
\begin{aligned}
\widehat Lu&= \sum_{i=1}^n \left(x_iu_{x_ix_i} + \widehat b_i(x,y)u_{x_i}\right) + \sum_{i,j=1}^n x_ix_j \widehat a_{ij}(x,y)u_{x_ix_j} \\
&\quad+\sum_{i=1}^n\sum_{l=1}^m x_i \widehat c_{il}(x,y)u_{x_iy_l} +\sum_{k,l=1}^m \widehat d_{kl}(x,y)u_{y_ky_l}
+ \sum_{l=1}^m \widehat e_l(x,y)u_{y_l},\quad\forall\, u \in C^2(S_{n,m}),
\end{aligned}
\end{equation}
in a neighborhood $(0,0)$. The main feature of the operator $\widehat L$ is that it is not strictly elliptic
as we approach the boundary of the domain $S_{n,m}$, because the smallest
eigenvalue of the second-order coefficient matrix tends to $0$ at a rate proportional to the distance to the boundary of the domain. For this reason, the signs of the coefficient functions $\widehat b_i(z)$ along $\partial S_{n,m}$ play a crucial role in the regularity of solutions;  we always assume that the drift coefficients $\widehat b_i(z)$ are positive functions along $\partial S_{n,m}$. The precise technical conditions imposed on the coefficients of the operator $\widehat L$ in our article are described in Assumption \ref{assump:Coeff_hat_L}.

The initial motivation of our article was to prove that nonnegative solutions to the parabolic equation defined by generalized Kimura operators, 
\begin{equation}
\label{eq:Parabolic_Kimura_homogeneous}
u_t-\widehat Lu=0 \quad\hbox{on } (0,\infty)\times S_{n,m},
\end{equation}
satisfy a scale-invariant Harnack inequality, \eqref{eq:Harnack_inequality_scale_invariant}. We prove this result in Theorem \ref{thm:Harnack_inequality_scale_invariant}. A different proof, based on
Moser's iteration method \cite{Moser_1961, Moser_1964}, is given
in~\cite{Epstein_Mazzeo_cont_est} to establish the Harnack inequality of
nonnegative solutions to the Kimura parabolic equation
\eqref{eq:Parabolic_Kimura_homogeneous}. In this paper our method of the proof
is based on a probabilistic technique due to K.-T. Sturm. In \cite{Sturm_1994},
Sturm establishes the Harnack inequality for operators that can be viewed
as lower-order perturbations of a model operator for which a Harnack inequality is known to hold. Adapting this idea to our framework, we view the generalized Kimura operator $\widehat L$ defined in \eqref{eq:Generator}, as a ``lower-order'' perturbation of an operator $L$ considered in \cite{Epstein_Mazzeo_cont_est}. The operator $L$ takes the form
\begin{equation}
\label{eq:Singular_operator}
\begin{aligned}
&Lu= \sum_{i=1}^n \left(x_i a_{ii} u_{x_ix_i}+b_ia_{ii}u_{x_i}\right) + \sum_{i,j=1}^n x_ix_j\tilde a_{ij} u_{x_ix_j} + \sum_{i=1}^n\sum_{l=1}^m 2x_ic_{il} u_{x_iy_l} +\sum_{k,l=1}^m d_{lk} u_{y_ly_k}\\
&\quad+ \sum_{i=1}^n x_i\left(\partial_{x_i} a_{ii}+\sum_{j=1}^n\left(\tilde a_{ij}+\delta_{ij}\tilde a_{ii}+x_j\partial_{x_j}\tilde a_{ij} + \tilde a_{ij}(b_j-1)\right)+\sum_{l=1}^m \partial_{y_l} c_{il}\right)u_{x_i}\\
&\quad+\sum_{i=1}^n x_i \left[ \sum_{j=1}^n \left(\partial_{x_i} b_j +\sum_{k=1}^n x_k\tilde a_{ik}\partial_{x_k} b_j +\sum_{l=1}^m c_{il} \partial_{y_l} b_j\right) \ln x_j\right] u_{x_i}\\
&\quad+\sum_{l=1}^m \left[\sum_{i=1}^n (x_i\partial_{x_i} c_{il}+b_ic_{il})+\sum_{k=1}^m \partial_{y_k} d_{lk} +\sum_{j=1}^n\left(\sum_{i=1}^n x_ic_{il}\partial_{x_i}b_j+\sum_{k=1}^m d_{lk} \partial_{y_k} b_j\right)\ln x_j\right] u_{y_l},
\end{aligned}
\end{equation}
where $u\in C^2(S_{n,m})$, and $\delta_{ij}$ denotes the Kronecker delta
symbol. The main difference between the \emph{standard} Kimura diffusion
operator $\widehat L$, defined in \eqref{eq:Generator}, and the \emph{singular}
Kimura diffusion operator $L$, defined in \eqref{eq:Singular_operator}, is the
addition of singular drift coefficients of logarithmic-type in the definition
of the operator $L$. The remaining technical conditions satisfied by the coefficients of the operator $L$ are described in Assumption \ref{assump:Coeff}. 

If $\widehat L$ is a standard Kimura diffusion, with
weights $\{b_i(x;y)\}$ that are non-constant along some boundary components,
then the natural representation of the dual operator $\widehat L^t$ acting on
$wd\mu$ includes a first order tangent vector field with logarithmically
divergent coefficients. It is therefore inevitable that such singular terms be considered. The particular choice of the operator $L$ in \cite{Epstein_Mazzeo_cont_est} and in our work, is also motivated by the fact that it can be written in divergence form, that is, for all $u,v\in C^2_c(\bar S_{n,m})$, we have that
$$
-\left(Lu,v\right)_{L^2(S_{n,m};d\mu)} = Q(u,v),
$$ 
where the weighted Sobolev space $L^2(S_{n,m};d\mu)$ consists of measurable functions, $u:S_{n,m}\rightarrow\RR$, that are $L^2$-integrable with respect to the measure $d\mu(z)$, defined by
\begin{align}
\label{eq:Measure}
d\mu(z):=\left(\prod_{i=1}^n x_i^{b_i(z)-1}\right) \, dz,\quad\forall\, z=(x,y)\in S_{n,m},
\end{align} 
and the symmetric bilinear form, $Q(u,v)$, is given by 
\begin{equation}
\label{eq:Bilinear_form}
\begin{aligned}
Q(u,v)&:=\int_{S_{n,m}} \left(\sum_{i=1}^n x_i a_{ii}(z) u_{x_i} v_{x_i}+\sum_{i,j=1}^n x_ix_j\tilde a_{ij}(z) u_{x_i} v_{x_j}\right.\\
&\quad\left.
+\sum_{i=1}^n\sum_{l=1}^m x_i c_{il}(z)\left(u_{x_i}v_{y_l}+u_{y_l}v_{x_i}\right)
+\sum_{l,k=1}^m d_{lk}(z) u_{y_l} v_{y_k} \right) d\mu(z).
\end{aligned}
\end{equation}
Singular Kimura diffusion operators of the form \eqref{eq:Singular_operator} have been studied by C. Epstein and R. Mazzeo in \cite{Epstein_Mazzeo_2010, Epstein_Mazzeo_cont_est}, where the authors establish that nonnegative solutions to the parabolic equation defined by the singular Kimura operator $L$ satisfy the scale-invariant Harnack inequality (\cite[Theorem 4.1]{Epstein_Mazzeo_cont_est}) and are H\"older continuous with respect to a suitable distance function defined in \S \ref{sec:Fundamental_solution} (\cite[Corollary 4.1]{Epstein_Mazzeo_cont_est}), and they prove supremum and infimum estimates of the fundamental solution associated to the operator $L$ (\cite[Corollaries 4.3 and 4.4]{Epstein_Mazzeo_cont_est}). The Harnack inequality and the supremum estimates of the fundamental solution of the operator $L$ motivate our choice to view $L$ as the model operator, and $\widehat L$ as the lower-order perturbation of $L$ in applying Sturm's approach \cite{Sturm_1994} to prove the Harnack inequality for the operator $\widehat L$.

The difficulty in applying Sturm's method to our framework is due to the fact that Harnack's inequality has to be satisfied by solutions of the parabolic equation defined by the operator $L$, 
\begin{equation}
\label{eq:Parabolic_singular_Kimura_homogeneous}
u_t- Lu=0\quad\hbox{on } (0,\infty)\times S_{n,m},
\end{equation}
that satisfy a suitable stochastic representation give by \eqref{eq:Stochastic_representation_cont}, as opposed to weak solutions to the equation \eqref{eq:Parabolic_singular_Kimura_homogeneous}, defined in \S \ref{sec:Inhomogeneous_initial_value_problem}. To explain the difference between weak solutions and solutions satisfying a suitable stochastic representation, we introduce the two concepts. 

Let $\Omega \subseteq S_{n,m}$ be an open set, and let $I:=(t_1,t_2)\subset \RR_+$. We consider the inhomogeneous initial-value problem,
\begin{equation}
\label{eq:Inhomogeneous_initial_value_problem}
\begin{aligned}
u_t-Lu&=g\quad\hbox{on } I\times\Omega,\\
u&=f\quad\hbox{on } \{t_1\}\times\Omega.
\end{aligned}
\end{equation}
Following \cite[\S 1.4 A]{Sturm_1995}, \cite[\S 7.1.1 b.]{Evans}, \cite[Theorem 10.9]{Brezis}, we define the notion of a weak solution as follows.
\begin{defn}[Weak solution]
\label{defn:Weak_solution}
Let $f\in \cH$ and $g \in L^2(I,\cF^*)$. A function $u\in \cF(I\times\Omega)$ is a solution to the inhomogeneous initial-value problem \eqref{eq:Inhomogeneous_initial_value_problem} if
\begin{enumerate}
\item[1.] For all $v\in \cF(I\times\Omega)$, we have that
\begin{equation}
\label{eq:Weak_sol_var_eq}
\int_I Q(u(t),v(t))\, dt +\int_I \left(\frac{du(t)}{dt}, v(t)\right)\, dt=\int_I \left(g(t),v(t)\right)\, dt, 
\end{equation}
where $(\cdot,\cdot)$ denotes the dual pairing of $\cF^*$ and $\cF$.
\item[2.] The initial condition is satisfied in the $\cH$-sense, that is
\begin{equation}
\label{eq:Weak_sol_initial_cond}
\|u(t)-f\|_{\cH}\rightarrow 0\quad\hbox{as } t\downarrow t_1.
\end{equation}
\end{enumerate}
\end{defn}
The technical definitions of the weighted Sobolev spaces $\cH$ and $\cF$, and of the dual space $\cF^*$ are given in \S \ref{sec:Inhomogeneous_initial_value_problem}. We prove in \S \ref{sec:Inhomogeneous_initial_value_problem}, that the unique weak solution, $u \in \cF(I\times\Omega)$, to the homogeneous initial-value problem \eqref{eq:Inhomogeneous_initial_value_problem}  with $f\in \cH$ and $g\equiv 0$ generates a strongly continuous semigroup, $\{T^{\Omega}_t\}_{t\geq 0}$, such that
\begin{equation}
\label{eq:Representation_homogeneous_sol}
u(t+t_1)=T^{\Omega}_t f,\quad\forall\, t \in [0,t_2-t_1].
\end{equation}
From \cite[Theorem 4.1]{Epstein_Mazzeo_cont_est}, we know that Harnack's inequality holds for nonnegative weak solutions to the parabolic equation \eqref{eq:Parabolic_singular_Kimura_homogeneous}.

To introduce the version of Harnack's inequality needed for our proof of Theorem \ref{thm:Harnack_inequality_scale_invariant}, we first need to define the singular Kimura stochastic differential equation associated to the operator $L$:
\begin{equation}
\label{eq:Kimura_SDE_singular}
\begin{aligned}
d X_i(t) &= \left(g_i(Z(t))+ X_i(t)\sum_{j=1}^{n}f_{ij}( Z(t))\ln  X_j(t)\right)\, dt+\sqrt{ X_i(t)}\sum_{j=1}^{n+m} \sigma_{ij}( Z(t))\, d W_j(t),\\
d Y_l(t) &= \left(e_l( Z(t))+\sum_{j=1}^{n}f_{l+n,j}( Z(t))\ln  X_j(t)\right)\, dt+\sum_{j=1}^{n+m} \sigma_{l+n,j}( Z(t))\, d W_j(t),
\end{aligned}
\end{equation}
for all $i=1,\ldots,n$ and $l=1,\ldots,m$. The coefficients of the stochastic differential equation \eqref{eq:Kimura_SDE_singular} are related to the coefficients of the operator $L$ through identities \eqref{eq:Coeff_Kimura_SDE_L}. The existence and uniqueness of weak solutions to the Kimura equation \eqref{eq:Kimura_SDE_singular} that satisfy the Markov property is established in \cite[Theorems 3.4 and 3.8]{Pop_2013a}. Denote $Q:=(t_1,t_2)\times\Omega$, and let
\begin{equation}
\label{eq:parabolic_boundary}
\eth Q:=\left([t_1,t_2]\times\partial_1\Omega\right) \cup\left(\{t_1\}\times\Omega\right),
\end{equation}
where $\partial_1\Omega := \partial\Omega\cap S_{n,m}$. We can now state the Harnack inequality satisfied by functions defined through a stochastic representation:
\begin{thm}[Stochastic representation and the Harnack inequality]
\label{thm:Stochastic_representation}
There is a positive constant, $K_0$, such that the following hold. Let $g \in C(\overline{\eth Q})$ be a nonnegative function, and let $u$ be the function defined by the stochastic representation
\begin{equation}
\label{eq:Stochastic_representation_cont}
u(t,z):=\EE_{\PP^z}\left[g(t-(t-t_1)\wedge\tau_{\Omega}), Z((t-t_1)\wedge\tau_{\Omega}))\right],\quad\forall\, (t,z) \in \bar Q,
\end{equation}
where $\tau_{\Omega}$ denotes the stopping time
$$
\tau_{\Omega} :=\inf\{ t\geq 0:\, Z(t) \notin\underline \Omega\}.
$$
Then the function $u$ satisfies the scale-invariant Harnack inequality, that is, for all $(t^0,z^0)\in\bar Q$ and $r>0$ such that $Q_{2r}(t^0,z^0) \subset Q$, we have that
\begin{equation}
\label{eq:Harnack_cont}
\esssup_{Q_r(t^0-2r^2,z^0)} u \leq K_0 \essinf_{Q_r(t^0,z^0)} u,
\end{equation}
where the parabolic cylinder $Q_r(t^0,z^0)$ is defined in \eqref{eq:Parabolic_cylinder}.
\end{thm}
The first part of our article,  \S \ref{sec:Inhomogeneous_initial_value_problem} - \S \ref{sec:Kimura_SDE}, is dedicated to proving  Theorem \ref{thm:Stochastic_representation}. This is the key ingredient, which allows us in Theorem \ref{thm:Harnack_inequality_scale_invariant}, to apply Sturm's probabilistic argument to prove the scale-invariant Harnack inequality for the standard Kimura operator $\widehat L$. An intermediary result of independent interest is Theorem \ref{thm:Connection_semigroup_Markov_solution}, where we prove that the unique weak solution, in the sense of Definition \ref{defn:Weak_solution}, to the homogeneous initial-value problem \eqref{eq:Homogeneous_initial_value_problem}, defined by the singular Kimura operator $L$, satisfies a Feynman-Kac formula.

\begin{thm}[Connection between the semigroup $\{T^{\Omega}_t\}_{t\geq 0}$ and Markov solutions to \eqref{eq:Kimura_SDE_singular}]
\label{thm:Connection_semigroup_Markov_solution}
Let $z \in \bar S_{n,m}$. Let $\{Z(t):Z(0)=z\}_{t \geq 0}$ be the unique weak solution satisfying the Markov property to the singular Kimura stochastic differential equation \eqref{eq:Kimura_SDE_singular}, with initial condition $Z(0)=z$. Then, given any bounded, Borel measurable function, $f:\underline\Omega\rightarrow\RR$, we have that
\begin{equation}
\label{eq:Connection_semigroup_Markov_process}
\left(T^{\Omega}_t f\right)(z) = \EE_{\PP^z}\left[f(Z(t)) \mathbf{1}_{\{t <\tau_{\Omega}\}}\right],\quad\forall\, t \geq 0,
\end{equation}
where $\PP^z$ is the probability distribution of the process $\{Z(t)\}_{t\geq 0}$.
\end{thm}

Assume that $f\in L^2(\Omega;d\mu)\cap L^{\infty}(\Omega)$ and $g\equiv 0$ in Definition \ref{defn:Weak_solution} of the initial-value problem \eqref{eq:Homogeneous_initial_value_problem}. Then Lemma \ref{lem:Existence_uniqueness_homogeneous} and Theorem \ref{thm:Connection_semigroup_Markov_solution} show that the unique weak solution, $u\in\cF((0,T)\times\Omega)$, to the homogeneous initial-value problem \eqref{eq:Inhomogeneous_initial_value_problem} admits the stochastic representation,
\begin{equation}
\label{eq:Stochastic_representation_whole_space}
u(t,z) = \EE_{\PP^z}\left[f(Z(t)) \mathbf{1}_{\{t <\tau_{\Omega}\}}\right],\quad\forall\, (t,z) \in (0,T)\in\underline\Omega.
\end{equation}
The stochastic representation \eqref{eq:Connection_semigroup_Markov_process} of weak solutions is a considerable generalization of classical results on stochastic representations concerning the assumptions on the degeneracy of the diffusion matrix, the boundedness of the drift coefficients, and on the a priori regularity of the weak solutions. A more detailed comparison between the classical results on stochastic representation of solutions relevant to our work and Theorem \ref{thm:Connection_semigroup_Markov_solution} is given in \S \ref{subsec:Comparison}.

\subsection{Outline of the article}
\label{subsec:Outline}
In \S \ref{sec:Inhomogeneous_initial_value_problem} - \S \ref{sec:Kimura_SDE}, we prove various results concerning the singular Kimura differential operator $L$ defined in \eqref{eq:Singular_operator}. In \S \ref{sec:Kimura_SDE}, we prove Theorem \ref{thm:Connection_semigroup_Markov_solution} which establishes the stochastic representation of weak solutions to the initial-value problem \eqref{eq:Homogeneous_initial_value_problem} defined by the singular Kimura operator $L$, and we prove Theorem \ref{thm:Stochastic_representation} which contains the Harnack inequality satisfied by functions defined by the stochastic representation \eqref{eq:Stochastic_representation_cont}. In \S \ref{sec:Harnack_inequality}, we apply Theorem \ref{thm:Stochastic_representation} to the proof of the Harnack inequality for nonnegative solutions to the parabolic equation defined by the standard Kimura operator $\widehat L$ given by \eqref{eq:Generator}. The main results of \S \ref{sec:Harnack_inequality} are Theorems \ref{thm:Harnack_inequality} and \ref{thm:Harnack_inequality_scale_invariant}.

We begin in \S \ref{sec:Inhomogeneous_initial_value_problem} by introducing the parabolic problem defined by the singular Kimura operator $L$ on subdomains $\Omega$ of $S_{n,m}$, and the notion of weak solution in Definition \ref{defn:Weak_solution}. We prove the existence and uniqueness of weak solutions in suitable weighted Sobolev spaces, and we show that the weak solutions to the homogeneous initial-value problem \eqref{eq:Homogeneous_initial_value_problem} generate a strongly continuous, contraction semigroup on $L^2(\Omega;d\mu)$, $\{T^{\Omega}_t\}_{t\geq 0}$. Sections \S \ref{sec:Fundamental_solution} - \S \ref{sec:Kimura_SDE} contain the gradual construction of a continuous Markov process, $\{Z(t)\}_{t\geq 0}$, associated to the semigroup $\{T^{\Omega}_t\}_{t\geq 0}$ (Theorem \ref{thm:Connection_semigroup_Markov_solution}), satisfying the property that it is the unique Markovian solution to the singular Kimura stochastic differential equation \eqref{eq:Kimura_SDE_singular}.

In \S \ref{sec:Fundamental_solution}, with the aid of the results obtained in \cite{Epstein_Mazzeo_cont_est}, we establish properties of the fundamental solution of the singular Kimura operator $L$ in Lemmas \ref{lem:Measurability_fundam} and \ref{lem:Regularity_fundam}, which we then use in \S \ref{subsec:Markov_process} to construct a continuous Markov process associated to the semigroup $\{T^{\Omega}_t\}_{t\geq 0}$, which is defined initially only at nonnegative dyadic time points (Lemma \ref{lem:Process_dyadic_times}). In \S \ref{subsec:Distribution_estimates}, we make use of the supremum estimates \eqref{eq:Upper_bound_transition_density} of the fundamental solution of the operator $L$ established in \cite[Corollary 4.3]{Epstein_Mazzeo_cont_est}, to derive $L^q$ and H\"older distribution estimates of the fundamental solution in Lemmas \ref{lem:L_q_transition_density} and \ref{lem:Continuity_transition_density}, respectively. The H\"older estimates of the fundamental solution are used in \S \ref{subsec:Markov_process} to extend, by continuity, the Markov process constructed Lemma \ref{lem:Process_dyadic_times} from nonnegative dyadic time points to all nonnegative times, obtaining a continuous Markov process with values in $\bar \Omega$ associated to the semigroup $\{T^{\Omega}_t\}_{t\geq 0}$ (Proposition \ref{prop:Stopped_process}). 

In \S \ref{sec:Martingale_problem}, we formulate in Definition \ref{def:Martingale_problem} the local martingale problem associated to the operator $L$ on subdomains $\Omega$ of $S_{n,m}$. Using the $L^q$-estimates of the fundamental solution and the existence of a continuous Markov process associated to the semigroup $\{T^{\Omega}_t\}_{t\geq 0}$, established in Proposition \ref{prop:Stopped_process}, we prove existence of solutions to the local martingale problem in Proposition \ref{prop:Solution_local_martingale_problem}. The solutions to the local martingale problem associated to the operator $L$ on subdomains $\Omega \subseteq S_{n,m}$ give us in Proposition \ref{prop:Stopped_solutions}, the existence of solutions to the singular Kimura stochastic differential equation \eqref{eq:Kimura_SDE_singular} associated to the semigroup $\{T^{\Omega}_t\}_{t\geq 0}$. The process we construct in Proposition \ref{prop:Stopped_solutions} is stopped at the first time it hits the non-degenerate boundary of the domain, $\partial_1\Omega$, defined in \eqref{eq:Non_degenerate_boundary}. This is not sufficient for our purposes, because we want to establish the connection between the semigroup $\{T^{\Omega}_t\}_{t\geq 0}$ and the solutions to the singular Kimura stochastic differential equation \eqref{eq:Kimura_SDE_singular} that live on $\bar S_{n,m}$, as opposed to being stopped on $\partial_1\Omega$. This result is proved in Theorem \ref{thm:Connection_semigroup_Markov_solution} using a patching argument due to Stroock and Varadhan described in \cite[Theorem 6.1.2]{Stroock_Varadhan}, and the uniqueness of Markovian solutions to the singular Kimura stochastic differential equation \eqref{eq:Kimura_SDE_singular}, established in \cite[Theorem 3.8]{Pop_2013a}.

In \S \ref{sec:Harnack_inequality}, we give in Theorem \ref{thm:Harnack_inequality_scale_invariant} the proof of the scale-invariant Harnack inequality for the standard Kimura diffusion operator $\widehat L$. This requires us to introduce in Definition \ref{defn:Probabilistic_solution} the notion of a probabilistic solution to the parabolic differential equation defined by the operator $\widehat L$. To motivate the definition of the probabilistic solution, we describe in \S \ref{subsec:Connection_L__hat_L} the probabilistic relation between the standard Kimura operator $\widehat L$ and the singular Kimura operator $L$, via Girsanov's Theorem \cite[Theorem 3.5.1]{KaratzasShreve1991}. 
Notice that unlike the definition of the probabilistic solution in \cite[ p. 596]{Sturm_1994}, we assume that our probabilistic solution in Definition \ref{defn:Probabilistic_solution}, is continuous, as opposed to only Borel measurable, since the expression inside the expectation in identity \eqref{eq:Probabilistic_solution} is evaluated at points on the boundary of a cylinder, which is a set of measure zero. In Lemma \ref{lem:Stochastic_representation_L_local}, we justify that the unique solutions in anisotropic H\"older spaces to the homogeneous initial-value problem \eqref{eq:Initial_value_problem_widehat_L} ($g\equiv 0$) defined by the operator $\widehat L$ obtained in \cite{Epstein_Mazzeo_annmathstudies}, are probabilistic solutions in the sense of Definition \ref{defn:Probabilistic_solution}, and so, in particular, they satisfy the scale-invariant Harnack inequality, \eqref{eq:Harnack_inequality_scale_invariant}.

\subsection{Comparison with previous research}
\label{subsec:Comparison}
Moser's classical articles \cite{Moser_1964, Moser_1967, Moser_1971}, which establish Harnack's inequality of nonnegative solutions to parabolic equations defined by strictly elliptic differential operators in divergence form, have been extended to a broad class of degenerate parabolic differential operators in divergence form in a number of works such as 
\cite[Theorem 4.1]{Epstein_Mazzeo_cont_est}, \cite[Lemma 2.3.5]{Fabes_Kenig_Serapioni_1982a}, \cite[Theorem 2.13]{Koch}, \cite[Theorem 2.13]{Lierl_Saloff-Coste_2012}, \cite{Sturm_1996}, \cite[Theorem 1.24]{Feehan_Pop_regularityweaksoln_v3}, among others. Using completely different methods, Harnack's inequality for strictly elliptic parabolic equations with bounded coefficients in non-divergence form has been establish in works such as 
\cite[Theorem 3.1]{Fabes_Stroock_1984}, \cite[Theorem 4.2.1]{Krylov_nonlinear_parabolic_pde} and \cite[Theorem 1]{Sturm_1994}. The framework of our scale-invariant Harnack inequality (Theorem \ref{thm:Harnack_inequality_scale_invariant}) differs from the aforementioned results in that the generalized Kimura diffusion operator $\widehat L$ is not strictly elliptic and has unbounded first-order coefficients.

While the stochastic representation of weak solutions to the homogeneous initial-value problem \eqref{eq:Homogeneous_initial_value_problem} obtained in Theorem \ref{thm:Connection_semigroup_Markov_solution}, is an auxiliary result in our article, it is a considerable generalization of classical results concerning Feynman-Kac formulas, which we now briefly review. In \cite[Theorem 4]{Sturm_1994}, the author proves stochastic representations of solutions to parabolic equations defined by operators that are strictly elliptic and have bounded coefficients. Let $Q:=(0,T)\times\Omega$, where $\Omega\subseteq \RR^N$, $N\in \NN$, and $T>0$. The solutions  considered in \cite[Theorem 4]{Sturm_1994} are assumed to belong to the Sobolev space $H^{1,2}_p(Q)$, which consists of functions $u:Q\rightarrow\RR$ in $L^p(Q)$, with the property that the distributional derivatives $u_t$, $u_{x_i}$ and $u_{x_ix_j}$ also belong to $L^p(Q)$, where $p>1+N$. Notice that Morrey's embedding theorem \cite[Theorem 5.6.4]{Evans}, \cite[Theorem 5.4 C']{Adams_1975} and the condition that $p>1+N$ imply that the functions $u\in W^{1,2}_p(Q)$ are H\"older continuous on $\bar Q$. Comparing this result with our stochastic representation in Theorem \ref{thm:Connection_semigroup_Markov_solution}, we see that the diffusion matrix of the Kimura operator $L$ is not strictly elliptic, the drift coefficients are not bounded because they have logarithmic singularities as we approach the boundary of the domain $S_{n,m}$, and our weak solutions belong to the weighted Sobolev space $\cF(Q)$ defined in \S \ref{sec:Inhomogeneous_initial_value_problem}, and thus, have less regularity than functions in $W^{1,2}_p(Q)$. In particular, Morrey's embedding theorem does not necessarily hold and we know nothing about the regularity of the second-order spatial derivatives of the weak solutions. Thus, we cannot apply It\^o's rule and the approximation argument used in the proof of \cite[Theorem 4]{Sturm_1994} to establish Theorem \ref{thm:Connection_semigroup_Markov_solution}. Instead, we adopt a completely different approach outlined in \S \ref{subsec:Outline} in more detail. Similar stochastic representation formulas to \cite[Theorem 4]{Sturm_1994} are established in \cite[Theorems 2.7.3 and 2.7.4]{Bensoussan_Lions} and \cite[\S II.3]{Portenko}. Stochastic representation formulas for solutions to a degenerate parabolic problem arising in Mathematical Finance are established in \cite[Theorems 1.14, 1.15 and 1.17]{Feehan_Pop_stochrepdirichlet}.

\subsection{Notations and conventions}
\label{subsec:Notations_conventions}
Let $\NN:=\{0,1,2,\ldots\}$. For all $n,m$ and $k$ positive integers, and
$U\subset\RR^n$ an open set, we denote by $C^k(U;\RR^m)$ the set of functions
$u:U\rightarrow\RR^m$ that are $k$ times continuously differentiable on $U$. We
let $C^k(\bar U;\RR^m)$ be the set of functions $u:\bar U\rightarrow\RR^m$ that
are $k$ times continuously differentiable on $\bar U$, such that the norm
$$
\|u\|_{C^k(\bar U;\RR^m)} :=\sum_{\stackrel{\alpha\in\NN^n}{|\alpha|\leq k}}\|D^{\alpha}u\|_{C(\bar U;\RR^m)} <\infty,
$$
where $|\alpha|=\alpha_1+\ldots+\alpha_n$, for all $\alpha=(\alpha_1,\ldots,\alpha_n)\in\NN^n$.

Let $a,b\in\RR$. We denote $a\wedge b:=\min\{a,b\}$ and $a\vee b:=\max\{a,b\}$.

\section{Inhomogeneous initial-value problems and semigroups}
\label{sec:Inhomogeneous_initial_value_problem}
In this section, we begin by stating the conditions satisfied by the coefficients of the singular Kimura operator $L$ given by \eqref{eq:Singular_operator}, in Assumption \ref{assump:Coeff}. We then introduce in Definition \ref{defn:Weak_solution} the notion of weak solution to the parabolic problem defined by the singular Kimura operator, and we recall in Lemmas \ref{lem:Existence_uniqueness_homogeneous} and \ref{lem:Existence_uniqueness_inhomogeneous} the existence and uniqueness results for weak solutions to homogeneous and inhomogeneous initial-value problems defined by $L$, respectively.

We begin by stating the assumptions on the coefficients of the singular Kimura
diffusion operator $L$ defined in \eqref{eq:Singular_operator}. For a set of indices, $I\subseteq\{1,\ldots,n\}$, we denote
\begin{align}
\label{eq:M_I}
M_I:=\left\{z=(x,y) \in S_{n,m}: x_i \in (0,1)\hbox{ for all } i \in I,\hbox{ and } x_j \in [1,\infty)\hbox{ for all } j \in I^c\right\},
\end{align}
where we let $I^c:=\{1,\ldots,n\}\backslash I$. We make the following assumptions about the coefficients of the bilinear form $Q(u,v)$ given by \eqref{eq:Bilinear_form}.
\begin{assump}[Coefficients of the operator $L$]
\label{assump:Coeff}
There are positive constants, $\bar b$, $\delta$ and $K$, such that
\begin{enumerate}
\item[1.] The coefficient matrix functions $\diag(a_{ii}(z))$, $(\tilde a(z))$ and $(c(z))$ are chosen such that, for all $z \in \bar M_{\{1,\ldots,n\}}$, $\xi\in\RR^n$ and $\eta\in\RR^m$, we have that 
\begin{equation}
\label{eq:Ellipticity_all_second_order_coeff}
\begin{aligned}
&\delta\left(|\xi|^2+|\eta|^2\right)\\
&\leq \sum_{i=1}^n a_{ii}(z)\xi_i^2+\sum_{i,j=1}^n \sqrt{x_ix_j}\tilde a_{ij}(z)\xi_i\xi_j + 8\sum_{i=1}^n\sum_{l=1}^m \sqrt{x_i}c_{il}(z)\xi_i\eta_l +4\sum_{l,k=1}^m d_{lk}(z)\eta_l\eta_k\\
&\leq K\left(|\xi|^2+|\eta|^2\right).
\end{aligned}
\end{equation}
Compare condition \eqref{eq:Ellipticity_all_second_order_coeff} with \cite[Condition (33)]{Epstein_Mazzeo_cont_est} and \cite[Definitions (3) and (6)]{Epstein_Mazzeo_cont_est} of the bilinear form.
\item[2.] Let $I\subsetneqq\{1,\ldots,n\}$. For all $z \in \bar M_I$, we assume that
\begin{equation}
\label{eq:Outside_neighborhood_O}
\begin{aligned}
x_j a_{jj}(z)&=1,\quad \forall\, j \in I^c,\\
\tilde a_{ij}(z) &=0,\quad\forall\, i,j=1,\ldots,n,\\
c_{il}(z) &= 0,\quad\forall\, i=1,\ldots,n,\quad\forall\, l=1,\ldots,m,\\
d_{lk}(z) &= \delta_{lk}, \quad\forall\, l,k=1,\ldots,m,
\end{aligned}
\end{equation}
where we recall that $\delta_{lk}$ denotes the Kronecker delta symbol.
\item[3.] Let $I\subsetneqq\{1,\ldots,n\}$. For all $z \in \bar M_{I}$, we assume that
\begin{equation}
\label{eq:Outside_neighborhood_O_a_ii}
\begin{aligned}
\delta\leq a_{ii}(z) \leq K,\quad \forall\, i \in I.
\end{aligned}
\end{equation}
\item[4.] The coefficients functions $b_i(z)$ satisfy, for all $i=1,\ldots,n$, 
\begin{align}
\label{eq:b_coeff_bound_lower}
 b_{i}(z) \geq \bar b>0,&\quad \forall\, z\in \partial S_{n,m}\cap\{x_i=0\},\\
\label{eq:b_coeff_bound_upper}
|b_{i}(z)| \leq K,&\quad\forall\, z\in \bar S_{n,m},\\
\label{eq:b_coeff_away_O}
b_i(z)=1,&\quad\forall\, z\in M^c_{\{1,\ldots,n\}}.
\end{align}
\item[5.] The coefficient functions $a_{ii}(z)$, $\tilde a_{ij}(z)$, $b_i(z)$, $c_{ik}(z)$, $d_{kl}(z)$ are smooth and bounded functions on $\bar S_{n,m}$, for all $i,j=1,\ldots,n$, and all $k,l=1,\ldots,m$.
\end{enumerate}
\end{assump}

\begin{rmk}[Remarks on Assumption \ref{assump:Coeff}]
Conditions \eqref{eq:Outside_neighborhood_O} and \eqref{eq:b_coeff_away_O} imply that for all sets of indices $I\subsetneqq\{1,\ldots,n\}$, and functions $u,v\in C^2_c(M_I)$, the quadratic form $Q(u,v)$ given by \eqref{eq:Bilinear_form}, takes the particular form:
$$
Q(u,v) =\int_{S_{n,m}} \left(\sum_{i\in I} x_ia_{ii}(z) u_{x_i}v_{x_i}+ \sum_{j\in I^c} u_{x_j}v_{x_j} +\sum_{l=1}^m u_{y_l}v_{y_l}\right)\, dz.
$$
On the other hand, when $I=\{1,\ldots,n\}$, the bilinear form $Q(u,v)$ takes the more general form \eqref{eq:Bilinear_form}, for all functions $u,v\in C^2_c(M_I)$. Thus, conditions \eqref{eq:Ellipticity_all_second_order_coeff}, \eqref{eq:b_coeff_bound_lower} and \eqref{eq:b_coeff_bound_upper} imply that the coefficients of the bilinear form satisfy natural assumptions on the open set $M_{\{1,\ldots,n\}}$.  This suffices for our purposes, because our goal is to establish analytical properties of solutions in a neighborhood of $(0,0)\in\bar S_{n,m}$, since the analytical properties of solutions away from the boundary of $S_{n,m}$ are standard results in the literature. 
\end{rmk}

Let now $\Omega\subseteq S_{n,m}$ be an open set, and denote 
\begin{align}
\label{eq:Degenerate_boundary}
\partial_0\Omega &:=\hbox{int} \left(\partial\Omega\cap\partial S_{n,m}\right),\\
\label{eq:Non_degenerate_boundary}	
\partial_1\Omega &:= \partial\Omega\cap S_{n,m},\\
\label{eq:Omega_and_degenerate_boundary}
\underline{\Omega}&:=\Omega\cup\partial_0\Omega.
\end{align}
Let $I=(t_1,t_2)$ be an open, bounded interval in $\RR$. To formulate the notion of weak solution and prove existence and uniqueness of solutions to the inhomogeneous initial-value problem \eqref{eq:Inhomogeneous_initial_value_problem}, we first need to introduce suitable function spaces. We follow \cite[\S 1.3 A]{Sturm_1995}, and denote for simplicity $\cH:=L^2(\Omega; d\mu)$. We let $H^1_0(\underline{\Omega}; d\mu)$ be the closure of smooth functions with compact support in $\underline{\Omega}$, $C^1_c(\underline{\Omega})$, with respect to the norm,
$$
\|u\|_{H^1(\Omega;d\mu)}:=\left(Q(u,u)+\|u\|^2_{\cH}\right)^{1/2}.
$$
For consistency with \cite[\S 1.3 A]{Sturm_1995}, we let $\cF:= H^1_0(\underline{\Omega};d\mu)$. Let $C(\bar I, \cH)$ be the space of continuous functions, $u:\bar I\rightarrow\cH$, endowed with the norm
$$
\|u\|_{C(\bar I, \cH)}:=\sup_{t\in\bar I} \|u(t)\|_{\cH} <\infty.
$$
The space of functions $C(\bar I, \cF)$ is defined similarly to $C(\bar I, \cH)$ by simply replacing the space $\cH$ with $\cF$ in the preceding definition. We let $L^2(I,\cF)$ denote the space of measurable functions, $u$, on $I\times\Omega$ endowed with the norm,
$$
\|u\|_{L^2(I, \cF)}:=\left(\int_I \|u(t)\|^2_{\cF} \, dt\right)^{1/2}<\infty.
$$
Let $\cF^*$ denote the dual space of $\cF$. We let $H^1(I,\cF^*)$ be the space of functions, $u$, such that $u\in L^2(I,\cF^*)$ and the distributional time derivative, $\frac{du}{dt}$, also belongs to $L^2(I,\cF^*)$. We endow the space $H^1(I,\cF^*)$ with the norm
$$
\left(\int_I \left(\|u(t)\|^2_{\cF^*} + \left\|\frac{du(t)}{dt}\right\|^2_{\cF^*}\right) \, dt\right)^{1/2}<\infty.
$$
Finally, we let
$$
\cF(I\times\Omega):=L^2(I,\cF)\cap H^1(I,\cF^*).
$$
We recall from \cite[p. 285]{Sturm_1995} that we have 
\begin{equation}
\label{eq:inclusion_cF_C}
\cF(I\times\Omega) \subset C(\bar I,\cH).
\end{equation}
From \cite[Proposition 1.2]{Sturm_1995}, we have the following existence and uniqueness of weak solutions to the \emph{homogeneous} initial-value problem

\begin{lem}[Existence and uniqueness of solutions to the homogeneous problem]
\label{lem:Existence_uniqueness_homogeneous}
\cite[Proposition 1.2]{Sturm_1995},
\cite[Theorem 3.4.1 and Remark 3.4.3]{Lions_Magenes1},
\cite[Theorem 10.9]{Brezis}
Let $f\in \cH$ and $I=(t_1,t_2).$ There is a unique weak solution, $u\in\cF(I\times\Omega)$, to the homogeneous initial-value problem
\begin{equation}
\label{eq:Homogeneous_initial_value_problem}
\begin{aligned}
u_t-Lu&=0\quad\hbox{on } I\times\Omega,\\
u&=f\quad\hbox{on } \{t_1\}\times\Omega.
\end{aligned}
\end{equation}
\end{lem}
It is shown in \cite[Property (1.19)]{Sturm_1995} that the unique weak solution, $u$, to the homogeneous initial-value problem with initial condition $f\in\cH$, has the property that
\begin{align}
\label{eq:Monotonicity_solutions}	
[0,\infty)\ni t\mapsto \|u(t)\|_{\cH} \quad\hbox{ is non-increasing}.
\end{align}
The uniqueness statement in Lemma \ref{lem:Existence_uniqueness_homogeneous}, the inclusion relation \eqref{eq:inclusion_cF_C}, and  property \eqref{eq:Monotonicity_solutions}	 show that there is a strongly continuous, contraction semigroup on $\cH$, $\{T^{\Omega}_t\}_{t\geq 0}$, such that the unique weak solution, $u$, to the homogeneous problem with initial condition $f\in\cH$, can be represented in the form \eqref{eq:Representation_homogeneous_sol}. Using Lemma \ref{lem:Existence_uniqueness_homogeneous} and Duhamel's principle, we obtain the existence and uniqueness of a weak solution to the \emph{inhomogeneous} initial-value problem \eqref{eq:Inhomogeneous_initial_value_problem}.

\begin{lem}[Existence and uniqueness of solutions to the inhomogeneous problem]
\label{lem:Existence_uniqueness_inhomogeneous}
\cite[Theorem 3.4.1 and Remark 3.4.3]{Lions_Magenes1},
\cite[Theorem 10.9]{Brezis}
Let $f\in \cH$ and $g\in L^2(I,\cH)$. Then there is a unique weak solution, $u\in\cF(I\times\Omega)$, to the inhomogeneous initial-value problem \eqref{eq:Inhomogeneous_initial_value_problem}, and we have that
\begin{equation}
\label{eq:Representation_inhomogeneous_sol}
u(t+t_1)=T^{\Omega}_t f +\int_{0}^t T^{\Omega}_{t-s}g(s+t_1,\cdot)\, ds,\quad\forall\, t \in [0,t_2-t_1].
\end{equation}
\end{lem}

\begin{proof}
To prove the conclusion of the lemma, we start from the representation \eqref{eq:Representation_inhomogeneous_sol} and show that the function on the right-hand side belongs to $\cF(I\times\Omega)$ and solves the inhomogeneous initial-value problem \eqref{eq:Inhomogeneous_initial_value_problem}, by using Lemma \ref{lem:Existence_uniqueness_homogeneous} and  energy estimates \cite[Theorem 2, \S 7.1.2]{Evans}. 
\end{proof}

\section{Properties of the fundamental solution}
\label{sec:Fundamental_solution}
We begin by summarizing the properties of the semigroup $\{T^{\Omega}_t\}_{t \geq 0}$ constructed in \S \ref{sec:Inhomogeneous_initial_value_problem}:
\begin{enumerate}
\item[1.] It is strongly continuous on $L^2(\Omega;d\mu)$: for all $f \in L^2(\Omega;d\mu)$, the map $[0,\infty) \ni t \mapsto T^{\Omega}_t f \in L^2(\Omega;d\mu)$ is continuous.
\item[2.] It is a contraction semigroup:
\begin{equation}
\label{eq:Contraction_semigroup}
\|T^{\Omega}_t f\|_{L^2(\Omega;d\mu)} \leq \|f\|_{L^2(\Omega;d\mu)},\quad\forall f\in L^2(\Omega;d\mu).
\end{equation}
\item[3.] It has the semigroup property: for all $s,t\geq 0$, we have that $T^{\Omega}_{t+s} = T^{\Omega}_t T^{\Omega}_s$ .
\item[4.] It has the symmetry property \cite[Lemma 1.5]{Sturm_1995}: for all $f,g \in L^2(\Omega;d\mu)$, we have that
\begin{equation}
\label{eq:Symmetry_semigroup}
\int_{\Omega} (T^{\Omega}_t f) g \, d\mu =  \int_{\Omega} f (T^{\Omega}_t g) \, d\mu.
\end{equation}
\end{enumerate} 

In this section we prove the existence of the fundamental solution, $p^{\Omega}(t,z^0,z)$, associated to the semigroup $\{T^{\Omega}_t\}_{t\geq 0}$, and we prove some of its regularity properties which are used extensively in the sequel. In \S \ref{subsec:Distribution_estimates}, we prove distribution estimates for the fundamental solution. The $L^q$-distribution estimates in Lemma \ref{lem:L_q_transition_density} are used in Proposition \ref{prop:Solution_local_martingale_problem} to establish existence of solutions to the stopped martingale problem associated to the operator $L$ (Definition \ref{def:Martingale_problem}), while the H\"older estimates of the fundamental solution in Lemma \ref{lem:Continuity_transition_density} are used in Proposition \ref{prop:Stopped_process} to construct a continuous Markov process associated to the semigroup $\{T^{\Omega}_t\}_{t\geq 0}$.

We start with the following 
\begin{lem}[Measurability of the fundamental solution]
\label{lem:Measurability_fundam}
There is a measurable function, $p^{\Omega}:(0,\infty)\times \underline{\Omega}\times \underline\Omega\rightarrow [0,\infty)$, such that 
\begin{equation}
\label{eq:L_2_integrability}
p^{\Omega}(t,z,\cdot) \in L^2(\Omega; d\mu),\quad\forall\, (t,z) \in (0,\infty)\times\underline\Omega,
\end{equation}
and we have that
\begin{align}
\label{eq:Representation_T_t}
\left(T^{\Omega}_t f\right)(z) = \int_{\Omega} p^{\Omega}(t,z,w) f(w)\, d\mu(w),\quad\forall\, (t,z)\in (0,\infty)\times\underline\Omega,\quad\forall\, f \in L^2(\Omega;d\mu).
\end{align}
\end{lem}

Because of the particular degeneracy of the diffusion matrix of the Kimura operator $L$, the H\"older continuity of solutions is established in \cite{Epstein_Mazzeo_annmathstudies, Epstein_Mazzeo_cont_est} with respect to a new distance function, $\rho$, which we now introduce. For points $z^0, z\in \bar S_{n,m}$, we denote by $\rho(z^0,z)$ the \emph{intrinsic metric} induced by the bilinear form, $Q(u,v)$; see \cite[Identity (40)]{Epstein_Mazzeo_cont_est}. Given Assumption \ref{assump:Coeff}, there is a positive constant, $c$, such that for all sets of indices $I,J\subseteq\{1,\ldots,n\}$ and all $z^0\in \bar M_I$ and $z\in \bar M_J$, we have 
\begin{equation}
\label{eq:Equivalent_intrinsic_metric}
\begin{aligned}
&c\left(\max_{i\in I\cap J} \left|\sqrt{x^0_i}-\sqrt{x_i}\right| + \max_{k\in (I\cap J)^c} |x^0_k-x_k| +\max_{l\in\{1,\ldots,m\}}|y^0_l-y_l|\right)\\
&\leq \rho(z^0,z)\\
&\leq c^{-1}\left(\max_{i\in I\cap J} \left|\sqrt{x^0_i}-\sqrt{x_i}\right| + \max_{k\in (I\cap J)^c} |x^0_k-x_k| +\max_{l\in\{1,\ldots,m\}}|y^0_l-y_l|\right).
\end{aligned}
\end{equation}
Compare the preceding inequalities with \cite[Property (42)]{Epstein_Mazzeo_cont_est}. For all $i=1,\ldots,n$, and $x^0_i, x_i\in \bar \RR_+$, we denote for brevity
by $\rho(x^0_i,x_i)$ the distance between the points $P^0$ and $P$ that have
have the $i$-th coordinates equal to $x^0_i$ and $x_i$, respectively, and all
the other coordinates are equal. We define analogously $\rho(y^0_l,y_l)$, for all $y^0_l,y_l\in\RR$ and all $l=1,\ldots,m$. For $z^0\in \underline \Omega$ and $r>0$, we let 
\begin{equation}
\label{eq:Ball_rho}
B_r(z^0):= \{z\in\bar S_{n,m}:\ \rho(z^0,z)<r\},
\end{equation}
denote the ball with center $z^0$ and radius $r$, with respect to the distance function $\rho$. When $z^0=(0,0)$, we write for brevity $B_r$ instead of $B_r(0,0)$.

\begin{proof}[Proof of Lemma \ref{lem:Measurability_fundam}]
By Lemma \ref{lem:Existence_uniqueness_inhomogeneous}, for every function $f\in L^2(S_{n,m};d\mu)$, the homogeneous initial-value problem \eqref{eq:Inhomogeneous_initial_value_problem} has a unique solution, $u \in \cF([0,\infty) \times\Omega)$, which can be represented by $u(t) = T^{\Omega}_t f$, for all $t \geq 0$. By \cite[Corollary 4.1]{Epstein_Mazzeo_cont_est}, the solution $u$ is H\"older continuous on $(0,\infty)\times\underline\Omega$, and so, the function $(T^{\Omega}_t f)(z)$ is well-defined at all points $(t,z)\in (0,\infty)\times\underline\Omega$. By \cite[Theorem 2.1]{Sturm_1995}, for all $0<t_1<t_2$ and all compact sets, $K\subset\underline\Omega$, there is a positive constant, $C=C(t_1,t_2,K)$, such that for all $t\in [t_1,t_2]$ and all $z\in K$, we have that
\begin{align}
\left|\left(T^{\Omega}_t f\right)(z)\right|	&\leq C\left(\int_{t_1/2}^{t_2} \left\|T^{\Omega}_s f\right\|_{L^2(\Omega;d\mu)}\, ds\right)^{1/2}\notag\\
&\label{eq:Upper_bound_T_t_f_z}
 \leq C \left(t_2-\frac{t_1}{2}\right)^{1/2} \|f\|_{L^2(\Omega;d\mu)}\quad\hbox{(using the contraction property of $\{T^{\Omega}_t\}_{t \geq 0}$).}
\end{align}
It follows that the map $L^2(\Omega;d\mu) \ni f \mapsto (T^{\Omega}_t f)(z)$ is continuous, and so, there is a Borel measurable function, $p^{\Omega}(t,z,\cdot)\in L^2(\Omega;d\mu)$, such that identity \eqref{eq:Representation_T_t} holds. Moreover, using the fact that
$$
\|p^{\Omega}(t,z,\cdot)\|_{L^2(\Omega;d\mu)} = \sup_{\|f\|_{L^2(\Omega;d\mu)}=1 } \left|(T^{\Omega}_tf)(z)\right|,
$$
inequality \eqref{eq:Upper_bound_T_t_f_z} gives us that
\begin{align}
\label{eq:L_2_norm_uniform_bound}
\|p^{\Omega}(t,z,\cdot)\|_{L^2(\Omega;d\mu)} \leq C ,\quad\forall\, (t,z)\in [t_1,t_2]	\times K.
\end{align}

Let $(t^0,z^0)\in (0,\infty)\times\underline\Omega$ and $r>0$ be such that $t^0-4r^2>0$. We denote 
$$
Q_{r}(t^0,z^0):=(t^0-r^2,t^0)\times \left(B_{r}(z^0) \cap\underline\Omega\right).
$$
Let $(t',z'), (t''z'')$ be points in $Q_r(t^0,z^0)$. By \cite[Corollary 4.1]{Epstein_Mazzeo_cont_est}, there are positive constants,  $\alpha\in (0,1)$ and $C=C(t^0,z^0,r)$, such that
\begin{align*}
\left|\left(T^{\Omega}_{t'}f\right)(z') - \left(T^{\Omega}_{t''}f\right)(z'')\right|	\leq C \left\|T^{\Omega}_tf\right\|_{L^{\infty}(Q_{2r}(t^0,z^0))} 
\left(\rho(z',z'')+\sqrt{|t'-t''|}\right)^{\alpha}.
\end{align*}
From inequality \eqref{eq:Upper_bound_T_t_f_z}, it follows that
\begin{equation}
\label{eq:Upper_bound_difference_T_t_f_z}
\begin{aligned}
\left|\left(T^{\Omega}_{t'}f\right)(z') - \left(T^{\Omega}_{t''}f\right)(z'')\right|	\leq C \|f\|_{L^2(\Omega;d\mu)} \left(\rho(z',z'')+\sqrt{|t'-t''|}\right)^{\alpha}.
\end{aligned}
\end{equation}
Using the fact that
$$
\left\|p^{\Omega}(t',z',\cdot)-p^{\Omega}(t'',z'',\cdot)\right\|_{L^2(\Omega;d\mu)} = \sup_{\|f\|_{L^2(\Omega;d\mu)}=1 } \left|\left(T^{\Omega}_{t'}f\right)(z') - \left(T^{\Omega}_{t''}f\right)(z'')\right|,
$$
inequality \eqref{eq:Upper_bound_difference_T_t_f_z} gives us that, for all $(t',z'),(t'',z'') \in Q_r(t^0,z^0)$, we have
\begin{align*}
\|p^{\Omega}(t',z',\cdot)-p^{\Omega}(t'',z'',\cdot)\|_{L^2(\Omega;d\mu)} \leq C \left(\rho(z',z'')+\sqrt{|t'-t''|}\right)^{\alpha}.
\end{align*}
Because the point $(t^0,z^0)\in(0,\infty)\times\underline\Omega$ was arbitrarily chosen, the preceding inequality implies that we can find a Borel measurable modification of the function $p^{\Omega}:(0,\infty)\times \underline\Omega\times \underline\Omega\rightarrow [0,\infty)$ that satisfies property \eqref{eq:L_2_integrability} and identity \eqref{eq:Representation_T_t}.
\end{proof}

\begin{lem}[Regularity of the fundamental solution]
\label{lem:Regularity_fundam}
The fundamental solution, $p^{\Omega}$, is continuous on $(0,\infty)\times \underline{\Omega}\times \underline\Omega$, and is smooth on $(0,\infty)\times \Omega\times \Omega$.
\end{lem}

\begin{proof}
By the symmetry property \eqref{eq:Symmetry_semigroup}, we have that
\begin{equation}
\label{eq:Symmetry_fundam}
p^{\Omega}(t,z^0,z) = p^{\Omega}(t,z,z^0),\quad\forall\, t\in (0,\infty),\quad\hbox{ for almost all } z^0, z\in\Omega.
\end{equation}
Let $t_0>0$ and $z^0\in\underline\Omega$. Because $p(t_0,z^0,\cdot)\in L^2(\Omega;d\mu)$, there is a unique weak solution, $u\in \cF((0,T)\times\Omega)$, for all $T>0$, to the homogeneous equation \eqref{eq:Inhomogeneous_initial_value_problem} with initial condition $f=p(t_0,z^0,\cdot)$. Using \cite[Corollary 4.1]{Epstein_Mazzeo_cont_est} together with \cite[Theorem 2.1]{Sturm_1995}, it follows that the solution $u$ is continuous on $(0,T)\times\underline\Omega$. Since $u(t)=T^{\Omega}_t p^{\Omega}(t_0,z^0,\cdot)$, the semigroup property implies that $u(t)=p^{\Omega}(t+t_0,z^0,\cdot)$, and so, the function $p^{\Omega}(\cdot,z^0,\cdot)$ is continuous on $(0,\infty)\times\underline\Omega$, where we used the fact that the positive constants $t_0$ and $T$ were arbitrarily chosen. Using the symmetry of the fundamental solution, \eqref{eq:Symmetry_fundam}, it follows that the function $p^{\Omega}$ is continuous on $(0,\infty)\times\underline\Omega\times\underline\Omega$. Interior regularity of solutions to parabolic equations with smooth coefficients implies that the function $p^{\Omega}$ is smooth on $(0,\infty)\times \Omega\times \Omega$.
\end{proof}

\begin{lem}[Fundamental solutions corresponding to different domains]
\label{lem:Fundam_sol_different_domains}
Let $\Omega_1\subset\Omega_2\subseteq S_{n,m}$ be domains. Then we have that
\begin{equation}
\label{eq:Fundam_sol_different_domains}
p^{\Omega_1}(t,z^0,z) \leq p^{\Omega_2}(t,z^0,z),\quad\forall\, (t,z^0,z)\in(0,\infty)\times\underline\Omega_1\times\underline\Omega_1.
\end{equation}
\end{lem}

\begin{proof}
Using the continuity of the fundamental solution $p^{\Omega_i}$ on $(0,\infty)\times\underline\Omega_i\times\underline\Omega_i$, for $i=1,2$, established in Lemma \ref{lem:Regularity_fundam}, it is sufficient to prove that, for all nonnegative functions $\varphi\in C^{\infty}_c(\underline\Omega_1)$, we have 
\begin{equation}
\label{eq:Fundam_sol_different_domains_aux}
\left(T^{\Omega_1}_t\varphi\right)(z^0) \leq \left(T^{\Omega_2}_t\varphi\right)(z^0),\quad\forall\, (t,z^0)\in (0,\infty)\times\underline\Omega_1.
\end{equation}
For $i=1,2$, we let $\cH_i:=L^2(\underline\Omega_i;d\mu)$, $\cF_i:=H^1_0(\underline\Omega_i;d\mu)$ and $u_i(t) := T^{\Omega_i}_t\varphi$, for all $t\in (0,\infty)$, and we set $u:= u_2-u_1$, and $v:= u^-$. For all $T>0$, we prove that $v\in \cF((0,T)\times\Omega_1)$ in the following
\begin{claim}
\label{claim:Test_function}
The test function $v:=u^-$ belongs to $\cF((0,T)\times\Omega_1)$.
\end{claim}
Before we give the proof of the claim, we show how it can be used to obtain inequality \eqref{eq:Fundam_sol_different_domains}. Applying identity \eqref{eq:Weak_sol_var_eq} to the function $u(t)$ and the test function $v(t)$, and using the fact that $u(0,\cdot)\equiv 0$, we obtain that
\begin{align*}
0= \|v(t)\|^2_{L^2(\Omega_1;d\mu)} +\int_0^t Q(v(s),v(s))\, ds,\quad\forall\, t\in (0,T).
\end{align*}
The symmetry of the bilinear form $Q(u,v)$ and the uniform ellipticity condition \eqref{eq:Ellipticity_all_second_order_coeff}, gives us that $Q(v(s),v(s)) \geq 0$, for all $s\in (0,T)$, and so, it follows that
$$
\|v(t)\|_{L^2(\Omega_1;d\mu)}=0,\quad\forall\, t\in (0,T),
$$
which yields the fact that inequality \eqref{eq:Fundam_sol_different_domains_aux} holds, for all $t\in (0,T)$. Since the positive constant $T$ and the nonnegative test function $\varphi\in C^{\infty}_c(\underline\Omega_1)$ were arbitrarily chosen, we immediately obtain the conclusion of the lemma.

It remains to give the

\begin{proof}[Proof of Claim \ref{claim:Test_function}]
Our goal is to prove that $\frac{dv}{dt} \in L^2((0,T), \cF_1^*)$ and that $v \in L^2((0,T), \cF_1)$. From the fact that $\frac{dv}{dt} = \frac{du}{dt} \mathbf{1}_{\{u<0\}}$, it follows that $\frac{dv}{dt} \in L^2((0,T), \cF_1^*)$. Because $u \in L^2((0,T),\cF_1)$, it follows from \cite[Lemma 7.6]{GilbargTrudinger} that $Dv = Du \mathbf{1}_{\{u<0\}}$ a.e. on $(0,T)\times \Omega'$, for all $\Omega' \subset\Omega$, where $D$ denotes the spatial derivatives. Thus, we obtain that $v$ has finite $L^2((0,T), H^1(\Omega_1;d\mu))$-norm, and now we construct a sequence of functions, $\{v^k\}_{k\in\NN} \subset L^2((0,T), \cF_1)$, such that
\begin{equation}
\label{eq:Convergence_v_k}
\|v^k - v\|_{L^2((0,T), \cF_1)} \rightarrow 0,\quad\hbox{ as } k \rightarrow\infty,
\end{equation}
which will imply that $v \in L^2((0,T), \cF_1)$. From \cite[Theorem 7.7]{Brezis}, it follows that $u \in C([0,T], \cF_1)$, and so, we can find a sequence of functions, $\{u^k\}_{k\in\NN} \subset L^2((0,T), \cF_1)$, such that
\begin{equation}
\label{eq:Convergence_u_k}
\|u^k - u\|_{L^2((0,T), \cF_1)} \rightarrow 0,\quad\hbox{ as } k \rightarrow\infty,
\end{equation}
and $u_k(t) \in C^1_c(\underline\Omega_1)$, for all $t\in (0,T)$ and all $k\in\NN$. By letting $v^k:=(u^k)^-$, we have that $Dv^k = Du^k\mathbf{1}_{\{u^k<0\}}$ by \cite[Lemma 7.6]{GilbargTrudinger}, and we also have that
\begin{align*}
\|v^k-v\|_{L^2((0,T), \cH_1)} & \leq \|u^k-u\|_{L^2((0,T), \cH_1)},\\
\|v^k-v\|_{L^2((0,T), \cH_1)} &= \|Du^k\mathbf{1}_{\{u^k<0\}}-Du\mathbf{1}_{\{u<0\}}\|_{L^2((0,T), \cH_1)}\\
&\leq  \|Du^k-Du\|_{L^2((0,T), \cH_1)}+\|Du^k\mathbf{1}_{\{u^k<0, u \geq 0\}}\|_{L^2((0,T), \cH_1)}\\
&\quad+\|Du\mathbf{1}_{\{u^k\geq 0, u<0\}}\|_{L^2((0,T), \cH_1)}.
\end{align*} 
Property \eqref{eq:Convergence_u_k} yields that both quantities above converge to $0$, as $k$ tends to $\infty$, and so \eqref{eq:Convergence_v_k} holds. 

We now conclude the proof of the fact that the function $v^k$ belongs to $L^2((0,T), \cF_1)$, for all $k\in\NN$. Following \cite[Lemma 7.6]{GilbargTrudinger}, for all $\eps>0$, we let
$$
w^k_{\eps} := \left(\left(v^k\right)^2+\eps^2\right)^{1/2}-\eps.
$$
Direct calculations give us 
$$
Dw^k_{\eps} = \frac{v^kDv^k}{\left(\left(v^k\right)^2+\eps^2\right)^{1/2}}.
$$
Recall that $u^k(t)\in C^1_c(\underline\Omega_1)$ and that $v^k(t) := (u^k(t))^-$, for all $t\in (0,T)$, which implies that $w^k_{\eps}(t) \in C^1_c(\underline\Omega_1)$, for all $t\in (0,T)$. Thus, the sequence of functions $\{w^k_{\eps}\}_{\eps>0} \subset L^2((0,T),\cF_1)$, and using the expressions of $w^k_{\eps}$ and $Dw^k_{\eps}$, it is clear that
$$
\|w^k_{\eps} - v^k\|_{L^2((0,T),\cF_1)} \rightarrow 0\quad\hbox{ as } \eps\downarrow 0.
$$
This implies that the functions $v^k$ belong to $L^2((0,T),\cF_1)$, for all $k\in\NN$. 

This concludes the proof of the claim and of the lemma.
\end{proof}

\end{proof}

When $\Omega=S_{n,m}$ we denote, for simplicity, $p^{S_{n,m}}=p$ and $T^{S_{n,m}}=T$. We next show that the functions $p(t,z,w)\, d\mu(w)$ can be viewed as transition probability densities.

\begin{lem}[Transition probability densities]
\label{lem:Transition_probabilities}
For all $(t,z)\in (0,\infty)\times\bar S_{n,m}$, we have that $p(t,z,\cdot)\in
L^1(S_{n,m};d\mu)$. Moreover, $p$ is a non-negative function, which satisfies
\begin{equation}
\label{eq:Transition_probabilities}
\int_{S_{n,m}} p(t,z,w)\, d\mu(w) =1,\quad\forall\, (t,z)\in (0,\infty)\times\bar S_{n,m}.
\end{equation}
\end{lem}

For $z^0\in \underline \Omega$ and $r>0$, we let 
\begin{equation}
\label{eq:Ball_Euclidean}
B^e_r(z^0):= \{z\in\bar S_{n,m}:\ |z^0-z|<r\},
\end{equation}
denote the Euclidean open ball with center $z^0$ and radius $r$, relative to the set $\bar S_{n,m}$. When $z^0=(0,0)\in\bar S_{n,m}$, we write for brevity $B^e_r$ instead of $B^e_r(0,0)$. 

\begin{proof}[Proof of Lemma \ref{lem:Transition_probabilities}]
The fact that $p(t,z,\cdot)$ is a non-negative function follows from \cite[Lemma 1.4]{Sturm_1995}. From Lemma \ref{lem:L_q_transition_density}, we obtain that $p(t,z,\cdot)\in L^1(S_{n,m};d\mu)$. It remains to prove the integrability property \eqref{eq:Transition_probabilities} which is equivalent to proving that
\begin{equation}
\label{eq:Transition_probabilities_prim}
\left(T_t 1\right)(z) =1,\quad\forall\, (t,z)\in (0,\infty)\times\bar S_{n,m}.
\end{equation}
Let $\{\varphi_k\}_{k\geq 1} \subset C^{\infty}_c(\bar S_{n,m})$ be a sequence of smooth functions such that $0\leq \varphi_k\leq 1$, and
$$
\varphi_k(z)=1\hbox{ for } z\in B^e_k,\quad\hbox{and } \varphi_k(z)=0\hbox{ for } z\in \bar S_{n,m}\backslash B^e_{2k}.
$$
From \eqref{eq:Representation_inhomogeneous_sol}, we have that the following identity holds in the $L^2(S_{n,m}; d\mu)$-sense:
$$
T_t\varphi_k=\varphi_k -\int_0^t T_s L\varphi_k\, ds.
$$
Using the contraction property \cite[Proposition 1.6]{Sturm_1995} of the semigroup $\{T_t\}_{t\geq 0}$ on $L^p(S_{n,m};d\mu)$, for all $p\in [1,\infty]$, and the preceding identity, we obtain that
$$
\|T_t\varphi_k-\varphi_k\|_{L^p(S_{n,m};d\mu)} \leq t \|L\varphi_k\|_{L^p(S_{n,m};d\mu)},\quad \forall\, k \geq 1.
$$
From conditions \eqref{eq:Ellipticity_all_second_order_coeff}, \eqref{eq:Outside_neighborhood_O}, \eqref{eq:Outside_neighborhood_O_a_ii}, \eqref{eq:b_coeff_bound_upper} and \eqref{eq:b_coeff_away_O}, it follows that there is a positive constant, $C=C(K,m,n,p)$, such that
$$
\|L\varphi_k\|_{L^p(S_{n,m};d\mu)} \leq Ck^{\frac{n+m}{p}-1},\quad\forall\, k \geq 1.
$$
Choosing $p$ large enough, we obtain that $\|L\varphi_k\|_{L^p(S_{n,m};d\mu)}\rightarrow 0$, as $k\rightarrow\infty$, and so,
there is a subsequence $\{T_t\varphi_k-\varphi_k\}_{k\geq 1}$ which converges to $0$, a.e. on $S_{n,m}$.

Using the upper bound estimates of the fundamental solution (see \eqref{eq:Upper_bound_transition_density} below), it follows that for all compact sets, $K\subset \bar S_{n,m}$, and all $t>0$, we have 
$$
\sup_{z\in K} |T_t\varphi_k(z)-T_t 1| \rightarrow 0,\quad\hbox{ as } k\rightarrow\infty.
$$ 
By \cite[Corollary 4.1]{Epstein_Mazzeo_cont_est}, the functions $T_t\varphi_k$ belong to $C(\bar S_{n,m})$, for all $k \geq 1$, and so the function $T_t 1$ is continuous on $\bar S_{n,m}$. Since the sequence $(T_t\varphi_k-\varphi_k) \rightarrow 0$ a.e. on $S_{n,m}$, as $k\rightarrow\infty$, $\varphi_k\rightarrow 1$ pointwise on $S_{n,m}$, as $k\rightarrow\infty$, and $T_t 1$ is continuous on $\bar S_{n,m}$, it follows that identity \eqref{eq:Transition_probabilities_prim} holds. This concludes the proof of identity \eqref{eq:Transition_probabilities}.
\end{proof}

\subsection{Distribution estimates}
\label{subsec:Distribution_estimates}
From \cite[Corollary 4.3]{Epstein_Mazzeo_cont_est}, we have the following upper estimates of the fundamental solution. There is a positive constant, $C$, such that
\begin{align}
\label{eq:Upper_bound_transition_density}
p(t,z^0,z) \leq C\frac{e^{-\frac{1}{8t}\rho^2(z^0,z)}}{\sqrt{\mu(B_{\sqrt{t}}(z^0)) \mu(B_{\sqrt{t}}(z))}},\quad \forall\, t>0,\quad\forall z^0, z \in\bar S_{n,m}.
\end{align}
We denote by $\tilde p(t,z^0,z)$ the function on the right-hand side of the preceding inequality, \eqref{eq:Upper_bound_transition_density}. We use the upper bound estimate of the fundamental solution \eqref{eq:Upper_bound_transition_density} to prove distribution estimates. Our goal is to prove the following two results:
\begin{lem}[$L^q$-distribution estimates]
\label{lem:L_q_transition_density}
Let $\Omega \subseteq S_{n,m}$ be a domain and $T>0$. Then there is a positive constant, $q_0=q_0(K,m,n)\in (1,2)$, such that for all $q\in [1,q_0)$, there are positive constants, $C_0=C_0(\bar b, \|b\|_{C^1(\bar S_{n,m};\RR^n)},m,n,q,T)$ and $\beta=\beta(q)<1$, such that 
\begin{align}
\label{eq:L_q_transition_density_aux}
 \|p^{\Omega}(t,z^0,\cdot)\|_{L^q(\Omega;d\mu)} \leq C_0 t^{-\beta},\quad \forall\, z^0 \in\underline\Omega,\quad\forall\, t\in (0,T].
\end{align}
and 
\begin{align}
\label{eq:L_q_transition_density}
\int_0^T \|p^{\Omega}(t,z^0,\cdot)\|_{L^q(\Omega;d\mu)}\, dt \leq C_0,\quad \forall\, z^0 \in\underline\Omega.
\end{align}
\end{lem}

\begin{lem}[H\"older distribution estimates]
\label{lem:Continuity_transition_density}
There is a positive constant, $\alpha_0=\alpha_0(K,m,n)$, such that for all $\alpha\in (\alpha_0,\infty)$, there is a positive constant, $\gamma=\gamma(\alpha,K,m,n)$, such that for all $T>0$, we can find a positive constant, $C=C(\alpha,\bar b,\|b\|_{C^1(\bar S_{n,m};\RR^n)},m,n,T)$, with the property that 
\begin{align}
\label{eq:Continuity_transition_density}
\int_{\Omega} \rho^{\alpha}(z^0,z)p^{\Omega}(t,z^0, z) \, d\mu(z) \leq C t^{n+m+\gamma},\quad \forall\, t \in (0,T],\quad\forall\, z^0 \in\underline\Omega.
\end{align}
\end{lem}

\begin{rmk}[Distribution estimates for standard Brownian motion]
In the case of $n$-dimensional Brownian motion, the transition probability densities are given by
$$
p(t,z^0,z) = \frac{1}{(2\pi t)^{n/2}}e^{-\frac{|z-z^0|^2}{2t}},\quad\forall\, z^0, z\in \RR^n, \quad\forall\, t>0.
$$
Direct calculations give us that, for all $\alpha>0$, there is a positive constant, $C=C(\alpha,n)$, such that
\begin{equation}
\label{eq:Continuity_transition_density_Brownian_motion}
\int_{\RR^n} |z-z^0|^{\alpha} p(t,z^0,z)\, dz = C t^{\alpha/2}.
\end{equation}
Therefore, estimate \eqref{eq:Continuity_transition_density} holds in the case of Brownian motion with $\gamma=\alpha/2-n$, when we assume that $\alpha>2n$. Moreover, we have that
$$
\int_{\RR^n} |p(t,z^0,z)|^q\, dz = \frac{(2\pi)^{n(1-q)/2}}{q^{n/2}} t^{(1-q)n/2},\quad\forall\, z^0\in\RR^n,\quad\forall\, t>0,
$$
from which it follows that, for all $z^0\in\RR^n$ and $T>0$, there is a positive constant, $C_0$, such that  
$$
\int_0^T \|p(t,z^0,\cdot)\|_{L^q(\RR^n)}\, dt < C_0,
$$
for all $q \in (0,(n+2)/n)$. Therefore, in the case of standard Brownian motion, estimate \eqref{eq:L_q_transition_density} holds, for all $q\in (0,(n+2)/n)$.
\end{rmk}

To prove Lemmas \ref{lem:Continuity_transition_density} and \ref{lem:L_q_transition_density} we need the following elementary result concerning the $\mu$-measure of the balls with respect to the distance function $\rho$. 
\begin{lem}[Measure of the balls]
\label{lem:Measure_of_balls}
Assume that the coefficient $b \in C^1(\bar S_{n,m}; \RR^n)$,  satisfies conditions \eqref{eq:b_coeff_bound_lower}, \eqref{eq:b_coeff_bound_upper} and \eqref{eq:b_coeff_away_O}. Then there are positive constants, $C=C(\bar b,\|b\|_{C^1(\bar S_{n,m};\RR^n)},m,n)$ and $r_0=r_0(\bar b, \|b\|_{C^1(\bar S_{n,m};\RR^n)},m,n)$, such that for all $r\in (0,r_0/2)$ and $ z^0\in \bar S_{n,m}$, we have
\begin{equation}
\label{eq:Measure_ball}
\begin{aligned}
\frac{1}{C}r^{m+n}  \prod_{i\in I(z^0)} \left|\sqrt{x^0_i}\vee r\right|^{2b_i(z^0)-1} 
\leq \mu(B_r(z^0))
\leq  Cr^{m+n}  \prod_{i\in I(z^0)} \left|\sqrt{x^0_i}\vee r\right|^{2b_i(z^0)-1},
\end{aligned}
\end{equation}
where we let
\begin{equation}
\label{eq:I_z_0}
I(z^0):=\left\{i\in\{1,\ldots,n\}:\ x^0_i\in [0,r_0]\right\}.
\end{equation}
\end{lem}

\begin{proof}
Because we assume that the coefficient $b$ belongs to $C^1(\bar S_{n,m};\RR^n)$ and that condition \eqref{eq:b_coeff_bound_lower} holds, there are positive constant, $c=\|b\|_{C^1(\bar S_{n,m};\RR^n)}$ and $r_1=r_1(\bar b,\|b\|_{C^1(\bar S_{n,m};\RR^n)})$ such that, for all $i=1,\ldots,n$, we have 
\begin{align*}
b_i(z) &\geq \frac{\bar b}{2},\quad\forall\, z=(x,y)\in \bar S_{n,m}\quad\hbox{such that } x_i\in (0,r_1),\\
|b_i(z)-b_i(z^0)| &\leq  c r,\quad\forall\, z\in B_r(z^0),\quad\forall\, r>0.
\end{align*}
Thus there is a positive constant, $r_0=r_0(\bar b, \|b\|_{C^1(\bar S_{n,m};\RR^n)},m,n)$, small enough so that 
\begin{equation}
\label{eq:Ineq_b_i_small_z}
\frac{\bar b}{4}<\beta_i-cr \leq b_i(z) \leq \beta_i+cr,\quad\forall z\in B_r(z^0),\quad\forall\, r\in (0,r_0/2),\quad\forall\, i \in I(z^0),
\end{equation}
where we denote $\beta_i:=b_i(z^0)$, for all $i=1,\ldots,n$. We prove the first inequality in \eqref{eq:Measure_ball}. From the preceding inequality, it follows that 
\begin{align*}
\mu(B_r(z^0)) &\geq \int_{B_r(z^0)} \prod_{i\in I(z^0)}^n x_i^{\beta_i+cr-1}\,dx_i \prod_{i\in I^c(z^0)}^n x_i^{b_i(z)-1}\,dx_i\prod_{l=1}^m dy_l,\quad\forall\, r\in (0,r_0/2),
\end{align*}
where we let $I^c(z^0):=\{1,\ldots,n\}\backslash I(z^0)$. Using property \eqref{eq:Equivalent_intrinsic_metric} of the distance function $\rho$, it follows from the preceding inequality, and conditions \eqref{eq:b_coeff_bound_upper} and \eqref{eq:b_coeff_away_O} that there is a positive constant, $C_1=C_1(r_0,K,m,n)$, such that
\begin{align*}
\mu(B_r(z^0)) 
&\geq C_1r^{m+|I^c(z^0)|}\prod_{i\in I(z^0)} \int_{\left\{\left|\sqrt{x_i}-\sqrt{x^0_i}\right|<r\right\}} x_i^{\beta_i+cr-1}\,dx_i \\
&= C_1r^{m+|I^c(z^0)|}\prod_{i\in I(z^0)} \int_{\left(\sqrt{x^0_i}-r\right)^2\vee 0}^{\left(\sqrt{x^0_i}+r\right)^2} x_i^{\beta_i+cr-1}\,dx_i,\quad\forall\, r\in (0,r_0/2).
\end{align*}
We can find a positive constant, $C_2$, such that
\begin{align*}
\int_{\left(\sqrt{x^0_i}-r\right)^2\vee 0}^{\left(\sqrt{x^0_i}+r\right)^2} x_i^{\beta_i+cr-1} \, dx_i&\geq C_2r\left|\sqrt{x^0_i}\vee r\right|^{2(\beta_i+cr)-1},
\end{align*}
and so, we obtain that there are positive constants, $C=C(\bar b,\|b\|_{C^1(\bar S_{n,m};\RR^n)},m,n)$ and $r_0=r_0(\bar b, \|b\|_{C^1(\bar S_{n,m};\RR^n)},m,n)$, such that
\begin{align*}
\mu(B_r(z^0)) 
&\geq Cr^{m+|I^c(z^0)|+|I(z^0)|}  \prod_{i\in I(z^0)} \left|\sqrt{x^0_i}\vee r\right|^{2(\beta_i+cr)-1}\\
&= Cr^{m+n}  \prod_{i\in I(z^0)} \left|\sqrt{x^0_i}\vee r\right|^{2(\beta_i+cr)-1}, \quad\forall\, r\in (0,r_0/2),
\end{align*}
from where the left-hand side of inequality \eqref{eq:Measure_ball} immediately follows. The right-hand side of inequality \eqref{eq:Measure_ball} is proved by a similar argument, and so, we omit the detailed proof.
\end{proof}

We first give the
\begin{proof}[Proof of Lemma \ref{lem:L_q_transition_density}]
Let $q \in [1,\infty)$. The range of the constant $q$ will be suitably chosen below. Notice that inequality \eqref{eq:L_q_transition_density} is a consequence of inequality \eqref{eq:L_q_transition_density_aux}, and so, it is sufficient to prove that \eqref{eq:L_q_transition_density_aux} holds. Choosing $\Omega_1:=\Omega$ and $\Omega_2:=S_{n,m}$ in Lemma \ref{lem:Fundam_sol_different_domains}, we see that it is sufficient to prove estimates \eqref{eq:L_q_transition_density_aux} and \eqref{eq:L_q_transition_density} for $p(t,z^0,z)$, instead of $p^{\Omega}(t,z^0,z)$. From the expression of the function $\tilde p(t,z^0,z)$ on the right-hand side of inequality \eqref{eq:Upper_bound_transition_density}, it is clear that it is enough to establish that 
there is a positive constant, $q_0=q_0(K,m,n)\in (1,2)$, such that for all $q\in [1,q_0)$, there are positive constants, $C=C(\bar b, \|b\|_{C^1(\bar S_{n,m};\RR^n)},m,n,q,T)$ and $\beta=\beta(q)<1$, such that we have
\begin{align}
\label{eq:L_q_transition_density_aux_tilde}
 \|\tilde p(t,z^0,\cdot)\|_{L^q(S_{n,m};d\mu)} \leq C_0 t^{-\beta},\quad \forall\, z^0 \in\bar S_{n,m},\quad\forall\, t\in (0,T].
\end{align}
Moreover, notice that the positive constant $T$ can be chosen as small as we like. Let $r_0$ be the positive constant appearing in the conclusion of Lemma \ref{lem:Measure_of_balls}. Without loss of generality, we may assume that
\begin{equation}
\label{eq:Upper_bound_T}
\sqrt{T} \leq \frac{r_0}{2}.
\end{equation}
It follows from the left-hand side of inequality \eqref{eq:Measure_ball}, using \eqref{eq:Upper_bound_T}, that there is a positive constant, $C=C(\bar b, \|b\|_{C^1(\bar S_{n,m};\RR^n)},m,n)$, such that
\begin{equation}
\label{eq:Lower_bound_ball}
\mu(B_{\sqrt{t}}(z)) \geq C t^{(m+n)/2} \prod_{i\in I(z)} \left(x_i\vee t\right)^{b_i(z)-1/2},\quad\forall z\in \bar S_{n,m},\quad\forall t\in (0,T],
\end{equation}
where we recall the definition of the set of indices $I(z)$ in \eqref{eq:I_z_0}. To estimate $\|\tilde p(t,z^0,\cdot)\|_{L^q(S_{n,m};d\mu)}$, we consider the set 
\begin{align*}
A_t(z^0)&:=\left\{z=(x,y)\in S_{n,m}:\, \left|\sqrt{x_i}-\sqrt{x^0_i}\right|<t^{\alpha}\hbox{ for } i\in I(z^0),\right.\\
&\qquad\left. |x_i-x^0_i|<t^{\alpha}\hbox{ for } i\in I^c(z^0),\hbox{ and }
 |y_l-y^0_l|<t^{\alpha} \hbox{ for } l\in\{1\ldots,m\}\right\},
\end{align*}
where we choose $\alpha\in (0,1/2)$, and we denote $A^c_t(z^0):=S_{n,m}\backslash A_t(z^0)$. We split the proof into two steps in which we estimate $\|\tilde p(t,z^0,\cdot)\|_{L^q(A^c_t(z^0);d\mu)}$ and $\|\tilde p(t,z^0,\cdot)\|_{L^q(A_t(z^0);d\mu)}$, respectively.

\setcounter{step}{0}
\begin{step}[Estimate of $\|\tilde p(t,z^0,\cdot)\|_{L^q(A^c_t(z^0);d\mu)}$]
\label{step:L_q_estimate_on_A_c_t_z_0}
In this step, we prove that there are positive constants, $C=C(\bar b, \|b\|_{C^1(\bar S_{n,m};\RR^n)},m,n)$ and $c=c(m,n)$, such that
\begin{equation}
\label{eq:L_q_estimate_on_A_c_t_z_0}
\|\tilde p(t,z^0,\cdot)\|_{L^q(A^c_t(z^0);d\mu)} \leq Ct^{-q(m+n)/2-qnK}e^{-\frac{qc}{t^{1-2\alpha}}}.
\end{equation}
From inequalities \eqref{eq:Lower_bound_ball} and \eqref{eq:Upper_bound_transition_density}, we have that
\begin{equation}
\label{eq:Integral_L_q_estimate_on_A_c_t_z_0}
\begin{aligned}
\|\tilde p(t,z^0,\cdot)\|^q_{L^q(A^c_t(z^0);d\mu)} 
&\leq C t^{-q(m+n)/2} \prod_{i\in I(z^0)} \left(x^0_i\vee t\right)^{-q(b_i(z^0)-1/2)/2} \\
&\quad\int_{A^c_t(z^0)} \prod_{j\in I(z)} \left(x_j\vee t\right)^{-q(b_j(z)-1/2)/2} e^{-\frac{q\rho^2(z,z^0)}{8t}}\, d\mu(z),
\end{aligned}
\end{equation}
which, with the aid of inequality \eqref{eq:b_coeff_bound_upper}, gives us
\begin{equation}
\label{eq:Integral_L_q_estimate_on_A_c_t_z_0_aux}
\begin{aligned}
\|\tilde p(t,z^0,\cdot)\|^q_{L^q(A^c_t(z^0);d\mu)} 
&\leq C t^{-q(m+n)/2} t^{-qnK} \int_{A^c_t(z^0)} e^{-\frac{q\rho^2(z,z^0)}{8t}}\, d\mu(z),
\end{aligned}
\end{equation}
Using the first inequality in \eqref{eq:Ineq_b_i_small_z}, together with \eqref{eq:b_coeff_bound_upper} and \eqref{eq:b_coeff_away_O}, we obtain for all $i=1,\ldots,n$ that 
\begin{equation}
\label{eq:Inequality_weight}
x_i^{b_i(z)-1} \leq x_i^{\bar b/4-1} \mathbf{1}_{\{x_i\in(0,r_0/2)\}}+x_i^{-K-1} \mathbf{1}_{\{x_i\in [r_0/2,1)\}} +\mathbf{1}_{\{x_i\in [1,\infty)\}},\quad\forall\, z\in S_{n,m},
\end{equation}
and we denote the function on the right-hand side of the preceding inequality by
\begin{equation}
\label{eq:Inequality_weight_function}
\varphi(x_i)= x_i^{\bar b/4-1} \mathbf{1}_{\{x_i\in(0,r_0/2)\}}+x_i^{-K-1} \mathbf{1}_{\{x_i\in [r_0/2,1)\}} +\mathbf{1}_{\{x_i\in [1,\infty)\}},\quad\forall\, z\in S_{n,m}.
\end{equation}
Using property \eqref{eq:Equivalent_intrinsic_metric} of the distance function $\rho(z,z^0)$, together with definition \eqref{eq:Measure} of the measure $d\mu(z)$, we obtain from \eqref{eq:Integral_L_q_estimate_on_A_c_t_z_0_aux} that there is a positive constant, $c=c(m,n)$, such that 
\begin{equation}
\label{eq:Integral_L_q_estimate_on_A_c_t_z_0_aux_1}
\begin{aligned}
&\|\tilde p^{\Omega}(t,z^0,\cdot)\|^q_{L^q(A^c_t(z^0);d\mu)} 
\leq C t^{-q(m+n)/2-qnK} \\
&\qquad\left(\sum_{i\in I(z^0)} \int_{\left\{\left|\sqrt{x_i}-\sqrt{x^0_i}\right|>t^{\alpha}\right\}} e^{-\frac{qc\rho^2(x_i,x^0_i)}{8t}}\varphi(x_i)\, dx_i \right.\\
&\qquad\quad\quad\prod_{\stackrel{j=1}{j\neq i}}^n \int_0^{\infty} e^{-\frac{qc\rho^2(x_j,x^0_j)}{8t}}\varphi(x_j)\, dx_j
\prod_{l=1}^m \int_{\RR} e^{-\frac{qc\rho^2(y_l,y^0_l)}{8t}}\, dy_l\\
&\qquad+\sum_{i\in I^c(z^0)} \int_{\left\{\left|x_i-x^0_i\right|>t^{\alpha}\right\}} e^{-\frac{qc\rho^2(x_i,x^0_i)}{8t}}\varphi(x_i)\, dx_i \\
&\qquad\quad\quad\prod_{\stackrel{j=1}{j\neq i}}^n \int_0^{\infty} e^{-\frac{qc\rho^2(x_j,x^0_j)}{8t}}\varphi(x_j)\, dx_j
\prod_{l=1}^m \int_{\RR} e^{-\frac{qc\rho^2(y_l,y^0_l)}{8t}}\, dy_l\\
&\qquad+\sum_{l=1}^m\prod_{i=1}^n \int_0^{\infty} e^{-\frac{qc\rho^2(x_i,x^0_i)}{8t}}\varphi(x_i)\, dx_i\\
&\left.\qquad\quad\quad
\int_{\left\{\left|y_l-y^0_l\right|>t^{\alpha}\right\}}  e^{-\frac{qc\rho^2(y_l,y^0_l)}{8t}}\, dy_l
\prod_{\stackrel{k=1}{k\neq l}}^m \int_{\RR}  e^{-\frac{qc\rho^2(y_k,y^0_k)}{8t}}\, dy_k\right).
\end{aligned}
\end{equation}
The parenthesis on the right-hand side of the preceding inequality can be written as the sum of three terms, $I_1+I_2+I_3$. We show that there are positive constants, $C=C(\bar b, \|b\|_{C^1(\bar S_{n,m};\RR^n)},m,n)$ and $c=c(m,n)$, such that
\begin{equation}
\label{eq:Integral_L_q_estimate_on_A_c_t_z_0_aux_2}
I_1+I_2+I_3 \leq Ce^{-\frac{cq}{t^{1-2\alpha}}},
\end{equation}
which implies estimate \eqref{eq:L_q_estimate_on_A_c_t_z_0} by inequality \eqref{eq:Integral_L_q_estimate_on_A_c_t_z_0_aux_1}.
We will only give the detailed proof of the fact that 
\begin{equation}
\label{eq:Integral_L_q_estimate_on_A_c_t_z_0_aux_3}
I_1 \leq Ce^{-\frac{cq}{t^{1-2\alpha}}},
\end{equation}
because the estimates for the integrals $I_2$ and $I_3$ can be obtained by a similar argument. Let $i\in I(z^0)$. Using property \eqref{eq:Equivalent_intrinsic_metric} of the distance function $\rho$, there are positive constants, $C=C(\bar b, \|b\|_{C^1(\bar S_{n,m};\RR^n)},m,n)$ and $c=c(m,n)$, such that
\begin{align*}
\int_{\left\{\left|\sqrt{x_i}-\sqrt{x^0_i}\right|>t^{\alpha}\right\}} e^{-\frac{qc\rho^2(x_i,x^0_i)}{8t}}x_i^{\bar b/4-1}\mathbf{1}_{\{x_i\in(0,r_0/2)\}}\, dx_i
&\leq e^{-\frac{qc}{8t^{1-2\alpha}}} \int_0^{r_0/2} x_i^{\bar b/4-1}\, dx_i\\
&\leq C e^{-\frac{qc}{t^{1-2\alpha}}}.
\end{align*}
Similarly, we obtain that
\begin{align*}
\int_{\left\{\left|\sqrt{x_i}-\sqrt{x^0_i}\right|>t^{\alpha}\right\}} e^{-\frac{qc\rho^2(x_i,x^0_i)}{8t}}x_i^{-K-1}\mathbf{1}_{\{x_i\in[r_0/2,1)\}}\, dx_i
&\leq e^{-\frac{qc}{8t^{1-2\alpha}}} \int_{r_0/2}^1 x_i^{-K-1}\, dx_i\\
&\leq C e^{-\frac{qc}{t^{1-2\alpha}}},
\end{align*}
and, we also have that
\begin{align*}
\int_{\left\{\left|\sqrt{x_i}-\sqrt{x^0_i}\right|>t^{\alpha}\right\}} e^{-\frac{qc\rho^2(x_i,x^0_i)}{8t}}\mathbf{1}_{\{x_i\in[1,\infty)\}}\, dx_i
&\leq \int_1^{\infty} e^{-\frac{qc|x_i-x^0_i|^2}{8t}}\, dx_i\\
&\leq C e^{-\frac{qc}{t}},
\end{align*}
for some positive constant, $c=c(m,n)$. In the last inequality, we used the fact that $|x_i-x^0_i|\geq |1-r_0|>0$, since $x_i\in [1,\infty)$ and $x^i_0\in (0,r_0)$, where we recall that $i\in I(z^0)$ and the set of indices $I(z^0)$ is defined in \eqref{eq:I_z_0}. Using definition \eqref{eq:Inequality_weight_function} of the function $\varphi(x_i)$, we see that there are positive constants, $C=C(\bar b,\|b\|_{C^1(\bar S_{n,m};\RR^n)},m,n)$ and $c=c(m,n)$, such that
$$
\int_{\left\{\left|\sqrt{x_i}-\sqrt{x^0_i}\right|>t^{\alpha}\right\}} e^{-\frac{qc\rho^2(x_i,x^0_i)}{8t}}\varphi(x_i)\, dx_i
\leq C e^{-\frac{qc}{t^{1-2\alpha}}}.
$$
We notice that, for all $j\in\{1,\ldots,n\}$ and $l\in\{1,\ldots,m\}$, we also have the very rough estimates:
\begin{align*}
\int_0^{\infty} e^{-\frac{qc\rho^2(x_j,x^0_j)}{8t}}\varphi(x_j)\, dx_j &\leq C,\\
\int_{\RR} e^{-\frac{qc\rho^2(y_l,y^0_l)}{8t}}\, dy_l &\leq C,
\end{align*}
where $C=C(\bar b, \|b\|_{C^1(\bar S_{n,m};\RR^n)},m,n)$ is a positive constant. Thus combining the preceding three inequalities it follows that estimate \eqref{eq:Integral_L_q_estimate_on_A_c_t_z_0_aux_3} holds. A similar argument implies estimate \eqref{eq:Integral_L_q_estimate_on_A_c_t_z_0_aux_2}, and inequality \eqref{eq:Integral_L_q_estimate_on_A_c_t_z_0_aux_1} yields estimate \eqref{eq:L_q_estimate_on_A_c_t_z_0}. This completes the proof of Step \ref{step:L_q_estimate_on_A_c_t_z_0}.
\end{step}

\begin{step}[Estimate of $\|\tilde p^{\Omega}(t,z^0,\cdot)\|_{L^q(A_t(z^0);d\mu)}$]
\label{step:L_q_estimate_on_A_t_z_0}
In this step, we prove that there is a positive constant, $C=C(\bar b, \|b\|_{C^1(\bar S_{n,m};\RR^n)},m,n)$, such that
\begin{equation}
\label{eq:L_q_estimate_on_A_t_z_0}
\|\tilde p^{\Omega}(t,z^0,\cdot)\|^q_{L^q(A_t(z^0);d\mu)} 
\leq C t^{(m+n)(\alpha-q/2)}
 \left(t^{nK(-q/2-\alpha q+2\alpha)}+t^{2nK(1-q)}\right),
\end{equation}
for all $t\in (0,T]$, where the positive constant $T$ is chosen to satisfy conditions \eqref{eq:Upper_bound_T_1} and \eqref{eq:Upper_bound_T_2} below. From inequality \eqref{eq:Upper_bound_T}, we may assume without loss of generality that the positive constant $r_0$ is small enough such that there is a positive constant, $C_1=C_1(\bar b, \|b\|_{C^1(\bar S_{n,m};\RR^n)},m,n)$, with the property that
$$
x_i \geq C_1,\quad\forall\, z=(x,y)\in A_t(z_0),\quad\forall\, i\in I^c(z^0),\quad\forall\, t\in [0,T].
$$
Using the fact that the coefficient function $b(z)$ belongs to $C^1(\bar S_{n,m};\RR^n)$, and letting $c_1=\|b\|_{C^1(\bar S_{n,m};\RR^n)}$, we have that
\begin{equation}
\label{eq:Bounds_b_i}
|b_i(z) - b_i(z^0)| \leq c_1 t^{\alpha},\quad\forall\, z\in A_t(z^0),\quad\forall\, i=1,\ldots,n. 
\end{equation}
From the first inequality in \eqref{eq:Ineq_b_i_small_z}, we have that $b_i(z^0)\geq \bar b/4$, for all $i\in I(z^0)$. Choosing the positive constant $T$ such that
\begin{equation}
\label{eq:Upper_bound_T_1}
T \leq \left(\frac{\bar b}{8c_2}\right)^{1/\alpha},
\end{equation}
we have that
$$
b_i(z^0) -c_2t^{\alpha} \geq \frac{\bar b}{8}>0,\quad\forall\, z=(x,y)\in A_t(z^0),\quad\forall\, i\in I(z^0),\quad\forall\, t\in [0,T],
$$
where $T$ satisfies the bounds \eqref{eq:Upper_bound_T} and \eqref{eq:Upper_bound_T_1}. From \eqref{eq:Bounds_b_i}, it follows that
\begin{equation}
\label{eq:Bounds_b_i_I_z_0}
b_i(z^0) +c_2t^{\alpha} \geq b_i(z) \geq b_i(z^0)-c_2t^{\alpha}>0,\quad\forall\, z=(x,y)\in A_t(z^0),\quad\forall\, i\in I(z^0).
\end{equation}
Choosing now $i\in I^c(z^0)$, we have that $x^0_i\geq r_0$, and so, it follows that
\begin{equation}
\label{eq:Lower_bound_x_i_not_in_I_Z_0}
x_i \geq x^0_i-t^{\alpha} \geq r_0-t^{\alpha} \geq \frac{r_0}{2},\quad\forall\, z=(x,y)\in A_t(z^0),\quad\forall\, t\in[0,T],
\end{equation}
 where we choose $T$ such that it satisfies the upper bound
\begin{equation}
\label{eq:Upper_bound_T_2}
T \leq \left(\frac{r_0}{2}\right)^{1/\alpha}.
\end{equation}
Inequality \eqref{eq:Integral_L_q_estimate_on_A_c_t_z_0} holds with $A^c_t(z^0)$ replaced by $A_t(z^0)$, and so, using property \eqref{eq:Equivalent_intrinsic_metric} of the distance function $\rho(z,z^0)$, definition \eqref{eq:Measure} of the measure $d\mu(z)$, and inequalities \eqref{eq:Bounds_b_i_I_z_0} and \eqref{eq:Lower_bound_x_i_not_in_I_Z_0}, we obtain that there are positive constants, $C=C(\bar b, \|b\|_{C^1(\bar S_{n,m};\RR^n)},m,n)$ and $c=c(m,n)$, such that 
\begin{align*}
&\|\tilde p(t,z^0,\cdot)\|^q_{L^q(A_t(z^0);d\mu)} 
\leq C t^{-q(m+n)/2} \prod_{i\in I(z^0)} \left(x^0_i\vee t\right)^{-q(b_i(z^0)-1/2)/2} \\
&\qquad\quad \prod_{i\in I(z^0)} \int_{\left\{\left|\sqrt{x_i}-\sqrt{x^0_i}\right|\leq t^{\alpha}\right\}} 
\left(x_i\vee t\right)^{-q(b_i(z^0)+c_2t^{\alpha}-1/2)/2} e^{-\frac{qc\left|\sqrt{x_i}-\sqrt{x^0_i}\right|^2}{8t}} x_i^{b_i(z^0)-c_2t^{\alpha}-1}\, dx_i\\
&\qquad\quad \prod_{j\in I^c(z^0)} \int_{\left\{\left|x_j-x^0_j\right|\leq t^{\alpha}\right\}} e^{-\frac{qc|x_j-x^0_j|^2}{8t}} \, dx_j
\prod_{l=1}^m \int_{\left\{\left|y_l-y^0_l\right|\leq t^{\alpha}\right\}} e^{-\frac{qc|y_l-y^0_l|^2}{8t}} \, dy_l.
\end{align*}
The preceding inequality holds for all $t\in (0,T]$, where $T$ satisfies both inequalities \eqref{eq:Upper_bound_T_1} and \eqref{eq:Upper_bound_T_2}. The integrals in the last two product terms of the preceding inequality can all be bounded by $t^{\alpha}$, and so, it follows that
\begin{align*}
&\|\tilde p(t,z^0,\cdot)\|^q_{L^q(A_t(z^0);d\mu)} 
\leq C t^{-q(m+n)/2} t^{\alpha(|I^c(z^0)|+m)}t^{q|I(z^0)|/2}\prod_{i\in I(z^0)} \left(x^0_i\vee t\right)^{-qb_i(z^0)/2} \\
&\qquad\quad \prod_{i\in I(z^0)} \int_{\left\{\left|\sqrt{x_i}-\sqrt{x^0_i}\right|\leq t^{\alpha}\right\}} 
\left(x_i\vee t\right)^{-q(b_i(z^0)+c_2t^{\alpha})/2}  x_i^{b_i(z^0)-c_2t^{\alpha}-1}\, dx_i.
\end{align*}
Direct calculations give us that there is a positive constant, $C=C(\bar b, \|b\|_{C^1(\bar S_{n,m};\RR^n)},m,n)$, such that for all $i\in I(z^0)$, we have that
\begin{align*}
\int_{\left\{\left|\sqrt{x_i}-\sqrt{x^0_i}\right|\leq t^{\alpha}\right\}} 
\left(x_i\vee t\right)^{-q(b_i(z^0)+c_2t^{\alpha})/2}  x_i^{b_i(z^0)-c_2t^{\alpha}-1}\, dx_i
\leq C\left(\sqrt{x^0_i}+t^{\alpha}\right)^{-qb_i(z^0)+2b_i(z^0)}.
\end{align*}
The preceding two inequalities yield
\begin{equation}
\label{eq:Integral_L_q_estimate_on_A_t_z_0}
\begin{aligned}
\|\tilde p(t,z^0,\cdot)\|^q_{L^q(A_t(z^0);d\mu)} 
&\leq C t^{(m+|I^c(z^0)|)(\alpha-q/2)} \\
&\quad\prod_{i\in I(z^0)} \left(x^0_i\vee t\right)^{-qb_i(z^0)/2} 
\left(\sqrt{x^0_i}+t^{\alpha}\right)^{-qb_i(z^0)+2b_i(z^0)}.
\end{aligned}
\end{equation}
Using condition \eqref{eq:b_coeff_bound_upper}, we see that if $0\leq x^0_i \leq \sqrt{t}$, we have that
$$
\left(x^0_i\vee t\right)^{-qb_i(z^0)/2} 
\left(\sqrt{x^0_i}+t^{\alpha}\right)^{-qb_i(z^0)+2b_i(z^0)}
\leq Ct^{K(-q/2-\alpha q+2\alpha)},
$$
while if $\sqrt{t} < x^0_i \leq r_0$, we obtain
$$
\left(x^0_i\vee t\right)^{-qb_i(z^0)/2} 
\left(\sqrt{x^0_i}+t^{\alpha}\right)^{-qb_i(z^0)+2b_i(z^0)}
\leq Ct^{2\alpha K(1-q)}.
$$
The preceding two inequalities together with estimate \eqref{eq:Integral_L_q_estimate_on_A_t_z_0} imply that
\begin{equation}
\label{eq:Integral_L_q_estimate_on_A_t_z_0_1}
\begin{aligned}
\|\tilde p(t,z^0,\cdot)\|^q_{L^q(A_t(z^0);d\mu)} 
&\leq C t^{(m+n)(\alpha-q/2)}
 \left(t^{nK(-q/2-\alpha q+2\alpha)}+t^{2\alpha nK(1-q)}\right),
\end{aligned}
\end{equation}
which immediately implies inequality \eqref{eq:L_q_estimate_on_A_t_z_0} since we assume that $\alpha\in (0,1/2)$ and $q\in [1,\infty)$.

This completes the proof of Step \ref{step:L_q_estimate_on_A_t_z_0}.
\end{step}

Combining inequalities \eqref{eq:L_q_estimate_on_A_c_t_z_0} and \eqref{eq:L_q_estimate_on_A_t_z_0}, we see that there are positive constants, $q_0=q_0(K,m,n)\in (1,2)$ and $\alpha\in (0,1/2)$, such that for all $q\in [1,q_0)$, there are positive constants, $C=C(\bar b, \|b\|_{C^1(\bar S_{n,m};\RR^n)},m,n,q)$ and $\beta=\beta(q)<1$, such that estimate \eqref{eq:L_q_transition_density_aux} holds. This completes the proof.
\end{proof}

We now give the proof of 
\begin{proof}[Proof of Lemma \ref{lem:Continuity_transition_density}]
Choosing $\Omega_1:=\Omega$ and $\Omega_2:=S_{n,m}$ in Lemma \ref{lem:Fundam_sol_different_domains}, and using \eqref{eq:Upper_bound_transition_density}, we see that it is sufficient to prove estimate \eqref{eq:Continuity_transition_density} for $\tilde p(t,z^0,z)$ instead of $p^{\Omega}(t,z^0,z)$. Notice that the positive constant $T$ can be chosen as small as we like. Let $r_0$ be the positive constant appearing in the conclusion of Lemma \ref{lem:Measure_of_balls}. Without loss of generality, we may assume that $T$ satisfies inequality \eqref{eq:Upper_bound_T}. It follows from the left-hand side of inequality \eqref{eq:Measure_ball}, using \eqref{eq:Upper_bound_T}, that there is a positive constant, $C=C(\bar b, \|b\|_{C^1(\bar S_{n,m};\RR^n)},m,n)$, such that
$$
\mu\left(B_{\sqrt{t}}(z)\right) \geq C t^{(m+n)/2} t^{nK},\quad\forall\, z\in \bar S_{n,m},\quad\forall\, t\in (0,T],
$$
where we recall the definition of the constant $K$ in \eqref{eq:b_coeff_bound_upper}. The supremum bound of the fundamental solution \eqref{eq:Upper_bound_transition_density}, together with the preceding inequality, gives us that
\begin{equation}
\label{eq:Cont_est_with_measure_balls}
\begin{aligned}
\int_{S_{n,m}}\rho^{\alpha}(z^0,z) \tilde p(t,z^0,z)\, d\mu(z)
\leq Ct^{-(m+n)/2-nK} \int_{S_{n,m}}\rho^{\alpha}(z^0,z) e^{-\frac{1}{8t}\rho^2(z^0,z)}\, d\mu(z).
\end{aligned}
\end{equation}
We see that there is a positive constant, $C=C(\alpha,m,n)$, such that
\begin{align*}
\int_{S_{n,m}}\rho^{\alpha}(z^0,z) e^{-\frac{1}{8t}\rho^2(z^0,z)}\, d\mu(z)
&\leq C\sum_{j=1}^n \int_{S_{n,m}}\rho^{\alpha}(x^0_j,x_j) e^{-\frac{1}{8t}\rho^2(z^0,z)}\, d\mu(z)\\
&\quad+C\sum_{k=1}^m \int_{S_{n,m}}\rho^{\alpha}(y^0_k,y_k) e^{-\frac{1}{8t}\rho^2(z^0,z)}\, d\mu(z).
\end{align*}
We estimate each term on the right-hand side of the preceding inequality. We show that there is a positive constant, $C=C(\alpha,\bar b,\|b\|_{C^1(\bar S_{n,m};\RR^n)},m,n)$, such that for all $j=1,\ldots,n$, we have
\begin{equation}
\label{eq:Cont_est_x}
\int_{S_{n,m}}\rho^{\alpha}(x^0_j,x_j) e^{-\frac{1}{8t}\rho^2(z^0,z)}\, d\mu(z) \leq Ct^{(m+\alpha)/2},
\end{equation}
and, for all indices $k=1,\ldots,m$, we have that
\begin{equation}
\label{eq:Cont_est_y}
\int_{S_{n,m}}\rho^{\alpha}(y^0_k,y_k) e^{-\frac{1}{8t}\rho^2(z^0,z)}\, d\mu(z) \leq Ct^{(m+\alpha)/2},
\end{equation}
We outline the proof of estimate \eqref{eq:Cont_est_x}, but inequality
\eqref{eq:Cont_est_y} can be deduced by a similar argument, and so, we do not include the detailed proof. 

Inequality \eqref{eq:Inequality_weight} and definition \eqref{eq:Measure} of the measure $d\mu(z)$ yield that there is a positive constant, $c=c(m,n)$, such that
\begin{equation}
\label{eq:Decomposition_est_cont_x}
\begin{aligned}
\int_{S_{n,m}}\rho^{\alpha}(x^0_j,x_j) e^{-\frac{1}{8t}\rho^2(z^0,z)}\, d\mu(z)
&\leq 
\int_0^{\infty} \rho^{\alpha}(x^0_j,x_j) e^{-\frac{c}{8t}\rho^2(x^0_j,x_j)}\varphi(x_j)\, dx_j\\
&\quad\prod_{\stackrel{i=1}{i\neq j}}^n \int_0^{\infty} e^{-\frac{c}{8t}\rho^2(x^0_i,x_i)}\varphi(x_i)\, dx_i 
 \prod_{k=1}^m\int_{\RR} e^{-\frac{c|y^0_k-y_k|^2}{8t}} \, dy_k,
\end{aligned}
\end{equation}
where we recall the definition of the function $\varphi$ in \eqref{eq:Inequality_weight_function}. From inequality \eqref{eq:Continuity_transition_density_Brownian_motion} applied with $\alpha=0$, we have that
\begin{equation}
\label{eq:Estimate_y_integral}
\int_{\RR} e^{-\frac{c|y^0_k-y_k|^2}{8t}} \, dy_k \leq C t^{1/2},\quad\forall\, k=1,\ldots,m,
\end{equation}
while using again identity \eqref{eq:Continuity_transition_density_Brownian_motion} and property \eqref{eq:Equivalent_intrinsic_metric} of the distance function $\rho$, there is a positive constant, $C=C(\bar b, K,m,n,T)$, such that for all $t\in (0,T]$, we have that
\begin{equation}
\label{eq:Estimate_x_integral}
\int_0^{\infty} e^{-\frac{c}{8t}\rho^2(x^0_i,x_i)} \varphi(x_i)\, dx_i \leq C,\quad\forall\, i=1,\ldots,n,\ i\neq j.
\end{equation}
Thus, using the preceding two inequalities, estimate \eqref{eq:Decomposition_est_cont_x} becomes
\begin{equation}
\label{eq:Decomposition_est_cont_x_aux}
\begin{aligned}
\int_{S_{n,m}}\rho^{\alpha}(x^0_j,x_j) e^{-\frac{c}{8t}\rho^2(z^0,z)}\, d\mu(z)
&\leq Ct^{m/2} \int_0^{\infty} \rho^{\alpha}(x^0_j,x_j) e^{-\frac{c}{8t}\rho^2(x^0_j,x_j)}\varphi(x_j)\, dx_j.
\end{aligned}
\end{equation}
It remains to estimate the integral on the right-hand side of the preceding inequality. Using definition \eqref{eq:Inequality_weight_function} of the function $\varphi(x_j)$, we write the integral on the right-hand side of inequality \eqref{eq:Decomposition_est_cont_x_aux} as a sum of three integrals, $I_1+I_2+I_3$, where the integral $I_1$ is taken over the interval $(0,r_0/2)$, the integral $I_2$ is over $(r_0/2,1)$,  and the last integral is over $(1,\infty)$. We estimate integral $I_1$, which satisfies the inequality
$$
I_1\leq C t^{\alpha/2} \int_0^{r_0/2} \left(\frac{\rho(x^0_j,x_j)}{\sqrt{t}}\right)^{\alpha} e^{-\frac{8}{c}\left(\frac{\rho(x^0_j,x_j)}{\sqrt{t}}\right)^2} x_j^{\bar b/4-1}\, dx_j.
$$
where $C=C(\alpha,m,n)$ is a positive constant. Because the function $s\mapsto s^{\alpha}e^{-s^2}$ is bounded on $\bar\RR_+$, and the function $s\mapsto s^{\bar b/4-1}$ is integrable on $(0,1)$, we see that there is a positive constant, $C=C(\bar b,r_0,m,n)$, such that
$I_1\leq Ct^{\alpha/2}$. A similar argument can be applied to estimate integrals $I_2$ and $I_3$ which satisfy the same estimate as $I_1$. Thus, using inequalities \eqref{eq:Decomposition_est_cont_x_aux}, \eqref{eq:Estimate_x_integral} and \eqref{eq:Estimate_y_integral}, it follows that there is a positive constant, $C=C(\alpha,\bar b,\|b\|_{C^1(\bar S_{n,m};\RR^n)},m,n)$, such that estimate \eqref{eq:Cont_est_x} holds. A similar argument can be applied to prove that estimate \eqref{eq:Cont_est_y} holds. Using \eqref{eq:Cont_est_x} and \eqref{eq:Cont_est_y} in inequality \eqref{eq:Cont_est_with_measure_balls} gives us that
\begin{align}
\int_{S_{n,m}}\rho^{\alpha} (z^0,z)\tilde p(t,z^0,z)\, d\mu(z)
\leq Ct^{(\alpha-n(2+K))/2}.
\end{align}
Choosing $\alpha_0 :=2(n+m)+n(2+K)$, the preceding inequality shows that, for all $\alpha\in (\alpha_0,\infty)$, there are positive constants, $C=C(\alpha,\bar b,\|b\|_{C^1(\bar S_{n,m};\RR^n)},m,n)$ and $\gamma=\gamma(\alpha,K,m,n)$, such that estimate \eqref{eq:Continuity_transition_density} holds. This completes the proof.
\end{proof}

\section{Construction of a continuous Markov process}
\label{subsec:Markov_process}
Let $\Omega\subseteq S_{n,m}$ be an open set. The goal of this section is to build a continuous, Markov process, $\{Z^{0,\Omega}(t)\}_{t\geq 0}$, associated to the semigroup $\{T^{\Omega}_t\}_{t\geq 0}$. The main result of this section is

\begin{prop}
\label{prop:Stopped_process}
Let $z\in\underline \Omega$. Then there is a filtered probability space, $\left(\cZ, \cF,\{\cF_t\}_{t \geq 0}, \PP^z\right)$, and a continuous, progressively measurable process, 
\begin{align}
\label{eq:Z_Omega}
Z^{0,\Omega}:[0,\infty)\times\cZ\rightarrow\bar\Omega,
\end{align}
such that for all bounded, Borel measurable functions, $f:\underline{\Omega}\rightarrow\RR$, we have that
\begin{equation}
\label{eq:Connection_semigroup}
\left(T^{\Omega}_t f\right)(z) = \EE_{\PP^z}\left[f(Z^{0,\Omega}(t))\mathbf{1}_{\{t<\tau_{\Omega}\}}\right],\quad\forall\, t\geq 0,
\end{equation}
where the random time $\tau_{\Omega}$ defined by
\begin{align}
\label{eq:tau_Omega}
\tau_{\Omega}:=\inf\left\{t\geq 0:\, Z^{0,\Omega}(t)\in\overline{\partial_1\Omega}\right\},
\end{align}
is a stopping time with respect to the filtration $\{\cF_t\}_{t\geq 0}$. Moreover, the process $\{Z^{0,\Omega}(t)\}_{t \geq 0}$ satisfies the Markov property, that is, for all bounded, Borel measurable functions, $f:\underline\Omega\rightarrow\RR$, and all $0\leq s\leq t$, we have
\begin{equation}
\label{eq:Markov_prop_stopped_process}
\EE_{\PP^z}\left[f\left(Z^{0,\Omega}(t)\right)\big{|} \cF_s\right] = \EE_{\PP^z}\left[f\left(Z^{0,\Omega}(t)\right)\big{|} Z^{0,\Omega}(s)\right]=\left(T^{\Omega}_{t-s}f\right)(Z^{0,\Omega}(s)),\quad\PP^z\hbox{-a.s.}
\end{equation}
and satisfies the property that
\begin{align}
\label{eq:Process_after_stopping_time}
\PP^z\left(Z^{0,\Omega} (t) \in\overline{\partial_1\Omega},\quad\forall\, t \geq\tau_{\Omega}\right)=1.
\end{align}
\end{prop}
The proof of Proposition \ref{prop:Stopped_process} makes use of the family of densities $\{p^{\Omega}(t) d\mu\}_{t\geq 0}$ constructed in Lemma \ref{lem:Measurability_fundam}.

Let $\Omega_{\Delta}:=\underline\Omega\cup\{\Delta\}$ be the one-point
compactification of $\underline\Omega$. The open neighborhoods of $\Delta$ are
sets of the form $(\underline\Omega\backslash K)\cup\{\Delta\}$, where
$K\subset \underline\Omega$ is a compact set. Let $\cB(\underline\Omega)$ and
$\cB(\Omega_{\Delta})$ denote the Borel $\sigma$-algebras on $\underline\Omega$
and $\Omega_{\Delta}$, respectively. The family of densities, $\{p^{\Omega}(t)
d\mu\}_{t> 0}$, are sub-Markovian on $(\underline\Omega,
\cB(\underline\Omega))$, in general. To obtain a Markovian family of
probability measures on $(\Omega_{\Delta}, \cB(\Omega_{\Delta}))$, we consider
the extension: 
\begin{equation*}
P^{\Omega}_{ext} : [0,\infty)\times\Omega_{\Delta}\times\cB(\Omega_{\Delta})\rightarrow [0,\infty),
\end{equation*}
and we let
\begin{equation}
\label{eq:Extended_measure}
P^{\Omega}_{ext} (t,z,B) = \int_{\underline\Omega\cap B} p^{\Omega}(t,z,w)\, d\mu(w) + \mathbf{1}_{\{\Delta\in B\}} \left(1-\int_{\underline\Omega} p^{\Omega}(t,z,w)\, d\mu(w)\right),
\end{equation} 
for all $t>0$, $z\in \underline\Omega$ and $B \in \cB(\Omega_{\Delta})$, and we let
\begin{equation}
\label{eq:Extended_measure_Delta}
P^{\Omega}_{ext} (t,\Delta,B) =	\mathbf{1}_{\{\Delta\in B\}}, 
\end{equation}
for all $t>0$ and $B \in \cB(\Omega_{\Delta})$. When $t=0$, we let
\begin{equation}
\label{eq:Extended_measure_t_0}
P^{\Omega}_{ext} (0,z,B) =	\mathbf{1}_{\{z\in B\}}, 
\end{equation}
for all $z\in \Omega_{\Delta}$ and $B \in \cB(\Omega_{\Delta})$. Then, we obtain that the family of transition probability measures $\{P^{\Omega}_{ext}(t)\}_{t\geq 0}$ is Markovian on $(\Omega_{\Delta}, \cB(\Omega_{\Delta}))$. 

Let $\DD_+$ denote the set of nonnegative dyadic numbers, and let 
$$
\cZ^0:=(\Omega_{\Delta})^{\DD_+}.
$$
We define the coordinate process, $Z^0:\DD_+\times\cZ^0\rightarrow \Omega_{\Delta}$, by
$$
Z^0(t,\omega):=\omega(t),\quad\forall\, t \in \DD_+, \quad\forall\, \omega\in\cZ^0.
$$
We consider the filtration generated by the cylinder sets,
$$
\cF'_t:=\sigma\left(Z^0(s):\, s\in[0,t]\cap\DD_+\right),\quad\forall\, t\in \DD_+,
$$
and we let $\cF'=\cup_{t\in \DD^+}\cF'_t$. 
 
\begin{rmk}[Extension of measurable functions from $\Omega$ to $\Omega_{\Delta}$]
Throughout the article we use the convention that any bounded, Borel measurable function $f$ on $\Omega$ is extended to $\Omega_{\Delta}$ by setting $f(\Delta)=0$.
\end{rmk}
We apply the Daniell-Kolmogorov Theorem \cite[Theorem 2.3.9]{Baudoin}, \cite[Theorem 2.2.2]{KaratzasShreve1991}, \cite[\S 7.2]{Fukushima_Oshima_Takeda} to obtain

\begin{lem}[Existence of a Markov process]
\label{lem:Process_dyadic_times}
Let $\nu(dz)$ be a probability measure on $\Omega_{\Delta}$. Then there is a probability measure, $\PP^{0,\nu}$, on the filtered measure space $(\cZ^0,\{\cF'_t\}_{t\in \DD_+}, \cF')$ such that for all $k\in\NN$, $0< t_1< t_2<\ldots<t_k$, $t_i\in\DD_+$, for $i=1,\ldots,k$, and all Borel measurable sets, $B_i\in\cB(\Omega_{\Delta})$, for $i=1,\ldots,k$, we have
\begin{equation}
\begin{aligned}
\label{eq:Finite_dim_marginal_distrib}
&\PP^{0,\nu}\left(Z^0(t_i)\in B_i,\,\forall\, i=1,\ldots,k\right) 
= \int_{\Omega_{\Delta}} \nu(dz)  \int_{B_1} P^{\Omega}_{ext} (t_1,z,dz_1) \\
&\qquad\quad\int_{B_2} P^{\Omega}_{ext} (t_2-t_1,z_1,dz_2)\ldots \int_{B_k} P^{\Omega}_{ext} (t_k-t_{k-1},z_{k-1},dz_k).
\end{aligned}
\end{equation}
Moreover, the following hold:
\begin{enumerate}
\item[1.] The law of $Z^0(0)$ is given by the probability measure $\nu(dz)$.
\item[2.] The coordinate process $Z^0$ is Markov, that is, for all bounded, Borel measurable functions $f$ on $\underline\Omega$, and all $0\leq s\leq t$ such that $s,t \in \DD_+$, we have that
\begin{align}
\label{eq:Markov_prop_dyadic_times}
\EE_{\PP^{0,\nu}} \left[f(Z^0(t))|\cF'_s\right] = \EE_{\PP^{0,\nu}} \left[f(Z^0(t))|Z^0(s)\right] = \left(T^{\Omega}_{t-s}f\right)
(Z^0(s))
\quad \PP^{0,\nu}\hbox{-a.s.}
\end{align}
\end{enumerate}
\end{lem}
\noindent
When we choose $\nu=\delta_z$ in Lemma \ref{lem:Process_dyadic_times}, for a point $z\in \Omega_{\Delta}$, we denote for simplicity $\PP^{0,z}:=\PP^{0,\nu}$ .

Let $\{\cF^0_t\}_{t\geq 0}$ be the right-continuous filtration defined by
$$
\cF^0_t :=\bigcap_{\stackrel{s>t}{s\in\DD_+}} \cF'_s,\quad\forall\, t \geq 0,
$$
and we set $\cF^0 :=\cup_{t \geq 0} \cF^0_t$. Clearly, we have that $\cF^0=\cF'$. We denote by $\tau_{\Delta}$ the random time defined by
\begin{align}
\label{eq:tau_Delta}
\tau_{\Delta}:=\inf\{t \in\DD_+:\, Z^0(t)=\Delta\}.
\end{align}
Notice that $\tau_{\Delta}$ takes values in $[0,\infty)$, as opposed to $\DD_+$.

\begin{rmk}[$\tau_{\Delta}$ is a stopping time ]
\label{rmk:tau_Delta_stopping_time}
 We show that $\tau_{\Delta}$ is a stopping time with respect to the right-continuous filtration $\{\cF^0_t\}_{t\geq 0}$. For all $t\geq 0$, the sets $\{\tau_{\Delta} \leq t\}$ can be represented in the form
\begin{align*}
\{\tau_{\Delta}\leq t\} &=   \{\tau_{\Delta} < t\} \cup \{\tau_{\Delta} = t\}\\
&= \left(\bigcup_{\stackrel{s<t}{s\in\DD_+}} \{Z^0(s)=\Delta\}\right) \cup \{Z^0(t)=\Delta\} \cup\left(\bigcap_{\stackrel{s>t}{s\in\DD_+}} \bigcup_{\stackrel{t\leq s'<s}{s'\in\DD_+}} \{Z^0(s')=\Delta\}\right),
\end{align*}
where the set $\{Z^0(t)=\Delta\}$ is empty in the preceding identity, if $t \notin\DD_+$. Using the fact that the filtration $\{\cF^0_t\}_{t \geq 0}$ is right-continuous, we see that
\begin{align}
\label{eq:Measurable_set_tau_Delta}
\{\tau_{\Delta} \leq t\} \in \cF^0_t,\quad\forall\, t \geq 0,
\end{align}
and so, the random time $\tau_{\Delta}$ is indeed a stopping time with respect to the filtration $\{\cF^0_t\}_{t\geq 0}$.
\end{rmk}

We now use the distribution estimate \eqref{eq:Continuity_transition_density} to establish the continuity of the paths of the coordinate process $\{Z^0(t)\}_{t\in\DD_+}$ in the Euclidean topology on $\underline\Omega$.
\begin{lem}[Continuity of the paths up to $\tau_{\Delta}$]
\label{lem:Continuity_paths_stopped}
Let $\nu(dz)$ be a probability measure on $\underline\Omega$. Then there is a measurable set, $\cZ^{0,\star}\in \cF^0$, such that
\begin{align}
\label{eq:measurable_subset1}
\PP^{0,\nu}(\cZ^{0,\star})=1,
\end{align}
and the paths of the process have the property that
\begin{equation}
\label{eq:Continuity_paths_stopped}
\begin{aligned}
&\PP^{0,\nu}\left(\omega\in\cZ^{0,\star}: [0,\tau_{\Delta})\cap\DD_+\ni t\mapsto Z^0(t,\omega) \hbox{ is continuous, and }\right.\\
&\qquad\left.\lim_{t \nearrow\tau_{\Delta}} Z^0(t,\omega) \hbox{ exists in } \bar\Omega\right)=1,
\end{aligned}
\end{equation}
\end{lem}

\begin{rmk}
The values of the process $Z^0(t)$, for all $t\in [0,\tau_{\Delta})\cap\DD_+$, belong to $\underline\Omega$, by definition of the stopping time $\tau_{\Delta}$. In identity \eqref{eq:Continuity_paths_stopped}, when we write 
$$
\lim_{t \nearrow\tau_{\Delta}} Z^0(t,\omega) \hbox{ exists in } \bar\Omega,
$$
we understand that the preceding limit exists in $\bar\Omega$ which we endow with the usual \emph{Euclidean} topology, \emph{not} in $\Omega_{\Delta}$ endowed with the one-point compactification topology. See also Lemma \ref{lem:Continuity_paths_compactification}.
\end{rmk}

\begin{proof}[Proof of Lemma \ref{lem:Continuity_paths_stopped}]
It is sufficient to show that, for all $T\in\DD_+$, we can find a $\cF^0$-measurable set $\cZ^{0,\star}$, such that identity \eqref{eq:measurable_subset1} holds, and such that
\begin{equation}
\label{eq:Continuity_paths_dyadic_times_stopped}
\begin{aligned}
&\PP^{0,\nu}\left(\omega\in\cZ^{0,\star}: [0,\tau^T_{\Delta})\cap\DD_+\ni t\mapsto Z^0(t,\omega) \hbox{ is continuous, and }\right.\\
&\qquad\left.\lim_{t \nearrow\tau^T_{\Delta}} Z^0(t,\omega) \hbox{ exists in } \bar\Omega\right)=1,
\end{aligned}
\end{equation}
where we define the stopping time $\tau_{\Delta}^T$ by $\tau^T_{\Delta}:=\tau_{\Delta}\wedge T$. Let $\alpha_0$ be the positive constant appearing in the statement of Lemma \ref{lem:Continuity_transition_density}, and choose $\alpha>\alpha_0$. Using the construction of the measure $\PP^{0,\nu}$ in Lemma \ref{lem:Process_dyadic_times}, we have that, for all $0\leq s<t\leq T$ such that $s,t\in\DD_+$,
\begin{align*}
&\EE_{\PP^{0,\nu}}\left[\rho^{\alpha}\left(Z^0(t),Z^0(s)\right)\mathbf{1}_{\{t<\tau_{\Delta}\}}\right] \\
&\quad\quad= \int_{\underline\Omega} \nu(dz) \int_{\underline\Omega} p^{\Omega}(s,z,z')\, d\mu(z')\int_{\underline\Omega} \rho^{\alpha}(z',z'')p^{\Omega}(t-s,z',z'')\, d\mu(z'')\\
&\quad\quad\leq \int_{\underline\Omega} d\nu(z) \int_{\underline\Omega} p^{\Omega}(s,z,z')\, d\mu(z')\sup_{z'\in\underline\Omega}
\int_{\underline\Omega} \rho^{\alpha}(z',z'')p^{\Omega}(t-s,z',z'')\, d\mu(z'').
\end{align*}
Lemma \ref{lem:Continuity_transition_density} shows that there are positive constants, $C_{\alpha}$ and $\gamma_{\alpha}$, such that
\begin{align*}
\sup_{z'\in\underline\Omega}\int_{\underline\Omega} \rho^{\alpha}(z',z'')p^{\Omega}(t-s,z',z'')\, d\mu(z'') \leq C_{\alpha}|t-s|^{n+m+\gamma_{\alpha}},\quad\forall\, 0\leq s<t\leq T.	
\end{align*}
Thus, the preceding two inequalities show that
\begin{align*}
\EE_{\PP^{0,\nu}}\left[\rho^{\alpha}\left(Z^0(t),Z^0(s)\right) \mathbf{1}_{\{t<\tau_{\Delta}\}}\right] \leq C_{\alpha} |t-s|^{1+\gamma_{\alpha}},\quad \forall\, 0\leq s<t<T.	
\end{align*}
Let $\beta:=\gamma_{\alpha}/\alpha$. From the proof of the Kolmogorov-\v Centsov Theorem \cite[Theorem 2.2.8 and Problem 2.2.9]{KaratzasShreve1991} applied with the distance function $\rho$ instead of the Euclidean distance, there is a measurable set, $\cZ^{0,*}\in \cF^0$, such that identity \eqref{eq:measurable_subset1} holds, and there is a positive constant, $C$, such that for all $\omega \in \cZ^{0,*}$, there is a positive measurable function, $\delta(\omega)>0$, such that
\begin{equation*}
\rho^{\alpha}\left(Z^0(t),Z^0(s)\right) \leq C |t-s|^{\beta},\quad\forall\, 0\leq s\leq t\leq \tau^T_{\Delta},\quad s,t \in \DD_+,\quad |t-s|<\delta(\omega).
\end{equation*}
Therefore, the paths $Z^0(t,\omega)$ are uniformly continuous for all $t\in [0,\tau_{\Delta}^T)\cap\DD_+$, and for all $\omega \in \cZ^{0,*}$, which shows that the limit
$$
\lim_{t \nearrow\tau^T_{\Delta}} Z^0(t,\omega) \hbox{ exists in } \bar\Omega,\quad\forall\, \omega \in \cZ^{0,*}.
$$
Because the positive constant $T$ was arbitrarily chosen, the conclusion follows.
\end{proof}

From identities \eqref{eq:Extended_measure_Delta} and \eqref{eq:Markov_prop_dyadic_times} (applied with $s=0$), we have that, for all bounded, Borel measurable functions $f$ on $\underline\Omega$, 
\begin{align}
\label{eq:Relationship_with_semigroup}
\EE_{\PP^{0,z}}\left[f(Z^0(t))\right]=\EE_{\PP^{0,z}}\left[f(Z^0(t))\mathbf{1}_{\{t <\tau_{\Delta}\}}\right] = \left(T^{\Omega}_t f\right)(z),\quad\forall\, z\in\underline\Omega,\quad\forall\, t \in\DD_+,
\end{align}
where we recall the convention that any Borel measurable function $f$ on $\underline\Omega$ is extended to $\Omega_{\Delta}$ by setting $f(\Delta)=0$.

We next prove the continuity of the paths of the process $\{Z^0(t)\}_{t\in\DD_+}$ in the topology of the one-point compactification. This result relies heavily on the fact that the set $\DD_+$ is countable to establish the measurability of certain sets in $\cZ^0$. The argument of the proof of Lemma \ref{lem:Continuity_paths_compactification} could not be applied if we replaced the set $\DD_+$ by $[0,\infty)$, and this is the reason why we start by constructing the process $\{Z^0(t)\}$ for all $t\in\DD_+$, as opposed to all $t\in [0,\infty)$.

\begin{lem}[Continuity of the paths in the topology of the one-point compactification]
\label{lem:Continuity_paths_compactification}
Let $z \in \underline\Omega$. Then there is a measurable set, $\cZ^{0,\star\star}\in \cF^0$, such that
\begin{align}
\label{eq:measurable_subset2}
\PP^{0,z}(\cZ^{0,\star\star})=1,
\end{align}
and the paths of the coordinate process, $\{Z^0(t,\omega)\}_{t \in \DD_+}$, are continuous in the topology of the one-point compactification, for all $\omega \in \cZ^{0,\star\star}$.
\end{lem}

\begin{proof}
Let $\cC$ denote the set of all $\omega \in \cZ^{0,\star}$ such that the path $\DD_+\ni t\mapsto Z^{0}(t,\omega) \in \Omega_{\Delta}$
is continuous in the topology of the one-point compactification of $\Omega_{\Delta}$. For all $t_0\in \DD_+$, we let $\cC_{t_0}$ denote the set of all $\omega \in \cZ^{0,\star}$ such that the path $\DD_+\ni t\mapsto Z^{0}(t,\omega) \in \Omega_{\Delta}$ is continuous \emph{at $t_0$} in the topology of the one-point compactification of $\Omega_{\Delta}$. Because the set $\DD_+$ is countable, it easily follows that $\cC_{t_0}\in \cF^0$, for all $t_0\in\DD_+$, and because we can write $\cC=\cap_{t_0\in\DD_+}\cC_{t_0}$, we also have that $\cC$ is a $\cF^0$-measurable set. Our goal is to prove that there is a set, $\cN\in \cF^0$, of $\PP^{0,z}$-measure zero, such that the measurable set $\cZ^{0,\star\star}:=\cC\backslash\cN$ satisfies  property \eqref{eq:measurable_subset2}. Because the set $\DD_+$ is countable, it is sufficient to show that, for all $t_0\in \DD_+$, there is a $\cF^0$-measurable set with $\PP^{0,z}$-measure zero, such that
\begin{align}
\label{eq:measurable_set_t_0}
\cC_{t_0}=\cZ^{0,\star}\backslash \cN_{t_0}.
\end{align}
We can then choose $\cN:=\cup_{t_0\in\DD_+}\cN_{t_0}$ to obtain that the measurable set $\cZ^{0,\star\star}:=\cC\backslash\cN$ satisfies  property \eqref{eq:measurable_subset2}. 

We write $\cC_{t_0}=\cC^1_{t_0}\cup\cC^2_{t_0}\cup\cC^3_{t_0}$, where we let
\begin{align*}
\cC^1_{t_0} &:= \cC_{t_0} \cap\{\omega\in\cZ^{0,\star}:\,\tau_{\Delta}>t_0\},\\	
\cC^2_{t_0} &:= \cC_{t_0} \cap\{\omega\in\cZ^{0,\star}:\,\tau_{\Delta}<t_0\},\\	
\cC^3_{t_0} &:= \cC_{t_0} \cap\{\omega\in\cZ^{0,\star}:\,\tau_{\Delta}=t_0\},
\end{align*}
and our goal is to prove that
\begin{align}
\label{eq:measure_cC1}
\cC^1_{t_0} &=  \left\{\omega\in\cZ^{0,\star}:\,\tau_{\Delta}>t_0\right\},\\	
\label{eq:measure_cC2}
\cC^2_{t_0}  &= \left\{\omega\in\cZ^{0,\star}:\,\tau_{\Delta}<t_0\right\},\\	
\label{eq:measure_cC3}
\cC^3_{t_0} &= \left\{\omega\in\cZ^{0,\star}:\,\tau_{\Delta}=t_0\right\}\backslash \cN_{t_0},	
\end{align}
where $\cN_{t_0}\in \cF^0$ is a set of $\PP^{0,z}$-measure zero. The preceding identities yield \eqref{eq:measurable_set_t_0}.

\setcounter{step}{0}
\begin{step}[Proof of identity \eqref{eq:measure_cC1}]
From Lemma \ref{lem:Continuity_paths_stopped}, it immediately follows that $\{\tau_{\Delta}>t_0\}\subseteq\cC_{t_0}$, and so, identity \eqref{eq:measure_cC1} is obvious.
\end{step}

\begin{step}[Proof of identity \eqref{eq:measure_cC2}]
If $\omega\in \{\tau_{\Delta}<t_0\}$, then there is $t'_0\in (\tau_{\Delta}, t_0)\cap \DD_+$, such that $Z^0(t'_0,\omega)=\Delta$.
Using the Markov property of the coordinate process established in Lemma \ref{lem:Process_dyadic_times}, together with identity \eqref{eq:Extended_measure_Delta}, we obtain that $Z^0(t,\omega)=\Delta$, for all $t \in (t'_0,\infty)$, and so, we have that $\omega \in \cC_{t_0}$. Therefore, it follows that $\{\tau_{\Delta}<t_0\} \subseteq \cC_{t_0}$, from which identity \eqref{eq:measure_cC2} is follows.
\end{step}

\begin{step}[Proof of identity \eqref{eq:measure_cC3}]
\label{step:measure_cC3}
We may assume, without loss of generality, that 
\begin{align}
\label{eq:Nontrivial_set_t_0}
\PP^{0,z} (\tau_{\Delta}=t_0) \neq 0,
\end{align}
otherwise we choose $\cN_{t_0}:=\{\tau_{\Delta}=t_0\}$ in \eqref{eq:measure_cC3}, which is a $\cF^0$-measurable set by Remark \ref{rmk:tau_Delta_stopping_time}. From the Markov property of the coordinate process and identity \eqref{eq:Extended_measure_Delta}, it is clear that $Z^0(t,\omega)=\Delta$, for all $t >t_0$ such that $t\in\DD_+$, and all $\omega\in\cC^3_{t_0}$, and so, the path $Z^0(\cdot,\omega)$ is continuous from the right at $t_0$. We only need to show that the path $Z^0(\cdot,\omega)$ is continuous from the left at $t_0$. We show that there is a set, $\cN_{t_0}\in\cF^0$, of $\PP^{0,z}$-measure zero, such that for all compact sets $K\subset\underline\Omega$, and all $\omega \in \cC^3_{t_0}\backslash \cN_{t_0}$, there is a positive constant, $\delta=\delta(K,\omega)$, such that
\begin{align}
\label{eq:Convergence_process_both_top}
Z^0(t,\omega) \in \underline\Omega\backslash K,\quad\forall\, t\in\DD_+ \hbox{ such that } t_0-\delta<t<t_0,
\end{align}
which is equivalent to the fact that
$$
\lim_{\stackrel{t\nearrow t_0}{t \in\DD_+}} Z^0(t,\omega) = \Delta,\quad\forall\, \omega \in \cC^3_{t_0}\backslash \cN_{t_0},
$$
from which identity \eqref{eq:measure_cC3} follows.

Let $\{K_k\}_{k\in \NN}$ be an increasing sequence of compact sets in $\underline\Omega$, such that $\underline\Omega=\cup_{k\in\NN} K_k$. Let $\{f_k\}_{k\in \NN}$ be a sequence of smooth functions, compactly supported in $\underline\Omega$, with values in $[0,1]$, such that $f_k(z)=1$, for all $z \in K_k$. By \cite[Corollary 4.1]{Epstein_Mazzeo_cont_est}, we know that
\begin{align}
\label{eq:Convergence_semigroup}
\lim_{t\nearrow t_0} \left(T^{\Omega}_t f_k\right)(z) - \left(T^{\Omega}_{t_0} f_k\right)(z)=0,\quad\forall\, z\in \underline\Omega.
\end{align}
From identity \eqref{eq:Relationship_with_semigroup}, we have that, for all $t \in \DD_+$ such that $t <t_0$,
\begin{align*}
&\left(T^{\Omega}_t f_k\right)(z) - \left(T^{\Omega}_{t_0} f_k\right)(z)
= \EE_{\PP^{0,z}} \left[f_k(Z^0(t))\mathbf{1}_{\{t<\tau_{\Delta}\}}\right] - \EE_{\PP^{0,z}} \left[f_k(Z^0(t_0))\mathbf{1}_{\{t_0<\tau_{\Delta}\}}\right]\\
&\qquad= \EE_{\PP^{0,z}} \left[\left(f_k(Z^0(t))-f_k(Z^0(t_0))\right)\mathbf{1}_{\{t_0<\tau_{\Delta}\}}\right]
+ \EE_{\PP^{0,z}} \left[f_k(Z^0(t))\mathbf{1}_{\{t<\tau_{\Delta}\leq t_0\}}\right].
\end{align*}
Using the continuity of the paths of the process $\{Z^0(t)\}_{t\geq 0}$ established in Lemma \ref{lem:Continuity_paths_stopped}, and the Dominated Convergence Theorem, we conclude that
$$
\lim_{\stackrel{t\nearrow t_0}{t \in\DD_+}}\EE\left[\left(f_k(Z^0(t))-f_k(Z^0(t_0))\right)\mathbf{1}_{\{t_0<\tau_{\Delta}\}}\right] =0,\quad\forall\, k\in\NN.
$$
The preceding identity together with \eqref{eq:Convergence_semigroup} give us that
\begin{align}
\label{eq:Convergence_process_K}
\lim_{\stackrel{t\nearrow t_0}{t \in\DD_+}} \EE\left[f_k(Z^0(t))\mathbf{1}_{\{t<\tau_{\Delta}\leq t_0\}}\right] = 0,\quad\forall\, k\in\NN.
\end{align}
We denote by 
$$
\cN_{t_0}^k:=\left\{\lim_{\stackrel{t\nearrow \tau_{\Delta}}{t \in\DD_+}} Z^0(t) \in K_k,\hbox{ and } \tau_{\Delta}=t_0\right\},\quad\forall\, k \in\NN,
$$ 
and again using the fact that $\DD_+$ is a countable set, we easily obtain that $\cN_{t_0}^k$ belongs to $\cF^0$. Since we assume that \eqref{eq:Nontrivial_set_t_0} holds, and we have that 
$$
\{t<\tau_{\Delta}\leq t_0\} \supset \{\tau_{\Delta}=t_0\},\quad\forall\, t\in \DD_+ \hbox{ such that } t<t_0,
$$
it follows from identities \eqref{eq:Continuity_paths_stopped} and \eqref{eq:Convergence_process_K}, and the fact that $f_k\equiv 1$ on $K_k$,  that 
$$
\PP^{0,z}\left(\cN_{t_0}^k \right) = 0, \quad\forall\, k\in\NN.
$$
We let $\cN_{t_0}:=\cup_{k\in\NN}\cN^k_{t_0}$. Then it is clear that the set $\cN_{t_0}\in \cF^0$ and $\PP^{0,z}(\cN_{t_0})=0$, and that, for all $\omega \in \{\tau_{\Delta}=t_0\}\backslash \cN_{t_0}$, property \eqref{eq:Convergence_process_both_top} holds. This implies identity \eqref{eq:measure_cC3}, and completes the proof of Step \ref{step:measure_cC3}.
\end{step}
Combining the preceding three steps we obtain identity \eqref{eq:measurable_set_t_0}. By letting $\cN:=\cup_{t_0\in\DD_+} \cN_{t_0}$, we obtain that $\cN\in\cF^0$ and $\PP^{0,z}(\cN)=0$, and the conclusion follows by letting $\cZ^{0,\star\star}:=\cC\backslash\cN$.
\end{proof}

We now extend the process $Z^0:\DD_+\times\cZ^{0,\star\star}\rightarrow\Omega_{\Delta}$ to a progressively measurable process, $Z^{0,\Omega}:[0,\infty)\times\cZ^{0,\star\star}\rightarrow\bar\Omega$. For this purpose, we first need to introduce a suitable $\sigma$-algebra on the range of the process, $\bar \Omega$. We let $\cG$ be the $\sigma$-algebra generated by the collection of Borel sets on $\underline\Omega$, $\cB(\underline\Omega)$, and $\overline{\partial_1\Omega}$, that is, we set
\begin{align}
\label{eq:sigma_algebra_range}
\cG:=\{B, B\cup\overline{\partial_1\Omega}:\, B\in \cB(\underline\Omega)\}.
\end{align} 

We can now give the proof of 

\begin{proof}[Proof of Proposition \ref{prop:Stopped_process}]
We let $\cZ:=\cZ^{0,\star\star}$, where $\cZ^{0,\star\star}$ is the
$\cF^0$-measurable set introduced in Lemma
\ref{lem:Continuity_paths_compactification}. We denote for simplicity
$\PP^z:=\PP^{0,z}$, where $\PP^{0,z}$ is the probability measure introduced in
Lemma \ref{lem:Continuity_paths_compactification}. Let $\cF:=\cF^0$ and
$\cF_t:=\cF^0_t$, for all $t\geq 0$. We construct the process
$\{Z^{0,\Omega}_t\}_{t\geq 0}$ in Steps \ref{step:Cont_process_1} and
\ref{step:Cont_process_2}. We first replace the  range $\Omega_{\Delta}$ of the
coordinate process $\{Z^0(t)\}_{t\in\DD_+}$ with $\bar\Omega$, and then we
extend the process from $\DD_+$ to $[0,\infty)$ to obtain
$\{Z^{0,\Omega}(t)\}_{t\geq 0}$. In Steps \ref{step:Connection_semigroup} and
\ref{step:Process_after_stopping_time}, we show that identities
\eqref{eq:Connection_semigroup} and \eqref{eq:Process_after_stopping_time}
hold, while in Step \ref{step_Markov_property}, we prove that the process
$\{Z^{0,\Omega}(t)\}_{t\geq 0}$ satisfies the Markov property. The fact that
the random time $\tau_{\Omega}$ defined by \eqref{eq:tau_Omega} is a stopping
time with respect to the filtration $\{\cF_t\}_{t\geq 0}$ is proved in Case
\ref{case:Boundary_set} of Step~\ref{step:Cont_process_2} below.

\setcounter{step}{0}
\begin{step}[Construction of a process with values in $\bar\Omega$]
\label{step:Cont_process_1}
We define the process $\widetilde Z^0:\DD_+\times\cZ\rightarrow \bar\Omega$ by 
\begin{align*}
\widetilde Z^{0}(t,\omega) = \left\{
\begin{array}{ll}
Z^{0}(t,\omega),&\quad\hbox{if } t <\tau_{\Delta},\\
\lim_{\stackrel{s\nearrow \tau_{\Delta}}{s \in\DD_+}} Z^{0}(s,\omega), &\quad\hbox{if } t \geq\tau_{\Delta},
\end{array} \quad\quad\forall t\in \DD_+,\right.	
\end{align*}
where the preceding limit is taken with respect to the Euclidean topology on $\bar\Omega$. By Lemma \ref{lem:Continuity_paths_stopped}, it follows that the process $\{\widetilde Z^0(t)\}_{t\geq 0}$ is well-defined, and has continuous paths supported in $\bar\Omega$, for all $\omega\in\cZ$. We next show that the process $\{\widetilde Z^0(t)\}_{t\in\DD_+}$ is adapted to the filtration $\{\cF_t\}_{t\in\DD_+}$, that is, for all $t\in\DD_+$, the random variable $\widetilde Z^0(t)$ is $\cF_t/\cG$-measurable. For all measurable sets $B\in \cB(\underline\Omega)$, we have by construction that $\{\widetilde Z^0(t)\in B\} = \{Z^0(t)\in B\}$, and so, clearly the set $\{\widetilde Z^0(t)\in B\}$ is contained in $\cF_t$. We also have that $\{\widetilde Z^0(t)\in \overline{\partial_1\Omega}\}=\{\tau_{\Delta} \leq t\}$, which belongs to $\cF_t$ by \eqref{eq:Measurable_set_tau_Delta}. We conclude that the random variables $\widetilde Z^0(t)$ are $\cF_t/\cG$-measurable, for all $t \in\DD_+$. 
\end{step}

\begin{step}[Construction of a process with values in $\bar\Omega$, for all $t\in [0,\infty)$]
\label{step:Cont_process_2}
We extend the process $\{\widetilde Z^0(t)\}_{t\in\DD_+}$ from $\DD_+$ to $[0,\infty)$ by continuity, that is, we let
\begin{equation}
\label{eq:Z_Omega_prim}
\begin{aligned}
Z^{0,\Omega}(t,\omega) = \left\{
\begin{array}{ll}
\widetilde Z^{0}(t,\omega),&\quad\hbox{if } t\in \DD_+,\\
\lim_{\stackrel{s\rightarrow t}{s \in\DD_+}} \widetilde Z^{0}(s,\omega), &\quad\hbox{if } t\notin\DD_+.
\end{array} \right.	
\end{aligned}
\end{equation}
From the preceding step, we see that the process $\{Z^{0,\Omega}(t)\}_{t\geq 0}$ has continuous paths supported in $\bar\Omega$, for all $t\geq 0$, and the random variables $Z^{0,\Omega}(t)$ are $\cF_t/\cG$-measurable, for all $t \in\DD_+$. We next show that the random variables $Z^{0,\Omega}(t)$ are $\cF_t/\cG$-measurable, for all $t \geq 0$. For this purpose, it is sufficient to prove that the sets 
$$
\left\{Z^{0,\Omega}(t)\in\overline{\partial_1\Omega}\right\}\quad\hbox{and}\quad \left\{Z^{0,\Omega}(t)\in B\right\},
$$
belong to $\cF_t$, where $B\subset\underline\Omega$ is a compact set, and $t\in [0,\infty)\backslash\DD_+$. We consider the following cases:

\setcounter{case}{0}
\begin{case}[Proof that $\{Z^{0,\Omega}(t)\in\overline{\partial_1\Omega}\}\in \cF_t$]
\label{case:Boundary_set}
From Lemma \ref{lem:Continuity_paths_compactification}, we know that the paths of the process $\{Z^0(t)\}_{t\in\DD_+}$ are continuous in the topology of the one-point compactification, and so,
$$
\lim_{\stackrel{t\nearrow\tau_{\Delta}}{t\in\DD_+}} Z^0(t,\omega)=\Delta,\quad\forall\,\omega\in\cZ,
$$
which implies that, for all compact sets $K\subset\underline\Omega$, and all $\omega\in \cZ$, there is a positive constant, $\delta=\delta(K,\omega)$, such that $Z^0(t,\omega)\in\underline\Omega\backslash K$, for all $t \in \DD_+$ such that $\tau_{\Delta}(\omega)-\delta<t<\tau_{\Delta}(\omega)$. Thus, Lemma \ref{lem:Continuity_paths_stopped} implies that
$$
\lim_{\stackrel{t\nearrow \tau_{\Delta}}{t \in\DD_+}} Z^0(t) \in \overline{\partial_1\Omega},\quad\forall\,\omega\in \cZ, 
$$ 
where the preceding limit is understood in the Euclidean topology of $\bar\Omega$. From \eqref{eq:Z_Omega_prim}, we see that $Z^{0,\Omega}(t)=\widetilde Z^0(t,\omega)$, for all $t \geq \tau_{\Delta}$ such that $t\in\DD_+$, and using also definition \eqref{eq:tau_Omega} of the random time $\tau_{\Omega}$, we conclude that 
\begin{align}
\label{eq:Stopping_times_are_the_same}
\tau_{\Delta}(\omega)=\tau_{\Omega}(\omega),\quad\forall\,\omega\in\cZ.
\end{align}
From Remark \ref{rmk:tau_Delta_stopping_time}, it follows that $\tau_{\Omega}$ is a stopping time with respect to the filtration $\{\cF_t\}_{t\geq 0}$. We also have that
$$
\{Z^{0,\Omega}(t) \in\overline{\partial_1\Omega}\} = \{\tau_{\Omega} \leq t\},
$$
and so, it follows that the preceding set is $\cF_t$-measurable. This concludes the proof of Case \ref{case:Boundary_set}.
\end{case} 

\begin{case}[Proof that $\{Z^{0,\Omega}(t)\in B\}\in \cF_t$, where $B\subset\underline\Omega$ is a compact set]
\label{case:Closed_sets}
Let now $B\subset\underline\Omega$ be a compact set, and let $k_0\in\NN$ be large enough, such that the relatively open sets $B^k:=\cup_{z\in B} B^e_{1/k}(z)$ have the property that  $B \subset B^k \Subset\underline\Omega$, for all $k\geq k_0$. Using the continuity of the paths of the process $\{Z^{0,\Omega}(t)\}_{t\geq 0}$, we obtain that
$$
\left\{Z^{0,\Omega}(t) \in B\right\} = \bigcap_{k\geq k_0} \bigcup_{l=1}^{\infty} \bigcap_{\stackrel{t<s<t+1/l}{s\in\DD_+}}\{Z^{0,\Omega}(s) \in B^k\}.
$$
Because the filtration $\{\cF_t\}_{t\geq 0}$ is right-continuous, and the sets $\{Z^{0,\Omega}(s) \in B^k\}$ are $\cF_s$-measurable, for all $s\in\DD_+$, it follows that the set $\{Z^{0,\Omega}(t) \in B\}$ is also $\cF_t$-measurable. This concludes the proof of Case \ref{case:Closed_sets}.
\end{case} 

Combining Cases \ref{case:Boundary_set} and \ref{case:Closed_sets}, it follows that the random variables $Z^{0,\Omega}(t)$ are $\cF_t/\cG$-measurable, for all $t \geq 0$. Because the paths of the process are continuous, \cite[Proposition 1.1.13]{KaratzasShreve1991} implies that the process $\{Z^{0,\Omega}(t)\}_{t\geq 0}$ is progressively measurable.
\end{step}

\begin{step}[Proof of identity \eqref{eq:Connection_semigroup}]
\label{step:Connection_semigroup}
From Steps \ref{step:Cont_process_1} and \ref{step:Cont_process_2}, together with identity \eqref{eq:Stopping_times_are_the_same}, we have that $Z^0(t,\omega) =Z^{0,\Omega}(t,\omega)$ on $t<\tau_{\Delta}$, for all $t\in\DD_+$ and $\omega\in\cZ$. Therefore, identity \eqref{eq:Relationship_with_semigroup} implies that \eqref{eq:Connection_semigroup} holds, for all $t \in\DD_+$. It remains to prove that \eqref{eq:Connection_semigroup} holds at all $t \in[0,\infty)\backslash\DD_+$.

We first assume that $f:\underline\Omega\rightarrow\RR$ is a compactly supported, continuous function on $\underline\Omega$. Using the continuity of the paths of the process $\{Z^{0,\Omega}(t)\}_{t \geq 0}$, we see that $f(Z^{0,\Omega}(s))$ converges to $f(Z^{0,\Omega}(t))$, for all $s\in\DD_+$, $s>t$, as $s$ tends to $t$. Also we see that $\mathbf{1}_{\{s<\tau_{\Omega}\}}$ converges to $\mathbf{1}_{\{t<\tau_{\Omega}\}}$, for all $s\in\DD_+$, $s>t$, as $s$ tends to $t$. The Dominated Convergence Theorem gives us that
$$
\lim_{\stackrel{s\downarrow t}{s\in\DD_+}} \EE_{\PP^z}\left[f(Z^{0,\Omega}(s)) \mathbf{1}_{\{s<\tau_{\Omega}\}}\right] 
=  \EE_{\PP^z}\left[f(Z^{0,\Omega}(t)) \mathbf{1}_{\{t<\tau_{\Omega}\}}\right].
$$
From \cite[Corollary 4.1]{Epstein_Mazzeo_cont_est}, we have that $(T^{\Omega}_s f)(z)$ converges to $(T^{\Omega}_t f)(z)$, for all $s\in\DD_+$, $s>t$, as $s$ approaches $t$, and so, we obtain that identity \eqref{eq:Connection_semigroup} holds at all $t\in [0,\infty)$, when $f\in C(\bar \Omega)$. 

Let now $f:\underline\Omega\rightarrow\RR$ be a bounded, Borel measurable function, and let $\{f_k\}_{k\in\NN}$ be a sequence of compactly supported, continuous functions on $\underline\Omega$ that converge a.e. on $\underline\Omega$ to $f$. It follows from \cite[Problem 2.2.4]{KaratzasShreve1991} that
\begin{equation}
\label{eq:Convergence_f_k_process}
f_k\left(Z^{0,\Omega}(t)\right)\mathbf{1}_{\{t<\tau_{\Omega}\}} \rightarrow f\left(Z^{0,\Omega}(t)\right)\mathbf{1}_{\{t<\tau_{\Omega}
\}}\quad \PP^z\hbox{-a.s.},\quad\hbox{as } k\rightarrow\infty,
\end{equation}
and the Dominated Convergence Theorem gives us that
$$
\EE_{\PP^z}\left[f_k\left(Z^{0,\Omega}(t)\right)\mathbf{1}_{\{t<\tau_{\Omega}\}}\right] \rightarrow \EE_{\PP^z}\left[f\left(Z^{0,\Omega}(t)\right)\mathbf{1}_{\{t<\tau_{\Omega}\}}\right], \quad\hbox{as } k\rightarrow\infty.
$$
Because identity \eqref{eq:Connection_semigroup} holds with $f$ replaced by $f_k$, since $f_k\in C(\bar \Omega)$, and $(T^{\Omega}_t f_k)(z)\rightarrow (T^{\Omega}_t f)(z)$, as $k\rightarrow\infty$, for all $z\in\underline\Omega$, the preceding convergence property yields that identity \eqref{eq:Connection_semigroup} holds with $f:\underline\Omega\rightarrow\RR$ a bounded, Borel measurable function. 
\end{step}

\begin{step}[Proof of identity \eqref{eq:Process_after_stopping_time}]
\label{step:Process_after_stopping_time}
Identity \eqref{eq:Process_after_stopping_time} follows from the construction of the process $\{Z^{0,\Omega}(t)\}_{t\geq 0}$ in Steps \ref{step:Cont_process_1} and \ref{step:Cont_process_2}, and property \eqref{eq:Stopping_times_are_the_same} of the stopping time $\tau_{\Omega}$.
\end{step}

\begin{step}[The Markov property]
\label{step_Markov_property}
We now prove that the process $\{Z^{0,\Omega}(t)\}_{t\geq 0}$ satisfies the Markov property, \eqref{eq:Markov_prop_stopped_process}. Recall from the proof of Step \ref{step:Connection_semigroup}, that given any bounded, Borel measurable function, $f:\underline\Omega\rightarrow\RR$, there is a sequence $\{f_k\}_{k\in\NN}$ of compactly supported, continuous functions, such that property \eqref{eq:Convergence_f_k_process} holds. Thus, using the Dominated Convergence Theorem, it is sufficient to show that the Markov property \eqref{eq:Markov_prop_stopped_process} holds for all functions $f:\underline\Omega\rightarrow\RR$ that are compactly supported and continuous. Let $s<u<t$ be such that $s\in\bar\RR_+$ and $u,t\in\DD_+$. Using the tower property and the fact that $\cF_s\subset\cF'_u$, we have that
\begin{equation}
\label{eq:Markov_prop_tower}
\EE_{\PP^z}\left[f(Z^{0,\Omega}(t))\big{|} \cF_s\right] = \EE_{\PP^z}\left[\EE_{\PP^z}\left[f(Z^{0,\Omega}(t))\big{|} \cF'_u\right]\big{|} \cF_s\right],
\end{equation}
and, using property \eqref{eq:Markov_prop_dyadic_times} applied with $s=u$, and our construction of the process $\{Z^{0,\Omega}(t)\}_{t\geq 0}$ from the process $\{Z^0(t)\}_{t\in\DD_+}$, we have that
\begin{equation}
\label{eq:Markov_prop_cond_expe_semigroup}
\EE_{\PP^z}\left[f(Z^{0,\Omega}(t))\big{|} \cF'_u\right] = \left(T^{\Omega}_{t-u}f\right)(Z^{0,\Omega}(u)).
\end{equation}
From \cite[Corollary 4.1]{Epstein_Mazzeo_cont_est}, we have that $(T^{\Omega}_{t-u} f)(z)$ converges to $(T^{\Omega}_{t-s} f)(z)$, as $u$ approaches $s$, for all $z\in\underline\Omega$, and using the continuity of the paths of the process $\{Z^{0,\Omega}(t)\}_{t\geq 0}$, it follows from \eqref{eq:Markov_prop_cond_expe_semigroup} that
$$
\lim_{\stackrel{u\downarrow s}{u\in\DD_+}} \EE_{\PP^z}\left[f(Z^{0,\Omega}(t))\big{|} \cF'_u\right] = \left(T^{\Omega}_{t-s}f\right)(Z^{0,\Omega}(s)).
$$
The preceding property together with identities \eqref{eq:Markov_prop_tower} and \eqref{eq:Markov_prop_cond_expe_semigroup}, gives us by taking limit as $u\in\DD_+$ tends to $s$ in \eqref{eq:Markov_prop_tower}, that
$$
\EE_{\PP^z}\left[f(Z^{0,\Omega}(t))\big{|} \cF_s\right] = \left(T^{\Omega}_{t-s}f\right)(Z^{0,\Omega}(s)),
$$
and so, identity \eqref{eq:Markov_prop_stopped_process} clearly holds, for all $t\in\DD_+$ and $s\in\bar\RR_+$. It remains to establish identity \eqref{eq:Markov_prop_stopped_process} for all $s,t\in\bar\RR_+$. This follows easily from the fact that the Markov property holds for all $t\in\DD_+$ and $s\in\bar\RR_+$, by choosing a sequence $\{t_k\}_{k\in\NN}\subset\DD_+$ converging to $t$, using \eqref{eq:Markov_prop_tower} with $t_k$ in place of $t$, and taking limit as $k\rightarrow\infty$. This completes the proof of Step \ref{step_Markov_property}.
\end{step}

This concludes the proof of Proposition \ref{prop:Stopped_process}.
\end{proof}

\section{Solution to the martingale problem}
\label{sec:Martingale_problem}
In this section, we introduce the martingale problem associated to the operator $L$ in Definition \ref{def:Martingale_problem}, and we prove existence of solutions in Proposition \ref{prop:Solution_local_martingale_problem}, with the aid of Proposition \ref{prop:Stopped_process} and of the $L^q(\Omega;d\mu)$-distribution estimates of the fundamental solutions of the operator $L$, established in Lemma \ref{lem:L_q_transition_density}.

Let $\cX^{\Omega}:=C([0,\infty);\bar \Omega)$ be the space of continuous functions, $\omega:[0,\infty)\rightarrow\bar\Omega$, where we endow $\bar\Omega$ with the $\sigma$-algebra $\cG$ defined in \eqref{eq:sigma_algebra_range}. For all $t\geq 0$, let $\cB^{\Omega}_t$ be the $\sigma$-algebra on $\cX^{\Omega}$ generated by the cylinder sets, 
\begin{equation}
\label{eq:Cylinder_set}
C:=\left\{\omega\in \cX^{\Omega}: \, \omega(t_i)\in B_i,\quad\forall\, i=1,\ldots,k\right\},
\end{equation}
where $k\in \NN$, and $0\leq t_1<\ldots<t_k\leq t$, $B_i\in\cG$, for all $i=1,\ldots,k$. Then $\{\cB^{\Omega}_t\}_{t\geq 0}$ is a filtration of $\sigma$-algebras. We let $\cB^{\Omega}:=\cup_{t\geq 0} \cB^{\Omega}_t$. Let ${\bf e}^{\Omega}$ be the stopping time defined by
\begin{align}
\label{eq:exit_time}
{\bf e}^{\Omega}(\omega):=\inf\{t\geq 0:\, \omega(t) \in\overline{\partial_1\Omega}\},\quad\forall\, \omega\in\cX^{\Omega}.
\end{align}

We can now define
\begin{defn}[The local martingale problem associated to the operator $L$ on the domain $\Omega$]
\label{def:Martingale_problem}
Let $z\in\underline\Omega$. We say that a probability measure, $\QQ^{z,\Omega}$, on the filtered probability space $(\cX^{\Omega}, \{\cB^{\Omega}_t\}_{t\geq 0},\cB^{\Omega})$ is a solution to the martingale problem associated to the operator $L$ on the domain $\Omega$, if for all functions $\varphi\in C^{\infty}_c(\underline{\Omega})$, the process defined by
\begin{equation}
\label{eq:Martingale}
M^{\varphi}(t,\omega):=\varphi(\omega(t\wedge {\bf e}^{\Omega}))-\varphi(\omega(0)) - \int_0^{t\wedge {\bf e}^{\Omega}} L\varphi(\omega(r))\, dr,\quad\forall\, t \geq 0,\quad\forall\, \omega\in \cX^{\Omega},
\end{equation}
is a $\QQ^{z,\Omega}$-martingale with respect to the filtration $\{\cB^{\Omega}_t\}_{t \geq 0}$, and
\begin{equation}
\label{eq:Initial_cond}
\QQ^{z,\Omega}(\omega(0)=z) = 1.
\end{equation}
\end{defn}

\begin{rmk}[Stopped martingale problem]
See \cite[\S 2]{Athreya_Barlow_Bass_Perkins_2002} for the definition of a stopped martingale problem in the spirit of our Definition \ref{def:Martingale_problem}.
\end{rmk}

The main result of this section is the existence of solutions the the local martingale problem associated to the singular Kimura operator $L$ on the domain $\Omega$:
\begin{prop}[Solutions to the local martingale problem]
\label{prop:Solution_local_martingale_problem}
Let $z\in\underline{\Omega}$. Then there is a solution, $\QQ^{z,\Omega}$, to
the local martingale problem associated to the operator $L$ on the domain $\Omega$. 
\end{prop}

\begin{rmk}[Application of Proposition \ref{prop:Solution_local_martingale_problem}]
We apply the existence of solutions to the local martingale problem associated to $L$ on the domain $\Omega$, to build weak solutions to the singular Kimura stochastic differential equation \eqref{eq:Kimura_SDE_singular}. The weak solutions to \eqref{eq:Kimura_SDE_singular} that we build in Proposition \ref{eq:Weak_solution_semigroup} are associated to the semigroup $\{T^{\Omega}_t\}_{t\geq 0}$ through identity \eqref{prop:Stopped_solutions}. This process is stopped when it hits the portion $\partial_1\Omega$ of the boundary of the domain $\Omega$. In Theorem \ref{thm:Connection_semigroup_Markov_solution}, we improve this result in that we no longer require the process to be stopped on $\partial_1\Omega$.
\end{rmk}

\begin{proof}[Proof of Proposition \ref{prop:Solution_local_martingale_problem}]
Let $\QQ^{z,\Omega}$ be the probability measure induced on the space $\cX^{\Omega}$ by the probability measure $\PP^z$ constructed in Proposition \ref{prop:Stopped_process}. That is, for all $k\in \NN$, $0\leq t_1<\ldots<t_k<\infty$, $B_i\in\cG$, for all $i=1,\ldots,k$, we set
$$
\QQ^{z,\Omega}(\omega(t_i)\in B_i,\, \forall\, i=1,\ldots, k) = \PP^z\left(Z^{0,\Omega}(t_i)\in B_i,\, \forall\, i=1,\ldots, k\right).
$$
It is clear that $\QQ^{z,\Omega}$ satisfies the initial condition
\eqref{eq:Initial_cond} because $\PP^z(Z^{0,\Omega}(0)=z)=1$. It remains to
show that the processes defined in \eqref{eq:Martingale} are $\QQ^{z,\Omega}$-martingales. We split the proof into two steps. We let $\varphi\in C^{\infty}_c(\underline{\Omega})$, and we see that $L\varphi$ is an unbounded function on $\Omega$. For this reason, we show in the first step that
\begin{equation}
\label{eq:Expectation_integral_term}
\EE_{\QQ^{z,\Omega}}\left[\int_0^{t\wedge{\bf e}^{\Omega}}|L\varphi(\omega(r))|\, dr\right]<\infty,\quad\forall\, t \geq 0.
\end{equation}
In the second step, we prove that the process $\{M^{\varphi}(t)\}_{t\geq 0}$ defined by \eqref{eq:Martingale} is indeed a $\QQ^{z,\Omega}$-martingale with respect to the filtration $\{\cB^{\Omega}_t\}_{t\geq 0}$.

\setcounter{step}{0}
\begin{step}[Proof of \eqref{eq:Expectation_integral_term}]
\label{step:Expectation_integral_term}
From the expression \eqref{eq:Singular_operator} of the differential operator $L$, and the fact that $\varphi\in C^{\infty}_c(\underline\Omega)$, we see that $L\varphi\in L^p(\Omega;d\mu)$, for all $p\geq 1$. Let $q_0\in (1,2)$ be the constant appearing in the conclusion of Lemma \ref{lem:L_q_transition_density}, and choose $q\in (1,q_0)$. Let $p\in (1,\infty)$ be the conjugate exponent of $q$. 

Using the fact that  
$$
\int_0^{t\wedge{\bf e}^{\Omega}}L\varphi(\omega(r))\, dr = \int_0^t L\varphi(\omega(r))
\mathbf{1}_{\{r<{\bf e}^{\Omega}\}}\, dr,
$$
and property \eqref{eq:Connection_semigroup}, we have that
\begin{align*}
\EE_{\QQ^{z,\Omega}}\left[\int_0^{t\wedge{\bf e}^{\Omega}}L\varphi(\omega(r))\, dr\right]
&=\EE_{\QQ^{z,\Omega}}\left[\int_0^t L\varphi(\omega(r)) \mathbf{1}_{\{r<{\bf e}^{\Omega}\}}\, dr\right]\\
&= \int_0^t \left(T^{\Omega}_r L\varphi\right)(z)\, dr\\
&= \int_0^t\int_{\underline\Omega}L\varphi(w)p^{\Omega}(r,z,w)\, d\mu(w) \, dr,
\end{align*}
and by H\"older's inequality, we obtain
\begin{align*}
\EE_{\QQ^{z,\Omega}}\left[\int_0^{t\wedge{\bf e}^{\Omega}}|L\varphi(\omega(r))|\, dr\right]
&\leq \int_0^t \|L\varphi\|_{L^p(\Omega;d\mu)} \|p^{\Omega}(r,z,\cdot)\|_{L^q(\Omega;d\mu)} \, dr.
\end{align*}
Inequality \eqref{eq:L_q_transition_density} shows that there is a positive constant, $C_0$, such that
\begin{align*}
\EE_{\QQ^{z,\Omega}}\left[\int_0^{t\wedge{\bf e}^{\Omega}}|L\varphi(\omega(r))|\, dr\right]
&\leq C_0 \|L\varphi\|_{L^p(\Omega;d\mu)}.
\end{align*}
This concludes the proof of inequality \eqref{eq:Expectation_integral_term}.
\end{step}

\begin{step}[The martingale property]
\label{step:Martingale_property}
We now prove that the process $\{M^{\varphi}(t)\}_{t\geq 0}$ defined by \eqref{eq:Martingale} is indeed a $\QQ^{z,\Omega}$-martingale with respect to the filtration $\{\cB^{\Omega}_t\}_{t\geq 0}$, that is, we show that
$$
\EE_{\QQ^{z,\Omega}}\left[M^{\varphi}(t)\Big{|}\cB_s\right] = M^{\varphi}(s),\quad \QQ^{z,\Omega}\hbox{-a.s.},\quad\forall\, 0\leq s\leq t.
$$
The preceding identity is equivalent to proving that, for all $0\leq s\leq t$, we have 
\begin{equation}
\label{eq:Martingale_reformulation}
\EE_{\QQ^{z,\Omega}}\left[\varphi(\omega(t \wedge {\bf e}^{\Omega})) \Big{|} \cB^{\Omega}_s\right] = \varphi(\omega(s \wedge {\bf e}^{\Omega})) + \EE_{\QQ^{z,\Omega}}\left[\int_s^t L\varphi(\omega(r)) \mathbf{1}_{\{r < {\bf e}^{\Omega}\}} \, dr \Big{|} \cB^{\Omega}_s\right],\quad \QQ^{z,\Omega}\hbox{-a.s.}
\end{equation}
We need the following
\begin{claim}
\label{claim:Martingale_conditional_expectation}
For all $0\leq s \leq t$, we have that
\begin{align}
\label{eq:Martingale_conditional_expectation}
\EE_{\QQ^{z,\Omega}}\left[\varphi(\omega(t\wedge {\bf e}^{\Omega})) \Big{|} \cB_s\right] 
&= \left(T^{\Omega}_{t-s}\varphi\right)(\omega(s))\mathbf{1}_{\{s<{\bf e}^{\Omega}\}},\quad \QQ^{z,\Omega}\hbox{-a.s.}
\end{align}
\end{claim}

\begin{proof}[Proof of Claim \ref{claim:Martingale_conditional_expectation}]
Because the right hand-side in identity \eqref{eq:Martingale_conditional_expectation} is $\cB^{\Omega}_s$-measurable, it is sufficient to show that, for all $k\in \NN$, $0\leq s_1<\ldots<s_k\leq s$, and all $\psi_i\in C_c(\underline\Omega)$, for all $i=1,\ldots,k$, we have that
\begin{align}
\label{eq:Conditional_expectation_interp}
\EE_{\QQ^{z,\Omega}}\left[\varphi(\omega(t\wedge {\bf e}^{\Omega})) \prod_{i=1}^k \psi_i(\omega(s_i))\right] 
&= \EE_{\QQ^{z,\Omega}}\left[ \left(T^{\Omega}_{t-s}\varphi\right)(\omega(s)) \mathbf{1}_{\{s<{\bf e}^{\Omega}\}} \prod_{i=1}^k \psi_i(\omega(s_i))\right].
\end{align}
Using the construction of the measure $\PP^z$ and the representation of the semigroup $\{T^{\Omega}_t\}_{t\geq 0}$ in terms of the fundamental solution $\{p^{\Omega}(t,\cdot,\cdot)\}_{t\geq 0}$, we obtain that the term on the left hand-side of the preceding identity takes the form
\begin{align*}
&\int_{\underline\Omega} \psi_1(z_1)p^{\Omega}(s_1,z,z_1)\, d\mu(z_1)
\int_{\underline\Omega} \psi_2(z_2)p^{\Omega}(s_2-s_1,z_1,z_2)\, d\mu(z_2)
\ldots\\
&\int_{\underline\Omega} \psi_k(z_k)p^{\Omega}(s_k-s_{k-1},z_{k-1},z_k)\, d\mu(z_k)
\int_{\underline\Omega} \varphi(w)p^{\Omega}(t-s_k,z_k,w)\, d\mu(w),
\end{align*}
while the term on the right-hand side takes the form
\begin{align*}
&\int_{\underline\Omega} \psi_1(z_1)p^{\Omega}(s_1,z,z_1)\, d\mu(z_1)
\int_{\underline\Omega} \psi_2(z_2)p^{\Omega}(s_2-s_1,z_1,z_2)\, d\mu(z_2)
\ldots\\
&\int_{\underline\Omega} \psi_k(z_k)p^{\Omega}(s_k-s_{k-1},z_{k-1},z_k)\, d\mu(z_k)
\int_{\underline\Omega} p^{\Omega}(s-s_k,z_k,z)\, d\mu(z)
\int_{\underline\Omega} \varphi(w)p^{\Omega}(t-s,z,w)\, d\mu(w).
\end{align*}
By the semigroup property, we obtain that
\begin{align*}
\int_{\underline\Omega} p^{\Omega}(s-s_k,z_k,z)\, d\mu(z)
\int_{\underline\Omega} \varphi(w)p^{\Omega}(t-s,z,w)\, d\mu(w)
&= \int_{\underline\Omega} \varphi(w)p^{\Omega}(t-s_k,z_k,w)\, d\mu(w),
\end{align*}
and therefore, identity \eqref{eq:Conditional_expectation_interp} immediately follows. This completes the proof.
\end{proof}
Without loss of generality, we may assume that $s<t$, otherwise the identity \eqref{eq:Martingale_reformulation} is obviously true. Let $\eps\in (0,t-s)$. It follows from \cite[Proposition 1.1.5]{Ethier_Kurtz} that we have the following identity in the $L^2(\Omega;d\mu)$-sense:
$$
T^{\Omega}_{t-s} \varphi -T^{\Omega}_{\eps}\varphi = \int_{\eps}^{t-s} T^{\Omega}_r L\varphi\, dr.
$$
By \cite[Corollary 4.1]{Epstein_Mazzeo_cont_est}, the functions $T^{\Omega}_{t-s} \varphi$, $T^{\Omega}_{\eps}\varphi$ and $T^{\Omega}_r L\varphi$ are bounded and H\"older continuous on $\underline\Omega$, and so, the preceding identity holds pointwise on  $\underline\Omega$, that is, we have 
$$
\left(T^{\Omega}_{t-s} \varphi\right)(w) -\left(T^{\Omega}_{\eps}\varphi\right)(w) 
= \int_{\eps}^{t-s} \left(T^{\Omega}_r L\varphi\right)(w)\, dr,\quad\forall\, w\in\underline\Omega.
$$
We now let $w=\omega(s)$ in the preceding identity, and using \eqref{eq:Martingale_conditional_expectation}, we obtain
\begin{align*}
&\EE_{\QQ^{z,\Omega}}\left[\varphi(\omega(t \wedge {\bf e}^{\Omega})) \Big{|} \cB^{\Omega}_s\right] 
- \EE_{\QQ^{z,\Omega}}\left[\varphi(\omega((s+\eps) \wedge {\bf e}^{\Omega})) \Big{|} \cB^{\Omega}_s\right] \\
&\qquad
= \int_{\eps}^{t-s} \EE_{\QQ^{z,\Omega}}\left[ L\varphi(\omega((s+r) \wedge {\bf e}^{\Omega})) \Big{|} \cB^{\Omega}_s\right]\, dr.
\end{align*}
Because the function $\varphi$ has compact support in $\underline\Omega$, we see that
$$
\varphi(\omega(t \wedge {\bf e}^{\Omega})) = \varphi(\omega(t)) \mathbf{1}_{\{t < {\bf e}^{\Omega}\}},\quad\forall\, t\geq 0,
$$
and so, the preceding equality becomes
\begin{align*}
&\quad\EE_{\QQ^{z,\Omega}}\left[\varphi(\omega(t)) \mathbf{1}_{\{t < {\bf e}^{\Omega}\}} \Big{|} \cB^{\Omega}_s\right] 
- \EE_{\QQ^{z,\Omega}}\left[\varphi(\omega(s+\eps)) \mathbf{1}_{\{s+\eps < {\bf e}^{\Omega}\}} \Big{|} \cB^{\Omega}_s\right] \\
&\quad\quad= \int_{\eps}^{t-s} \EE_{\QQ^{z,\Omega}}\left[ L\varphi(\omega(s+r)) \mathbf{1}_{\{s+r < {\bf e}^{\Omega}\}} \Big{|}\cB^{\Omega}_s\right]\, dr.
\end{align*}
Using property \eqref{eq:Expectation_integral_term}, we may apply Fubini's Theorem on the right-hand side of the preceding equality to obtain
\begin{equation}
\begin{aligned}
\label{eq:Martingale_reformulation_eps}
&\quad\EE_{\QQ^{z,\Omega}}\left[\varphi(\omega(t)) \mathbf{1}_{\{t < {\bf e}^{\Omega}\}} \Big{|} \cB^{\Omega}_s\right] 
- \EE_{\QQ^{z,\Omega}}\left[\varphi(\omega(s+\eps)) \mathbf{1}_{\{s+\eps < {\bf e}^{\Omega}\}} \Big{|} \cB^{\Omega}_s\right] \\
&\quad\quad= \EE_{\QQ^{z,\Omega}}\left[\int_{\eps}^{t-s}  L\varphi(\omega(s+r)) \mathbf{1}_{\{s+r < {\bf e}^{\Omega}\}}\, dr \Big{|}\cB^{\Omega}_s\right].
\end{aligned}
\end{equation}
We now show that by letting $\eps\downarrow 0$ in identity \eqref{eq:Martingale_reformulation_eps}, we obtain \eqref{eq:Martingale_reformulation}. Notice that by the Dominated Convergence Theorem, the fact that the function $\varphi$ is bounded, continuous and has compact support in $\underline\Omega$, and that the paths $t\mapsto \omega(t)$ are continuous, we have that
$$
\EE_{\QQ^{z,\Omega}}\left[\left|\varphi(\omega(s+\eps)) \mathbf{1}_{\{s+\eps < {\bf e}^{\Omega}\}} - \varphi(\omega(s)) \mathbf{1}_{\{s < {\bf e}^{\Omega}\}}\right|\right] \rightarrow 0,\quad\hbox{ as }\eps\downarrow 0.
$$
Because the random variable $\varphi(\omega(s)) \mathbf{1}_{\{s < {\bf e}^{\Omega}\}}$ is $\cB^{\Omega}_s$-measurable, we have that
\begin{equation}
\label{eq:Martingale_reformulation_eps_first_term}
\EE_{\QQ^{z,\Omega}}\left[\varphi(\omega(s+\eps)) \mathbf{1}_{\{s+\eps < {\bf e}^{\Omega}\}} \Big{|}\cB^{\Omega}_s\right]  \rightarrow \varphi(\omega(s)) \mathbf{1}_{\{s < {\bf e}^{\Omega}\}},\quad \QQ^{z,\Omega}\hbox{-a.s.}
\end{equation}
Moreover, using identity \eqref{eq:Expectation_integral_term}, we can apply the Dominated Convergence Theorem to the integral term in identity \eqref{eq:Martingale_reformulation_eps} to obtain that
$$
\EE_{\QQ^{z,\Omega}}\left[\left| \int_0^{t-s}  L\varphi(\omega(s+r)) \mathbf{1}_{\{s+r < {\bf e}^{\Omega}\}}\, dr- \int_{\eps}^{t-s}  L\varphi(\omega(s+r)) \mathbf{1}_{\{s+r < {\bf e}^{\Omega}\}}\, dr\right|\right]\rightarrow 0,\quad\hbox{ as }\eps\downarrow 0.
$$
Therefore, it follows that
\begin{equation}
\label{eq:Martingale_reformulation_eps_second_term}
\begin{aligned}
&\EE_{\QQ^{z,\Omega}}\left[ \int_{\eps}^{t-s}  L\varphi(\omega(s+r)) \mathbf{1}_{\{s+r < {\bf e}^{\Omega}\}}\, dr \Big{|}\cB^{\Omega}_s\right]\\ &\qquad\rightarrow 
\EE_{\QQ^{z,\Omega}}\left[ \int_0^{t-s}  L\varphi(\omega(s+r))\mathbf{1}_{\{s+r < {\bf e}^{\Omega}\}}\, dr \Big{|}\cB^{\Omega}_s\right],
\quad \QQ^{z,\Omega}\hbox{-a.s.},
\end{aligned}
\end{equation}
as $\eps$ tends to $0$. From identity \eqref{eq:Martingale_reformulation_eps}, using properties \eqref{eq:Martingale_reformulation_eps_first_term} and \eqref{eq:Martingale_reformulation_eps_second_term}, we obtain \eqref{eq:Martingale_reformulation}, which immediately implies that the process $\{M^{\varphi}(t)\}_{t \geq 0}$ is a $\QQ^{z,\Omega}$-martingale with respect to the filtration $\{\cB^{\Omega}_t\}_{t\geq 0}$.
\end{step}
This concludes the proof.
\end{proof}

\section{Solution to the Kimura stochastic differential equation with singular drift}
\label{sec:Kimura_SDE}
In this section, we establish the relationship between the semigroup $\{T^{\Omega}_t\}_{t\geq 0}$ constructed in \S \ref{sec:Inhomogeneous_initial_value_problem}, and the unique Markov solutions to the Kimura stochastic differential equation with singular drift \eqref{eq:Kimura_SDE_singular}. In \S \ref{subsec:Stopped_solutions_Kimura_SDE_singular}, we introduce the stochastic differential equation \eqref{eq:Kimura_SDE_singular} and we describe the relationship between the coefficients of the stochastic differential equation \eqref{eq:Kimura_SDE_singular} and those of the operator $L$, defined in \eqref{eq:Singular_operator}. In Proposition \ref{prop:Stopped_solutions}, we establish the first relationship between the semigroup $\{T^{\Omega}_t\}_{t\geq 0}$ and solutions to the singular Kimura stochastic differential equation \eqref{eq:Kimura_SDE_singular} that are stopped when the process hits the non-degenerate boundary of a domain $\Omega\subseteq S_{n,m}$. In \S \ref{subsec:Semigroup_SDE}, we prove Theorem \ref{thm:Connection_semigroup_Markov_solution}, where we improve Proposition \ref{prop:Stopped_solutions} in that we no longer require the solution to the Kimura stochastic differential equation \eqref{eq:Kimura_SDE_singular} to be stopped on the non-degenerate boundary of the domain $\Omega$. Theorem \ref{thm:Connection_semigroup_Markov_solution} is then used in \S \ref{subsec:Stochastic_representation_and_Harnack} to prove Theorem \ref{thm:Stochastic_representation} where we establish that functions defined by the stochastic representation formula \eqref{eq:Stochastic_representation_cont} satisfy the scale-invariant Harnack inequality \eqref{eq:Harnack_cont}. This is the key result that we use in \S \ref{sec:Harnack_inequality} to complete the proof of the Harnack inequality for nonnegative solution to the parabolic equations defined by the standard Kimura operator $\widehat L$.

\subsection{Stopped solutions to the Kimura stochastic differential equation with singular drift}
\label{subsec:Stopped_solutions_Kimura_SDE_singular}
Recall that the singular Kimura stochastic differential equation with
logarithmic drift is given by \eqref{eq:Kimura_SDE_singular}. The coefficients of the stochastic differential equation \eqref{eq:Kimura_SDE_singular} are related to the coefficients of the differential operator $L$ defined in \eqref{eq:Singular_operator}, as follows. For all $i=1,\ldots, n$, $j=1,\ldots, n+m$, and $l=1,\ldots,m$, we let
\begin{equation}
\label{eq:Coeff_Kimura_SDE_L}
\begin{aligned}
g_i(z) &:=	b_ia_{ii}+x_i\left(\partial_{x_i} a_{ii}+\sum_{j=1}^n\left(\tilde a_{ij}+\delta_{ij}\tilde a_{ii}+x_j\partial_{x_j}\tilde a_{ij} + \tilde a_{ij}(b_j-1)\right)+\sum_{l=1}^m \partial_{y_l} c_{il}\right),\\
e_l(z) &:= \sum_{i=1}^n\left(x_i\partial_{x_i}c_{il}+b_ic_{il}\right)+\sum_{k=1}^m \partial_{y_k} d_{lk},\\
f_{ij}(z) &:= \sum_{j=1}^n \left(\partial_{x_i} b_j +\sum_{k=1}^n x_k\tilde a_{ik}\partial_{x_k} b_j +\sum_{l=1}^m c_{il} \partial_{y_l} b_j\right),\\
f_{n+l,j}(z) &:= \sum_{i=1}^n x_ic_{il}\partial_{x_i} b_j+\sum_{k=1}^m d_{lk} \partial_{y_k} b_j.
\end{aligned}
\end{equation}
To construct the dispersion coefficient matrix, $(\sigma(z))$, appearing in \eqref{eq:Kimura_SDE_singular}, we introduce the diffusion matrix, $(D(z))$, by letting, for all $i,j=1,\ldots,n$ and all $l,k=1,\ldots,m$,
\begin{equation}
\label{eq:Matrix_D}
\begin{aligned}
D_{ii}(z) &:=2a_{ii}(z)+2x_i\tilde a_{ii}(z),\\
D_{ij}(z) &:=2\sqrt{x_ix_j}\tilde a_{ij}(z),\quad i\neq j,\\
D_{i,n+l}(z)=D_{n+l,i}(z) &:=4\sqrt{x_i} c_{il}(z),\\
D_{n+l,n+k}(z) &=2d_{lk}(z).
\end{aligned}
\end{equation}
We now argue that there is a suitable function $(\sigma(z))$ such that identity \eqref{eq:Relation_sigma_D} holds. Notice that conditions \eqref{eq:Ellipticity_all_second_order_coeff}, \eqref{eq:Outside_neighborhood_O}  and \eqref{eq:Outside_neighborhood_O_a_ii} imply that the matrix $(D(z))$ is strictly elliptic. From Assumption \ref{assump:Coeff}, it follows that the coefficients $a_{ii}$, $\tilde a_{ij}$, $c_{il}$ and $d_{lk}$ are smooth functions of the variable $z=(x,y)$ on $\bar S_{n,m}$, and in particular they are smooth functions of the variable $(\sqrt{x}, y)$ on $\bar S_{n,m}$, where we denote
\begin{equation}
\label{eq:Sqrt_x}
\sqrt{x} = \left(\sqrt{x_1}, \sqrt{x_2}, \ldots, \sqrt{x_n}\right),\quad\forall\, x\in \RR^n_+.
\end{equation}
We obtain that the matrix $D$ defined in \eqref{eq:Matrix_D} is smooth in the variables $(\sqrt{x}, y)$ on $\bar S_{n,m}$. We let $\widetilde D(\sqrt{x},y) = D(x,y)$, for all $(x,y)\in \bar S_{n,m}$. Because the matrix $\widetilde D$ is strictly elliptic and smooth, we can build an extension from $\bar S_{n,m}$ to $\RR^{n+m}$, which we denote the same as the matrix $\widetilde D$, such that the extended matrix remains strictly elliptic and smooth on $\RR^{n+m}$. We can now apply \cite[Lemma 6.1.1]{FriedmanSDE} to the matrix $\widetilde D$, to obtain that there is a matrix $\widetilde \sigma \in C^{\infty}(\RR^{n+m};\RR^{n+m}\times\RR^{n+m})$, such that $\widetilde\sigma\widetilde\sigma^*=\widetilde D$. Letting now
$$
\sigma(x^2,y):= \widetilde \sigma(x,y), \quad\forall (x,y)\in \bar S_{n,m},
$$ 
we obtain that
\begin{equation}
\label{eq:Relation_sigma_D}
(\sigma\sigma^*)(z)=D(z),\quad \forall\, z\in\bar S_{n,m},
\end{equation}
and that $\sigma$ is a smooth function in the variables $(\sqrt{x},y)$ on $\bar S_{n,m}$. In particular, the function $\sigma$ belongs to the classical H\"older space $C^{\alpha}(\bar S_{n,m})$ and to the anisotropic H\"older space $C^{\alpha}_{WF}(\bar S_{n,m})$ (defined in \eqref{eq:Anisotropic_Holder_space}), for all $\alpha \in [0,1/2]$. 

The choice of the `square root', $(\sigma(z))$, of the positive-definite matrix $(D(z))$ is irrelevant for the question of existence and uniqueness of weak solutions to the stochastic differential equations \eqref{eq:Kimura_SDE_singular}, as \cite[Problem 5.4.7]{KaratzasShreve1991} shows. 

Let $\Omega\subseteq S_{n,m}$ be an open set, and let $z\in\underline{\Omega}$. In this section, we consider weak solutions, $\{Z^{\Omega}(t)\}_{0\leq t <\tau_{\Omega}}$, to the stochastic differential equation \eqref{eq:Kimura_SDE_singular}, defined up to the stopping time $\tau_{\Omega}$,
\begin{equation}
\label{eq:Exit_time_Omega}
\tau_{\Omega} :=\lim_{k\rightarrow\infty} \tau_{\Omega_k},
\end{equation}
where 
\begin{equation}
\label{eq:Omega_k}
\underline\Omega_k:=\{z\in\underline\Omega:\ \dist(z,\partial_1\Omega)>1/k\},\quad\forall\, k\in\NN,
\end{equation}
is an increasing family of compact sets in $\underline\Omega$, chosen such that $\underline\Omega=\cup_{k\in\NN} \underline\Omega_k$, and
\begin{equation}
\label{eq:tau_Omega_k}
\tau_{\Omega_k}:= \inf\{t \geq 0:\, Z^{\Omega}(t)\notin \underline\Omega_k \},\quad\forall\, k\in\NN.
\end{equation}

\begin{prop}[Weak solutions to the stochastic differential equation \eqref{eq:Kimura_SDE_singular} on $\Omega$]
\label{prop:Stopped_solutions}
Let $\Omega\subseteq S_{n,m}$ be an open set, and let $z\in\underline{\Omega}$. Then there is a filtered probability space, $(\widetilde\cX, \widetilde\cB, \{\widetilde\cB_t\}_{t\geq 0},\PP^{z,\Omega})$, an $(n+m)$-dimensional Brownian motion, $\{W(t)\}_{t\geq 0}$, and a process, $\{Z^{\Omega}(t)\}_{t\geq 0}$, that solves the Kimura stochastic differential equation with singular drift \eqref{eq:Kimura_SDE_singular}, for all $0\leq t<\tau_{\Omega}$, and satisfies the initial condition, $Z^{\Omega}(0)=z$, $\PP^{z,\Omega}$-a.s. Moreover, the process $\{Z^{\Omega}(t)\}_{t\geq 0}$ satisfies the Markov property, and for all bounded, Borel measurable functions, $f:\underline\Omega\rightarrow\RR$, and $t\geq 0$, we have that
\begin{align}
\label{eq:Weak_solution_semigroup}
\left(T^{\Omega}_t f\right)(z) = \EE_{\PP^{z,\Omega }} \left[f(Z^{\Omega}(t)) \mathbf{1}_{\{t <\tau_{\Omega}\}}\right],\quad\forall z\in\underline\Omega.
\end{align}
\end{prop}

\begin{proof}
Our method of the proof is based on the existence of solutions to the martingale problem associated to the operator $L$ on the domain $\Omega$, established in Proposition \ref{prop:Solution_local_martingale_problem}, and the proof of \cite[Proposition 5.4.6]{KaratzasShreve1991}. We need to be careful that our process is defined for all $t\in [0,\tau_{\Omega})$, as opposed to $[0,\infty)$, when we make use of the arguments used to prove \cite[Proposition 5.4.6]{KaratzasShreve1991}. Let $\psi_k:\underline\Omega\rightarrow[0,1]$ be a smooth function with compact support, such that $\psi_k(z)=1$, for all $z\in \underline\Omega_k$. Let $\{Z^{\Omega}(t)\}_{t\geq 0}$ be the coordinate process on the filtered probability space $(\cX^{\Omega}, \cB^{\Omega}, \{\cB^{\Omega}_t\}_{t\geq 0})$, that is, we let $Z^{\Omega}(t,\omega)=\omega(t)$, for all $\omega\in\cX^{\Omega}$, and let $\QQ^{z,\Omega}$ be the solution to the martingale problem associated to the operator $L$ on $\Omega$, constructed in Proposition \ref{prop:Solution_local_martingale_problem}. 

We denote the drift coefficients of the stochastic differential equation \eqref{eq:Kimura_SDE_singular} by
\begin{equation}
\label{eq:Drift_coeff}
\begin{aligned}
\delta_i(z) &:= b_i(z) + x_i \sum_{j=1}^n f_{ij}(z) \ln x_j,\quad\forall\, i=1,\ldots,n,\\ 	
\delta_{n+l}(z) &:= e_l(z) + \sum_{j=1}^n f_{n+l,j}(z) \ln x_j,\quad\forall\, l=1,\ldots,m, 	
\end{aligned}
\end{equation}
and we denote the covariance coefficient matrix of the stochastic differential equation \eqref{eq:Kimura_SDE_singular} by
\begin{equation}
\label{eq:Definition_matrix_alpha}
\begin{aligned}
\alpha_{ij}(z)&:=\sqrt{x_ix_j} D_{ij}(z),\quad\forall\, i,j=1,\ldots,n,\\
\alpha_{il}(z)=\alpha_{li}(z)&:=\sqrt{x_i} D_{i,n+l}(z),\quad\forall\, i=1,\ldots,n,\,\forall\, l=1,\ldots,m,\\
\alpha_{lk}(z)&:=  D_{n+l,n+k}(z),\quad\forall\, l,k=1,\ldots,m,
\end{aligned}
\end{equation}
where $(D(z))$ is defined in \eqref{eq:Matrix_D}, for all $z\in\bar S_{n,m}$. Similarly to the proof of \cite[Proposition 5.4.6]{KaratzasShreve1991}, we consider the processes,
\begin{align*}
M_i(t) &:= Z^{\Omega}_i(t) - Z^{\Omega}_i(0) - \int_0^t \delta_i(Z^{\Omega}(s))\, ds,\quad\forall\, 0\leq t<\tau_{\Omega},\\
Q_{ij}(t) &:= M_i(t)M_j(t) - M_i(0)M_j(0) - \int_0^t \alpha_{ij}(Z^{\Omega}(s))\, ds,\quad\forall\, 0\leq t<\tau_{\Omega},
\end{align*}
for all $i,j=1,\ldots,n+m$. By choosing $\varphi(z)=z_i\psi_k(z)$, for all $i=1,\ldots,n+m$, as test function in \eqref{eq:Martingale}, it follows from Proposition \ref{prop:Solution_local_martingale_problem} that the process $\{M_i(t\wedge \tau_{K_k})\}_{t\geq 0}$ 
is a $\QQ^{z,\Omega}$-martingale with respect to the filtration
$\{\cB^{\Omega}_t\}_{t\geq 0}$, for all $k\in\NN$.  By choosing next
$\varphi(z)=z_iz_j\psi_k(z)$, for all $i,j=1,\ldots,n+m$, as test function in
\eqref{eq:Martingale}, it follows from Proposition
\ref{prop:Solution_local_martingale_problem} that the process $\{Q_{ij}(t\wedge
\tau_{K_k})\}_{t\geq 0}$ is also a $\QQ^{z,\Omega}$-martingale with respect to
the filtration $\{\cB^{\Omega}_t\}_{t\geq 0}$, for all $k\in\NN$. Because the
processes $\{M_i(t)\}$ and $\{Q_{ij}(t)\}$ are defined only for $t\in
[0,\tau_{\Omega})$, they are not local martingales in the classical sense of
\cite[Definition 1.5.15]{KaratzasShreve1991}, and so, we cannot apply directly
\cite[Theorem 3.4.2]{KaratzasShreve1991} to obtain the existence of a suitable
$(n+m)$-dimensional Brownian motion. We now describe the modifications we need
to use the proof of \cite[Theorem 3.4.2]{KaratzasShreve1991} to obtain the existence of weak solutions to equation \eqref{eq:Kimura_SDE_singular}, for all $0\leq t<\tau_{\Omega}$.

From identities \eqref{eq:Relation_sigma_D} and \eqref{eq:Definition_matrix_alpha}, using the fact that the diffusion matrix $(D(z))$ is strictly elliptic, it follows that the matrix $(\alpha(z))$, defined in \eqref{eq:Definition_matrix_alpha}, is symmetric and nonnegative definite on $\bar S_{n,m}$. Thus, there is an orthogonal matrix, $(Q(z))=(q_{ij}(z))$, and a nonnegative diagonal matrix, $(\Lambda(z))=\diag(\lambda_i(z))$, such that $\alpha(z)=Q^*(z)\Lambda(z)Q(z)$, for all $z\in\bar S_{n,m}$. Following \cite[Identity (3.4.7)]{KaratzasShreve1991}, we define the process
\begin{align*}
N_i(t) = \sum_{j=1}^{n+m}\int_0^t q_{ji}\left(Z^{\Omega}(s)\right)\, dM_j(s),\quad\forall\, 0\leq t<\tau_{\Omega},
\end{align*}
for all $i=1,\ldots,n+m$, and as in \cite[Identity (3.4.8)]{KaratzasShreve1991}, we see that
\begin{align*}
\left\langle N_i, N_j \right\rangle(t) = \delta_{ij}\int_0^t \lambda_{i}\left(Z^{\Omega}(s)\right)\, ds,\quad\forall\, 0\leq t<\tau_{\Omega},
\end{align*}
where $\left\langle \cdot,\cdot\right\rangle$ denotes the quadratic covariation. Then there is an extension of the filtered probability space $(\cX^{\Omega}, \cB^{\Omega}, \{\cB^{\Omega}_t\}_{t\geq 0}, \QQ^{z,\Omega})$, which we denote by $(\widetilde\cX, \widetilde\cB, \{\widetilde\cB_t\}_{t\geq 0}, \PP^{z,\Omega})$, that supports an independent $(n+m)$-dimensional Brownian motion, $\{\widetilde B(t)\}_{t \geq 0}$, such that by letting
\begin{align*}
\widetilde W_i(t) &=\int_0^{t\wedge\tau_{\Omega}} \mathbf{1}_{\left\{\lambda_i\left(Z^{\Omega}(s)\right)>0\right\}} \frac{1}{\sqrt{\lambda_i\left(Z^{\Omega}(s)\right)}}\, dN_i(s)\\
&\quad+\int_0^t\left(\mathbf{1}_{\left\{\lambda_i\left(Z^{\Omega}(s)\right)=0\right\}} \mathbf{1}_{\left\{s<\tau_{\Omega}\right\}}+ \mathbf{1}_{\left\{s\geq\tau_{\Omega}\right\}}\right)\, d\widetilde B_i(s),\quad\forall\, i=1,\ldots,n+m,
\end{align*}
the process $\{\widetilde W(t)\}_{t\geq 0}$ is an $(n+m)$-dimensional Brownian motion. We now let $\rho_{ij}(z):=q_{ij}(z)\sqrt{\lambda_j(z)}$, and similarly to \cite[Identity (3.4.12)]{KaratzasShreve1991}, we obtain the representation
$$
M_i(t) = \sum_{j=1}^{n+m}\int_0^t \rho_{ij}\left(Z^{\Omega}(s)\right)\, d\widetilde W_j(s),\quad\forall\, 0\leq t<\tau_{\Omega}.
$$
This implies that the process $\{Z^{\Omega}(t)\}_{t\in[0,\infty)}$ satisfies
\begin{equation}
\label{eq:Weak_sol_1}
Z^{\Omega}_i(t) = Z^{\Omega}_i(0) + \int_0^t \delta_i(Z^{\Omega}(s))\, ds+\sum_{j=1}^{n+m}\int_0^t \rho_{ij}\left(Z^{\Omega}(s)\right)\, d\widetilde W_j(s),\quad\forall\, 0\leq t<\tau_{\Omega},
\end{equation}
for all $i=1,\ldots,n+m$. From identity \eqref{eq:Relation_sigma_D} and the fact that the matrix $(\rho(z))$ has the property that $\rho(z)\rho^*(z)=\alpha(z)$, where $(\alpha(z))$ is defined in \eqref{eq:Definition_matrix_alpha}, it follows from \cite[Problem 5.4.7]{KaratzasShreve1991}, that there is an orthogonal matrix $(R(z))$, such that 
\begin{align}
\label{eq:Rotation_matrix}
\sqrt{x_i}\sigma_i(z)=\rho_i(z)R(z),\quad\hbox{and}\quad \sigma_{n+l}(z)=\rho_{n+l}(z)R(z),
\end{align}
for all $i=1,\ldots,n$ and all $l=1,\ldots,m$. By letting 
$$
W(t)=\int_0^{t\wedge\tau_{\Omega}} R^*(Z^{\Omega}(s))\, d\widetilde W(s) +\mathbf{1}_{\{t\geq \tau_{\Omega}\}}\int_{\tau_{\Omega}}^t \, d\widetilde W(s),\quad\forall\, t \geq 0,
$$
the process $\{W(t)\}_{t\geq 0}$ is an $(n+m)$-dimensional Brownian motion, and using identities \eqref{eq:Weak_sol_1}, \eqref{eq:Drift_coeff} and \eqref{eq:Rotation_matrix}, we see that the process $\{Z^{\Omega}(t)\}_{0\leq t<\tau_{\Omega}}$ is indeed a weak solution to the stochastic differential equation \eqref{eq:Kimura_SDE_singular}. 

From the proof of Proposition \ref{prop:Solution_local_martingale_problem}, it follows that the laws of the coordinate process $\{Z^{\Omega}(t)\}_{t\geq 0}$ restricted to the probability space $(\cX^{\Omega}, \cB^{\Omega}, \{\cB^{\Omega}_t\}_{t\geq 0}, \QQ^{z,\Omega})$, and that of the process $\{Z^{0,\Omega}(t)\}_{t\geq 0}$ constructed in Proposition \ref{prop:Stopped_process}, agree. The extension of the process $\{Z^{\Omega}(t)\}_{t\geq 0}$ from the probability space  $(\cX^{\Omega}, \cB^{\Omega}, \{\cB^{\Omega}_t\}_{t\geq 0}, \QQ^{z,\Omega})$ to the larger space $(\widetilde\cX, \widetilde\cB, \{\widetilde\cB_t\}_{t\geq 0}, \PP^{z,\Omega})$, is done my adding an $(n+m)$-dimensional independent Brownian motion, and so, the laws of the process $\{Z^{\Omega}(t)\}_{t\geq 0}$  on the space $(\widetilde\cX, \widetilde\cB, \{\widetilde\cB_t\}_{t\geq 0}, \PP^{z,\Omega})$, and that of the process $\{Z^{0,\Omega}(t)\}_{t\geq 0}$ constructed in Proposition \ref{prop:Stopped_process}, also agree. Thus, properties \eqref{eq:Markov_prop_stopped_process} and \eqref{eq:Process_after_stopping_time}, satisfied by the process $\{Z^{0,\Omega}(t)\}_{t\geq 0}$ imply that the process $\{Z^{\Omega}(t)\}_{t\geq 0}$ also satisfies the Markov property and \eqref{eq:Weak_solution_semigroup}, respectively. This concludes the proof.
\end{proof}

When $\Omega=S_{n,m}$, we denote by $\{Z(t)\}_{t \geq 0}$, instead of $\{Z^{\Omega}(t)\}_{t \geq 0}$, the process constructed in Proposition \ref{prop:Stopped_solutions}, and we denote by $\PP^z$, instead of $\PP^{z,\Omega}$ the probability measure appearing in the statement of Proposition \ref{prop:Stopped_solutions}. In this case, we have that $\tau_{S_{n,m}}=\infty$, since $\partial_1 S_{n,m}=\emptyset$, and the stochastic differential equation \eqref{eq:Kimura_SDE_singular} is satisfied for all $t \in (0,\infty)$. Because the process $\{Z(t)\}_{t \geq 0}$ satisfies the Markov property, it follows from \cite[Theorem 3.8]{Pop_2013a}, that the process $\{Z(t)\}_{t \geq 0}$ is the unique weak solution to the stochastic differential equation \eqref{eq:Kimura_SDE_singular}, in the class of Markov processes.

\subsection{Connection between semigroups and solutions to stochastic differential equations}
\label{subsec:Semigroup_SDE}
In this section, we give the proof of Theorem \ref{thm:Connection_semigroup_Markov_solution} where we establish with the aid of Proposition \ref{prop:Stopped_solutions}, the connection between the semigroup $\{T^{\Omega}_t\}_{t\geq 0}$, constructed in \S \ref{sec:Inhomogeneous_initial_value_problem}, and the unique weak solutions, $\{Z(t)\}$, to the stochastic differential equation \eqref{eq:Kimura_SDE_singular} on $S_{n,m}$, which satisfy the Markov property, established in \cite[Theorem 3.8]{Pop_2013a}. 

\begin{proof}[Proof of Theorem \ref{thm:Connection_semigroup_Markov_solution}]
Recall the definition of the increasing sequence of sets $\{\underline\Omega_k\}_{k\in\NN}$ in \eqref{eq:Omega_k}, and of the stopping times $\{\tau_{\Omega_k}\}_{k\in\NN}$ in \eqref{eq:tau_Omega_k}. Let $\{Z^{\Omega}(t)\}_{0\leq t<\tau_{\Omega}}$ be the Markov process constructed in Proposition \ref{prop:Stopped_solutions}. Using identity \eqref{eq:Weak_solution_semigroup}, we see that \eqref{eq:Connection_semigroup_Markov_process} follows if we prove that the probability distributions of the random variables $Z^{\Omega}(t)\mathbf{1}_{\{t<\tau_{\Omega}\}}$ under $\PP^{z,\Omega}$, and of $Z(t)\mathbf{1}_{\{t<\tau_{\Omega}\}}$ under $\PP^z$ agree, for all $t\geq 0$. In turn, using the fact that the sequence of stopping times $\{\tau_{\Omega_k}\}_{k\in\NN}$ converges to $\tau_{\Omega}$, as $k$ tends to $\infty$, this follows if we prove that the probability distributions of the random variables $Z^{\Omega}(t)\mathbf{1}_{\{t\leq\tau_{\Omega_k}\}}$ under $\PP^{z,\Omega}$, and of $Z(t)\mathbf{1}_{\{t\leq \tau_{\Omega_k}\}}$ under $\PP^z$ agree, for all $k\in\NN$ and all $t\geq 0$. Recall from the construction in Proposition \ref{prop:Stopped_solutions} that the probability distribution $\PP^{z,\Omega}$, restricted to the space $(\cX^{\Omega}, \cB^{\Omega})$, agrees with the solution to the martingale problem, $\QQ^{z,\Omega}$. When $\Omega=S_{n,m}$, we denote for brevity
$$
(\cX^{\Omega}, \{\cB^{\Omega}_t\}_{t\geq 0},\cB^{\Omega},\QQ^{z,\Omega}) = (\cX, \{\cB_t\}_{t\geq 0},\cB,\QQ^z),
$$
and we let $\{Z^0(t)\}_{t\geq 0}$ be the coordinate process on $(\cX,\cB)$. Thus, it is sufficient to prove that the probability distribution of the random variable
$$
\cX \ni \omega\mapsto Z^0(t,\omega) \mathbf{1}_{\{t\leq \tau_{\Omega_k}\}}
$$
is the same under the probability measures $\QQ^{z,\Omega}$ and $\QQ^{z}$, for all $k\in\NN$ and $t\geq 0$. 

Our method of the proof consists in building a new probability measure, $\widetilde\QQ^z$, on $(\cX,\cB)$ by combining the measures $\QQ^{z,\Omega}$ and $\QQ^{z}$, such that for all $j\in\NN$, $0\leq t_1<t_2<\ldots<t_j$, $B_i\in\cB(\bar S_{n,m})$, for $i=1,\ldots,j$, we have that
\begin{equation}
\label{eq:tilde_Q_Omega}
\begin{aligned}
&\widetilde \QQ^z\left(Z^0(t_i)\in B_i,\, t_i\leq\tau_{\Omega_k},\,\forall\, i=1,\ldots,j\right)\\
&\qquad= \QQ^{z,\Omega}\left(Z^0(t_i)\in B_i,\, t_i\leq\tau_{\Omega_k},\,\forall\, i=1,\ldots,j\right).
\end{aligned}
\end{equation}
We then prove that $\widetilde \QQ^z$ is also a solution of the Martingale Problem \ref{def:Martingale_problem} associated to the operator $L$ on $S_{n,m}$, and that the coordinate process $\{Z^0(t)\}_{t\geq 0}$ satisfies the Markov property under $\widetilde \QQ^z$. From the uniqueness of weak solutions to the singular Kimura stochastic differential equation established in \cite[Theorem 3.8]{Pop_2013a}, it follows that the distributions of the coordinate process $\{Z^0(t)\}_{t\geq 0}$ under the measures $\widetilde\QQ^z$ and $\QQ^z$ must agree. Combining this property with \eqref{eq:tilde_Q_Omega}, it follows that indeed the distributions of the random variable $Z^0(t)\mathbf{1}_{\{t\leq \tau_{\Omega_k}\}}$ are the same under the probability measures
$\QQ^{z,\Omega}$ and $\QQ^{z}$, for all $k\in\NN$ and $t\geq 0$, which completes the proof.

\setcounter{step}{0}
\begin{step}[Construction of the probability measure $\widetilde\QQ^z$]
\label{step:Definition_tilde_Q}
Because $\cX^{\Omega}$ is a measurable set in $\cX$, we can view the probability measure $\QQ^{z,\Omega}$ as one defined on $(\cX,\cB)$. As usual, we let 
$$
\cB_{\tau_{\Omega_k}}:=\sigma\left(\omega(t\wedge\tau_{\Omega_k}):\, t \geq 0\right),
$$
and we consider the unique, up to sets of $\QQ^z$-measure zero, regular conditional probability distribution of $\QQ^{z}$ given $\cB_{\tau_{\Omega_k}}$, given by \cite[Theorem 5.3.18]{KaratzasShreve1991}. We denote by $\{\QQ^z_{\omega}\}$ a regular conditional probability distribution of $\QQ^{z}$ given $\cB_{\tau_{\Omega_k}}$, and using \cite[Theorem 1.3.4]{Stroock_Varadhan}, we can choose $\{\QQ^z_{\omega}\}$ such that
$$
\QQ^z_{\omega}(A)=\mathbf{1}_{\{\omega\in A\}},\quad\forall\, \omega\in\cX,\quad\forall\, A\in\cB_{\tau_{\Omega_k}},
$$
as opposed to the preceding identity holding for all $\omega\in\cX\backslash \cN$, where $\cN$ is a set of $\QQ^z$-measure zero. 
Using \cite[Definition 5.3.16]{KaratzasShreve1991} and the preceding identity, we see that the hypotheses of \cite[Theorem 6.1.2]{Stroock_Varadhan} are satisfied, and so, we can construct a unique probability measure on $(\cX,\cB)$, which we denote by
$$
\widetilde\QQ^z =\QQ^{z,\Omega} \otimes_{\tau_{\Omega_k}} \QQ^z,
$$
such that $\widetilde\QQ^z$ coincides with $\QQ^{z,\Omega}$ on $\cB_{\tau_{\Omega_k}}$, that is, identity \eqref{eq:tilde_Q_Omega} holds, and $\delta_{\omega}\otimes_{\tau_{\Omega_k}}\QQ^{z}$ is a regular conditional probability distribution of $\widetilde\QQ^z$ on $\cB_{\tau_{\Omega_k}}$, where the probability measure  $\delta_{\omega}\otimes_{\tau_{\Omega_k}}\QQ^{z}$ is defined as in \cite[Lemma 6.1.1]{Stroock_Varadhan}. Thus, it follows that for all $j\in\NN$, $0\leq t_1<t_2<\ldots<t_j$, $B_i\in\cB(\bar S_{n,m})$, for $i=1,\ldots,j$, we can choose a version of the conditional probability such that
\begin{align*}
\QQ^z\left(Z^0(t_i)\in B_i,\,\forall\, i=1,\ldots,j\Big{|}\cB_{\tau_{\Omega_k}}\right)(\omega) =
\QQ^z_{\omega}\left(Z^0(t_i)\in B_i,\,\forall\, i=1,\ldots,j\right),\quad\forall\, \omega\in\cX,
\end{align*}
and so, we can compute
\begin{equation}
\label{eq:tilde_Q}
\begin{aligned}
\widetilde \QQ^z\left(Z^0(t_i)\in B_i,\,\forall\, i=1,\ldots,j\right)
&= \QQ^{z,\Omega}\left(\QQ^z\left(Z^0(t_i)\in B_i,\,\forall\, i=1,\ldots,j\Big{|}\cB_{\tau_{\Omega_k}}\right)\right).
\end{aligned}
\end{equation}
This completes the construction of the probability measure $\widetilde\QQ^z$ on $(\cX,\cB)$.
\end{step}

\begin{step}[Proof that the probability measure $\widetilde\QQ^z$ solves the Martingale Problem \ref{def:Martingale_problem} on $S_{n,m}$]
\label{step:Martingale_problem_tilde_Q}
In this step, it is sufficient to show that the process defined by \eqref{eq:Martingale}, with ${\bf e}^{\Omega}=\infty$, is a $\widetilde\QQ^z$-martingale, for all functions $\varphi\in C^{\infty}_c(\bar S_{n,m})$. We apply \cite[Theorem 6.1.2]{Stroock_Varadhan} with $\theta(t):=M^{\varphi}(t)$ and $s=0$. Because the measure $\QQ^{z,\Omega}$ solves the Martingale Problem \ref{def:Martingale_problem} associated to the operator $L$ on $\Omega$, it follows that the process $\{M^{\varphi}(t\wedge\tau_{\Omega_k})\}_{t\geq 0}$ is a $\QQ^{z,\Omega}$-martingale. Recall that the probability measure $\QQ^z$ solves the Martingale Problem \ref{def:Martingale_problem} associated to the operator $L$ on $S_{n,m}$. It follows that the process $\{M^{\varphi}(t)-M^{\varphi}(t\wedge\tau_{\Omega_k}(\omega))\}_{t\geq 0}$ is a $\QQ^z_{\omega}$-martingale, for all $\omega\in\cX$. Thus, we obtain from \cite[Theorem 6.1.2]{Stroock_Varadhan} that $\{M^{\varphi}(t)\}$ is a $\widetilde\QQ^z$-martingale, which completes the proof of Step \ref{step:Martingale_problem_tilde_Q}.
\end{step}

\begin{step}[Proof of the Markov property of the coordinate process $\{Z^0(t)\}_{t\geq 0}$ under $\widetilde\QQ^z$]
\label{step:Markov_property_tilde_Q}
In this step, our goal is to prove that for all bounded, Borel measurable functions, $f:\bar S_{n,m}\rightarrow\RR$, and all $0\leq s<t$, we have that
\begin{equation*}
\EE_{\widetilde\QQ^z}\left[f\left(Z^0(t)\right)\big{|}\cB_s\right] = \EE_{\widetilde\QQ^z}\left[f\left(Z^0(t)\right)\big{|}Z^0(s)\right].
\end{equation*}
The preceding identity is equivalent to proving that, for all measurable sets, $A\in\cB_s$, we have that
\begin{equation}
\label{eq:Markov_property_tilde_Q_with_set}
\EE_{\widetilde\QQ^z}\left[\mathbf{1}_Af\left(Z^0(t)\right)\right] = \EE_{\widetilde\QQ^z}\left[\mathbf{1}_A\EE_{\widetilde\QQ^z}\left[f\left(Z^0(t)\right)\big{|}Z^0(s)\right]\right].
\end{equation}
Identity \eqref{eq:Markov_property_tilde_Q_with_set} follows by adding the following three equalities:
\begin{align}
\label{eq:Markov_property_tilde_Q_with_set_1}
\EE_{\widetilde\QQ^z}\left[\mathbf{1}_A \mathbf{1}_{\{\tau_{\Omega_k}>t>s\}}f\left(Z^0(t)\right)\right] 
&= \EE_{\widetilde\QQ^z}\left[\mathbf{1}_A\mathbf{1}_{\{\tau_{\Omega_k}>s\}}\EE_{\widetilde\QQ^z}\left[f\left(Z^0(t)\right)
\mathbf{1}_{\{\tau_{\Omega_k}>t\}}\big{|}Z^0(s) \right] \right]\\
\label{eq:Markov_property_tilde_Q_with_set_2}
\EE_{\widetilde\QQ^z}\left[\mathbf{1}_A \mathbf{1}_{\{t\geq \tau_{\Omega_k}>s\}}f\left(Z^0(t)\right)\right] 
&= \EE_{\widetilde\QQ^z}\left[\mathbf{1}_A\mathbf{1}_{\{\tau_{\Omega_k}>s\}}\EE_{\widetilde\QQ^z}\left[f\left(Z^0(t)\right)
\mathbf{1}_{\{t\geq \tau_{\Omega_k}\}}\big{|}Z^0(s) \right] \right]\\
\label{eq:Markov_property_tilde_Q_with_set_3}
\EE_{\widetilde\QQ^z}\left[\mathbf{1}_A \mathbf{1}_{\{t>s\geq \tau_{\Omega_k}\}}f\left(Z^0(t)\right)\right] 
&= \EE_{\widetilde\QQ^z}\left[\mathbf{1}_A\mathbf{1}_{\{s\geq \tau_{\Omega_k}\}}\EE_{\widetilde\QQ^z}\left[f\left(Z^0(t)\right)
\big{|}Z^0(s) \right] \right].
\end{align}
We now prove each one of the preceding equalities.

\setcounter{case}{0}
\begin{case}[Proof of identity \eqref{eq:Markov_property_tilde_Q_with_set_1}]
Using property \eqref{eq:tilde_Q} of the probability measure $\widetilde\QQ^z$, we have that
\begin{align*}
\EE_{\widetilde\QQ^z}\left[\mathbf{1}_A \mathbf{1}_{\{\tau_{\Omega_k}>t>s\}}f\left(Z^0(t)\right)\right] 
&= \EE_{\QQ^{z,\Omega}}\left[\mathbf{1}_A \mathbf{1}_{\{\tau_{\Omega_k}>s\}}\EE_{\QQ^{z,\Omega}}\left[\mathbf{1}_{\{\tau_{\Omega_k}>t\}}f\left(Z^0(t)\right)\big{|}\cB_s\right]\right],
\end{align*}
and using the Markov property for the measure $\QQ^{z,\Omega}$ established in Proposition \ref{prop:Stopped_solutions}, it follows that
\begin{align*}
\EE_{\widetilde\QQ^z}\left[\mathbf{1}_A \mathbf{1}_{\{\tau_{\Omega_k}>t>s\}}f\left(Z^0(t)\right)\right] 
&= \EE_{\QQ^{z,\Omega}}\left[\mathbf{1}_A \mathbf{1}_{\{\tau_{\Omega_k}>s\}}\EE_{\QQ^{z,\Omega}}\left[\mathbf{1}_{\{\tau_{\Omega_k}>t\}}f\left(Z^0(t)\right)\big{|}Z^0(s)\right]\right],
\end{align*}
which immediately implies \eqref{eq:Markov_property_tilde_Q_with_set_1} from property \eqref{eq:tilde_Q} of the probability measure $\widetilde\QQ^z$.
\end{case}

\begin{case}[Proof of identity \eqref{eq:Markov_property_tilde_Q_with_set_2}]
By property \eqref{eq:tilde_Q} of the probability measure $\widetilde\QQ^z$ and the tower property of conditional expectation, we have that
\begin{align*}
\EE_{\widetilde\QQ^z}\left[\mathbf{1}_A \mathbf{1}_{\{t\geq \tau_{\Omega_k}>s\}}f\left(Z^0(t)\right)\right] 
&=\EE_{\QQ^{z,\Omega}}\left[\mathbf{1}_A \mathbf{1}_{\{\tau_{\Omega_k}>s\}}\EE_{\QQ^{z,\Omega}}\left[\EE_{\QQ^z}\left[\mathbf{1}_{\{t\geq \tau_{\Omega_k}\}}f\left(Z^0(t)\right)\big{|}\cB_{\tau_{\Omega_k}}\right]\big{|}\cB_s\right]\right].
\end{align*}
The Markov property for the measure $\QQ^{z,\Omega}$ established in Proposition \ref{prop:Stopped_solutions}, gives that
\begin{align*}
\EE_{\widetilde\QQ^z}\left[\mathbf{1}_A \mathbf{1}_{\{t\geq \tau_{\Omega_k}>s\}}f\left(Z^0(t)\right)\right] 
&=\EE_{\QQ^{z,\Omega}}\left[\mathbf{1}_A \mathbf{1}_{\{\tau_{\Omega_k}>s\}}\EE_{\QQ^{z,\Omega}}\left[\EE_{\QQ^z}\left[\mathbf{1}_{\{t\geq \tau_{\Omega_k}\}}f\left(Z^0(t)\right)\big{|}\cB_{\tau_{\Omega_k}}\right]\big{|}Z^0(s)\right]\right],
\end{align*}
and so, using again \eqref{eq:tilde_Q}, we have that
\begin{align*}
\EE_{\widetilde\QQ^z}\left[\mathbf{1}_A \mathbf{1}_{\{t\geq \tau_{\Omega_k}>s\}}f\left(Z^0(t)\right)\right] 
&=\EE_{\widetilde\QQ^z}\left[\mathbf{1}_A \mathbf{1}_{\{\tau_{\Omega_k}>s\}}\EE_{\widetilde\QQ^z}\left[\EE_{\widetilde\QQ^z}\left[\mathbf{1}_{\{t\geq \tau_{\Omega_k}\}}f\left(Z^0(t)\right)\big{|}\cB_{\tau_{\Omega_k}}\right]\big{|}Z^0(s)\right]\right].
\end{align*}
Applying again the tower property to the most inner conditional expectation, it follows that
\begin{align*}
\EE_{\widetilde\QQ^z}\left[\mathbf{1}_A \mathbf{1}_{\{t\geq \tau_{\Omega_k}>s\}}f\left(Z^0(t)\right)\right] 
&=\EE_{\widetilde\QQ^z}\left[\mathbf{1}_A \mathbf{1}_{\{\tau_{\Omega_k}>s\}}\EE_{\widetilde\QQ^z}\left[\mathbf{1}_{\{t\geq \tau_{\Omega_k}\}}f\left(Z^0(t)\right)\big{|}Z^0(s)\right]\right],
\end{align*}
that is, we obtain identity \eqref{eq:Markov_property_tilde_Q_with_set_2}.
\end{case}

\begin{case}[Proof of identity \eqref{eq:Markov_property_tilde_Q_with_set_3}]
By property \eqref{eq:tilde_Q} of the probability measure $\widetilde\QQ^z$ and the tower property of conditional expectation, we have that
\begin{align*}
\EE_{\widetilde\QQ^z}\left[\mathbf{1}_A \mathbf{1}_{\{t>s\geq \tau_{\Omega_k}\}}f\left(Z^0(t)\right)\right] 
&= \EE_{\widetilde\QQ^z}
\left[ \mathbf{1}_A \mathbf{1}_{\{s\geq \tau_{\Omega_k}\}}\EE_{\QQ^z} \left[ \EE_{\QQ^z}\left[\mathbf{1}_{\{t> \tau_{\Omega_k}\}} f\left(Z^0(t)\right)\big{|} \cB_s \right] \big{|} \cB_{\tau_{\Omega_k}} \right] \right].
\end{align*}
The Markov property for the measure $\QQ^z$ established in Proposition \ref{prop:Stopped_solutions}, gives that
\begin{align*}
\EE_{\widetilde\QQ^z}\left[\mathbf{1}_A \mathbf{1}_{\{t>s\geq \tau_{\Omega_k}\}}f\left(Z^0(t)\right)\right] 
&= \EE_{\widetilde\QQ^z}
\left[ \mathbf{1}_A \mathbf{1}_{\{s\geq \tau_{\Omega_k}\}}\EE_{\QQ^z} \left[ \EE_{\QQ^z}\left[\mathbf{1}_{\{t> \tau_{\Omega_k}\}} f\left(Z^0(t)\right)\big{|} Z^0(s) \right] \big{|} \cB_{\tau_{\Omega_k}} \right] \right],
\end{align*}
which we can rewrite using property \eqref{eq:tilde_Q} of the probability measure $\widetilde\QQ^z$,
\begin{align*}
\EE_{\widetilde\QQ^z}\left[\mathbf{1}_A \mathbf{1}_{\{t>s\geq \tau_{\Omega_k}\}}f\left(Z^0(t)\right)\right] 
&= \EE_{\widetilde\QQ^z}
\left[ \mathbf{1}_A \mathbf{1}_{\{s\geq \tau_{\Omega_k}\}}\EE_{\widetilde\QQ^z} \left[ \EE_{\widetilde\QQ^z}\left[\mathbf{1}_{\{t> \tau_{\Omega_k}\}} f\left(Z^0(t)\right)\big{|} Z^0(s) \right] \big{|} \cB_{\tau_{\Omega_k}} \right] \right],
\end{align*}
The tower property of conditional expectation yields
\begin{align*}
\EE_{\widetilde\QQ^z}\left[\mathbf{1}_A \mathbf{1}_{\{t>s\geq \tau_{\Omega_k}\}}f\left(Z^0(t)\right)\right] 
&= \EE_{\widetilde\QQ^z}
\left[ \mathbf{1}_A \mathbf{1}_{\{s\geq \tau_{\Omega_k}\}} \EE_{\widetilde\QQ^z}\left[\mathbf{1}_{\{t> \tau_{\Omega_k}\}} f\left(Z^0(t)\right)\big{|} Z^0(s) \right] \right],
\end{align*}
which is equivalent to identity \eqref{eq:Markov_property_tilde_Q_with_set_3}
\end{case}
Combining identities \eqref{eq:Markov_property_tilde_Q_with_set_1}-\eqref{eq:Markov_property_tilde_Q_with_set_3}, we obtain \eqref{eq:Markov_property_tilde_Q_with_set}, which completes the proof of Step \ref{step:Markov_property_tilde_Q}.
\end{step}

Steps \ref{step:Definition_tilde_Q}, \ref{step:Martingale_problem_tilde_Q} and \ref{step:Markov_property_tilde_Q} complete the proof.
\end{proof}

\subsection{Stochastic representation of weak solutions and the Harnack inequality}
\label{subsec:Stochastic_representation_and_Harnack}
We now give the proof of Theorem \ref{thm:Stochastic_representation} which establishes the fact that functions defined by the stochastic representation \eqref{eq:Stochastic_representation_cont} satisfy the scale-invariant Harnack inequality \eqref{eq:Harnack_cont}. Our proof relies on the connection established in Theorem \ref{thm:Connection_semigroup_Markov_solution} between the semigroup $\{T^{\Omega}_t\}_{t\geq 0}$ and the unique weak solutions to the stochastic differential equation \eqref{eq:Kimura_SDE_singular}. Let $\Omega\subset S_{n,m}$ be an open, bounded set, and let $I=(t_1,t_2)$ be an open, bounded interval in $\RR_+$. Let $\{Z(t)\}_{t \geq 0}$ be the unique weak solution satisfying the Markov property to the singular Kimura stochastic differential equation \eqref{eq:Kimura_SDE_singular}, with initial condition $Z(0)=z\in \bar S_{n,m}$, established in \cite[Theorems 3.4 and 3.8]{Pop_2013a}. Let $\PP^z$ denote the law of the process $\{Z(t)\}_{t \geq 0}$. 

We denote $Q:=(t_1,t_2)\times\Omega$, and we recall the definition of the parabolic portion of the boundary $\eth Q$ in \eqref{eq:parabolic_boundary}. For a point $(t^0,z^0)\in \RR\times\bar S_{n,m}$, and a positive constant, $r$, we let
\begin{equation}
\label{eq:Parabolic_cylinder}
Q_r(t^0,z^0):=(t^0-r^2,t^0)\times B_r(z^0),
\end{equation}
where we recall that $B_r(z^0)$ denotes the open ball centered at $z^0$ of radius $r$ with respect to the distance function $\rho$.

We can now give the  
\begin{proof}[Proof of Theorem \ref{thm:Stochastic_representation}]
Because we assume that the boundary function $g$ is nonnegative and continuous on $\overline{\eth Q}$, there is a sequence of nonnegative, smooth functions, $\{g_k\}_{k\geq 0}\subset C^{\infty}(\bar Q)$, such that
\begin{equation}
\label{eq:Approximate_boundary_sequence_convergence}
\|g_k-g\|_{C(\overline{\eth Q})} \rightarrow 0,\quad\hbox{as } k \rightarrow \infty.
\end{equation}
Analogously to the representation \eqref{eq:Stochastic_representation_cont}, we define the sequence of functions, $\{u_k\}_{k\geq 0}$, by letting
\begin{equation}
\label{eq:Approximate_stochastic_representation}
u_k(t,z):=\EE_{\PP^z}\left[g_k(t-(t-t_1)\wedge\tau_{\Omega}, Z((t-t_1)\wedge\tau_{\Omega}))\right],\quad\forall\, (t,z) \in \bar Q,\quad\forall\, k\geq 0.
\end{equation}
Using property \eqref{eq:Approximate_boundary_sequence_convergence}, it immediately follows from the preceding equality that
\begin{equation}
\label{eq:Approximate_sequence_convergence}
\|u_k-u\|_{L^{\infty}(\bar Q)} \rightarrow 0,\quad\hbox{as } k \rightarrow \infty,
\end{equation}
and so, it is sufficient to prove that the Harnack inequality, \eqref{eq:Harnack_cont}, holds for each nonnegative function $u_k$, in order to deduce that the Harnack inequality also holds for the function $u$ defined by \eqref{eq:Stochastic_representation_cont}. We now prove that each nonnegative function $u_k$ is a local weak solution to the equation $u_t-Lu=0$ on $Q$. Therefore, it follows form \cite[Theorem 4.1]{Epstein_Mazzeo_cont_est} that the Harnack inequality holds for $u_k$. Let 
$$
h_k(t,z):=\partial_t g_k(t,z)-Lg_k(t,z),\quad\forall\, (t,z)\in\bar Q,\quad\forall\, k\in\NN.
$$
We fix $(t,z) \in\bar Q$, and using the fact that $g_k\in C^{\infty}(\bar Q)$, we may apply It\^o's rule to the process $\{g_k(t-r, Z(r))\}_{0\leq r \leq t-t_1}$. We obtain that
\begin{align*}
d g_k(t-r, Z(r)) &= \left[-\partial_t g_k(t-r, Z(r)) + Lg_k(t-r, Z(r))\right]\, dr + dM(r)\\
&= -h_k(t-r, Z(r)) \, dr + dM(r),
\end{align*}
where the process $\{M(r\wedge\tau_{\Omega})\}_{r \geq 0}$ is a martingale. Stopping the process at time $(t-t_1)\wedge\tau_{\Omega}$, we obtain that
\begin{align*}
g_k(t-(t-t_1)\wedge\tau_{\Omega}, Z((t-t_1)\wedge\tau_{\Omega}))
&= g_k(t,z)	-\int_0^{(t-t_1)\wedge\tau_{\Omega}} h_k(t-r, Z(r)) \, dr\\
&\quad + M((t-t_1)\wedge\tau_{\Omega}),
\end{align*}
from which it follows that
\begin{align*}
g_k(t,z) &= \EE_{\PP^z}\left[g_k(t-(t-t_1)\wedge\tau_{\Omega}, Z((t-t_1)\wedge\tau_{\Omega}))\right]	
+ \EE_{\PP^z} \left[\int_0^{(t-t_1)\wedge\tau_{\Omega}} h_k(t-r, Z(r)) \, dr\right].
\end{align*}
Using our definition \eqref{eq:Approximate_stochastic_representation} of the function $u_k(t,z)$, the preceding expression implies that
\begin{equation}
\label{eq:Decomposition_u_k}
u_k(t,z)= g_k(t,z) - \EE_{\PP^z} \left[\int_0^{(t-t_1)\wedge\tau_{\Omega}} h_k(t-r, Z(r)) \, dr\right].
\end{equation}
The integral term on the right-hand side of the preceding expression can be written as
$$
\EE_{\PP^z} \left[\int_0^{(t-t_1)\wedge\tau_{\Omega}} h_k(t-r, Z(r)) \, dr\right] = \EE_{\PP^z} \left[\int_0^{t-t_1} h_k(t-r, Z(r)) \mathbf{1}_{\{r < \tau_{\Omega}\}} \, dr\right],
$$
and using identity \eqref{eq:Connection_semigroup_Markov_process}, is follows that
\begin{align*}
\EE_{\PP^z} \left[\int_0^{(t-t_1)\wedge\tau_{\Omega}} h_k(t-r, Z(r)) \, dr\right] 
&= \int_0^{t-t_1} \left(T^{\Omega}_{r} h_k\right)(t-r,\cdot) (z)\, dr\\
&= \int_0^{t-t_1} \left(T^{\Omega}_{t-(r+t_1)} h_k\right)(r+t_1,\cdot) (z)\, dr.
\end{align*}
Identity \eqref{eq:Decomposition_u_k} together with the preceding equality, gives us that
\begin{equation*}
u_k(t,z)= g_k(t,z) - \int_0^{t-t_1} \left(T^{\Omega}_{t-(r+t_1)} h_k\right)(r+t_1,\cdot) (z)\, dr,\quad\forall\, (t,z) \in Q.
\end{equation*}
Because the function $g_k$ belongs to $C^{\infty}(\bar Q)$, it is clear that $g_k$ solves the equation $u_t-Lu=h_k$ on $Q$. Using the fact that $h_k \in L^2((t_1,t_2), L^2(\Omega;d\mu))$, it follows by Lemma \ref{lem:Existence_uniqueness_inhomogeneous} that the integral term on the right-hand side of the preceding identity, is a weak solution to the inhomogeneous problem $u_t-Lu=-h_k$ on $Q$. Thus, the function $u_k$ is a weak solution to the homogeneous problem $u_t-Lu=0$ on $Q$, and applying \cite[Theorem 4.1]{Epstein_Mazzeo_cont_est}, we obtain that $u_k$ satisfies the Harnack inequality. Letting $k$ tend to $\infty$ in \eqref{eq:Approximate_stochastic_representation} and using \eqref{eq:Approximate_sequence_convergence}, it follows that the function $u$ defined by \eqref{eq:Stochastic_representation_cont} also satisfies the Harnack inequality \eqref{eq:Harnack_cont}. This completes the proof.
\end{proof}

We have the following corollary of Theorem \ref{thm:Stochastic_representation}. This is a technical result needed in the proof of Lemma \ref{lem:Iterated_Harnack}.

\begin{cor}
\label{cor:Stochastic_representation}
There is a positive constant, $K_0$, such that the following hold. Let $T>t_1$ and let $g \in C(\overline{\eth Q})$ be a nonnegative function, and let
\begin{equation}
\label{eq:Stochastic_representation}
u(t,z):=\EE_{\PP^z}\left[g(t-(t-t_1)\wedge\tau_{\Omega}, Z((t-t_1)\wedge\tau_{\Omega}))\mathbf{1}_{\{(t-t_1)\wedge\tau_{\Omega} < T-t_1\}}\right],\quad\forall\, (t,z) \in \bar Q.
\end{equation}
Then the function $u$ satisfies the scale-invariant Harnack inequality, that is, for all $(t^0,z^0)\in\bar Q$ and $r>0$ such that $Q_{2r}(t^0,z^0) \subset Q$, we have that the scale-invariant Harnack inequality \eqref{eq:Harnack_cont} holds.
\end{cor}

\begin{proof}
Similarly to the proof of Theorem \ref{thm:Stochastic_representation}, we let $\{g_k\}_{k\geq 0}\subset C^{\infty}(\bar Q)$ be a sequence of nonnegative, smooth functions, such that 
\begin{equation}
\label{eq:Approximate_boundary_sequence_convergence_1}
\begin{aligned}
\|g-g_k\|_{C(\bar Q\cap\{t < T\})} \rightarrow 0,\quad\hbox{as } k\rightarrow\infty,\\
\|g_k\|_{C(\bar Q\cap\{t > T\})} \rightarrow 0,\quad\hbox{as } k\rightarrow\infty.
\end{aligned}
\end{equation}
Theorem \ref{thm:Stochastic_representation} yields that the sequence of functions $\{u_k\}_{k\geq 0}$ defined by \eqref{eq:Approximate_stochastic_representation} satisfies the Harnack inequality \eqref{eq:Harnack_cont}, and it is sufficient to prove that property \eqref{eq:Approximate_sequence_convergence} holds in order to conclude that the function $u$ defined by \eqref{eq:Stochastic_representation} also satisfies the Harnack inequality \eqref{eq:Harnack_cont}. 

From definition \eqref{eq:Stochastic_representation} of the function $u(t,z)$ and property \eqref{eq:Approximate_boundary_sequence_convergence_1}, it follows that there is a positive constant, $C$, such that for all $(t,z)\in \bar Q$, we have that
\begin{equation}
\label{eq:Convergence_1}
|u(t,z)-u_k(t,z)| \leq \|g-g_k\|_{C(\bar Q\cap\{t < T\})} + \|g_k\|_{C(\bar Q\cap\{t > T\})} + C \PP^z((t-t_1)\wedge\tau_{\Omega}=T-t_1).
\end{equation}
We can assume without loss of generality that $t\geq T$. Let $v\in \cF((0,t-T+1)\times\Omega)$ be the unique weak solution to the homogeneous initial-value problem \eqref{eq:Homogeneous_initial_value_problem}, with initial condition $v(0,\cdot)\equiv 1$, given by Lemma \ref{lem:Existence_uniqueness_homogeneous}. Then Theorem \ref{thm:Connection_semigroup_Markov_solution} gives us that
$$
v(s,z) =\PP^z(\tau_{\Omega}>s),\quad\forall\, (s,z)\in (0,t-T+1)\times\underline\Omega.
$$
Using the fact that
\begin{align*}
\PP^z(\tau_{\Omega}=s) &=\lim_{\eps\downarrow 0} \PP^z(\tau_{\Omega}>s-\eps)-\PP^z(\tau_{\Omega}>s)\\
&=v(s-\eps,z)-v(s,z),
\end{align*}
and that the function $v$ is continuous by \cite[Corollary 4.1]{Epstein_Mazzeo_cont_est}, it follows that
$$
\PP^z(\tau_{\Omega}=s)=0,\quad\forall\, (s,z)\in (0,t-T+1)\times\underline\Omega.
$$
Thus, the preceding inequality together with \eqref{eq:Approximate_boundary_sequence_convergence_1} and \eqref{eq:Convergence_1} yield that property \eqref{eq:Approximate_sequence_convergence} holds. We can now conclude that the function $u$ defined in \eqref{eq:Stochastic_representation} satisfies the Harnack inequality \eqref{eq:Harnack_cont}, since each element of the sequence $\{u_k\}_{k\in\NN}$ also satisfy \eqref{eq:Harnack_cont}. This completes the proof.
\end{proof}

\section{Harnack's inequality for standard Kimura operators}
\label{sec:Harnack_inequality}
In this section, we consider the standard Kimura differential operator $\widehat L$ defined in \eqref{eq:Generator}. In Theorem \ref{thm:Harnack_inequality}, we prove the Harnack inequality for nonnegative solutions to the parabolic equation 
\begin{equation}
\label{eq:Parabolic_hat_L}
u_t-\widehat Lu=0,
\end{equation}
by relying on Theorem \ref{thm:Stochastic_representation} and adapting the argument used to prove \cite[Theorem 1]{Sturm_1994}. The solutions that we consider are probabilistic, in the sense that they admit a suitable stochastic representation, \eqref{eq:Probabilistic_solution}. We organize this section into three parts. In \S \ref{subsec:Properties_sol_hat_L}, we describe the assumptions imposed on the coefficients of the differential operator $\widehat L$, and we review some properties of the solutions to the parabolic equation \eqref{eq:Parabolic_hat_L}. In \S \ref{subsec:Connection_L__hat_L}, we establish the connection between the differential operators $L$ and $\widehat L$, and we introduce the notion of a probabilistic solution in Definition \ref{defn:Probabilistic_solution}, which we then use in \S \ref{subsec:Proof_of_Harnack} to give the proof of Harnack's inequality in Theorem \ref{thm:Harnack_inequality}, and of the scale-invariant Harnack inequality in Theorem \ref{thm:Harnack_inequality_scale_invariant}.

\subsection{Properties of solutions to the parabolic equation $u_t-\widehat L u=0$}
\label{subsec:Properties_sol_hat_L} 
We first introduce Assumption \ref{assump:Coeff_hat_L}, which describe the conditions that we impose on the coefficients of the standard Kimura differential operator $\widehat L$. We then review the existence, uniqueness and regularity of solutions in anisotropic H\"older spaces to the inhomogeneous initial-value problem defined by the operator $\widehat L$,
\begin{equation}
\label{eq:Initial_value_problem_widehat_L}
\begin{aligned}
\left\{
\begin{array}{rll}
u_t-\widehat Lu&=g&\quad\hbox{on } (0,\infty)\times S_{n,m},\\
u(0,\cdot) &=f&\quad\hbox{on } S_{n,m},
\end{array} \right.	
\end{aligned}
\end{equation}
obtained in \cite{Epstein_Mazzeo_cont_est_diag, Pop_2013b}. In Lemma \ref{lem:Stochastic_representation_widehat_L_local}, we establish the stochastic representation of the solutions in anisotropic H\"older spaces to problem \eqref{eq:Initial_value_problem_widehat_L}.

We begin with
\begin{assump}[Coefficients of the operator $\widehat L$]
\label{assump:Coeff_hat_L}
There are positive constants, $\bar b$ and $\delta$, such that the following hold:
\begin{enumerate}
\item[1.] The coefficient matrix functions $(\widehat a(z))$, $(\widehat c(z))$ and $(\widehat d(z))$ are chosen such that, for all $z \in \bar M_{\{1,\ldots,n\}}$, $\xi\in\RR^n$ and $\eta\in\RR^m$, we have that 
\begin{equation}
\label{eq:Ellipticity_all_second_order_coeff_hat_L}
\begin{aligned}
&\delta\left(|\xi|^2+|\eta|^2\right)\\
&\leq \sum_{i=1}^n \xi_i^2+\sum_{i,j=1}^n \sqrt{x_ix_j}\widehat a_{ij}(z)\xi_i\xi_j + 8\sum_{i=1}^n\sum_{l=1}^m \sqrt{x_i}\widehat c_{il}(z)\xi_i\eta_l +4\sum_{l,k=1}^m \widehat d_{lk}(z)\eta_l\eta_k\\
&\leq K\left(|\xi|^2+|\eta|^2\right).
\end{aligned}
\end{equation}
Compare condition \eqref{eq:Ellipticity_all_second_order_coeff_hat_L} with \cite[Condition (33) and Definition (1)]{Epstein_Mazzeo_cont_est}.
\item[2.] Let $I\subsetneqq\{1,\ldots,n\}$. For all $z \in \bar M_{I}$, we assume that
\begin{equation}
\label{eq:Outside_neighborhood_O_hat_L}
\begin{aligned}
x_j \widehat a_{jj}(z)&=1,\quad \forall\, j \in I^c,\\
\widehat a_{ij}(z) &=0,\quad\forall\, i,j=1,\ldots,n,\\
\widehat c_{il}(z) &= 0,\quad\forall\, i=1,\ldots,n,\quad\forall\, l=1,\ldots,m,\\
\widehat d_{kl}(z) &= \delta_{lk}, \quad\forall\, k,l=1,\ldots,m.
\end{aligned}
\end{equation}
\item[3.] Let $I\subsetneqq\{1,\ldots,n\}$. For all $z \in \bar M_{I}$, we assume that
\begin{equation}
\label{eq:Outside_neighborhood_O_a_ii_hat_L}
\begin{aligned}
\widehat a_{ii}(z) \geq \delta,\quad \forall\, i \in I.
\end{aligned}
\end{equation}
\item[4.] The coefficients functions $\widehat b_i(z)$ satisfy 
\begin{align}
\label{eq:b_coeff_bound_lower_hat_L}
\widehat b_{i}(z) \geq \bar b,&\quad \forall\, z\in \partial S_{n,m}\cap\{x_i=0\},\quad\forall\, i=1,\ldots,n,\\
\label{eq:b_coeff_away_O_hat_L}
\widehat b_i(z)=1,&\quad\forall z\in M^c_{\{1,\ldots,n\}}, \quad \forall\, i=1,\ldots,n.
\end{align}
\item[5.] The coefficient functions $(\widehat a(z))$, $\widehat b(z)$, $(\widehat c(z))$, $(\widehat d(z))$ are smooth, bounded functions on $\bar S_{n,m}$, and we have that
\begin{equation}
\label{eq:Upper_bound_coeff_widehat}
\begin{aligned}
&\sum_{i,j=1}^n\|\widehat a_{ij}\|_{C(\bar S_{n,m})} + \sum_{i=1}^n\sum_{l=1}^m\|\widehat c_{il}\|_{C(\bar S_{n,m})}
+\sum_{l,k=1}^m\|\widehat d_{lk}\|_{C(\bar S_{n,m})}\\
&+\sum_{i=1}^n\|\widehat b_i\|_{C(\bar S_{n,m})} + \sum_{l=1}^m\|\widehat e_l\|_{C(\bar S_{n,m})} \leq K.
\end{aligned}
\end{equation}
\end{enumerate}
\end{assump}
The stochastic differential equations associated to the standard Kimura diffusion operator $\widehat L$ can be written in the form,
\begin{equation}
\label{eq:Kimura_SDE}
\begin{aligned}
d\widehat X_i(t) &= \widehat b_i(\widehat Z(t))\, dt+\sqrt{\widehat X_i(t)}\sum_{j=1}^{n+m} \widehat \sigma_{ij}(\widehat Z(t))\, d\widehat W_j(t),\quad\forall\, t>0,\\
d\widehat Y_l(t) &= \widehat e_l(\widehat Z(t))\, dt+\sum_{j=1}^{n+m} \widehat \sigma_{l+n,j}(\widehat Z(t))\, d\widehat W_j(t),\quad\forall\, t>0,
\end{aligned}
\end{equation}
where $i=1,\ldots,n$ and $l=1,\ldots,m$. We denote $\widehat Z(t)=(\widehat X(t),\widehat Y(t))$. We choose the diffusion matrix $(\widehat\sigma(z))$ appearing in \eqref{eq:Kimura_SDE}, such that $(\widehat\sigma\widehat\sigma^*)(z)=\widehat D(z)$, where for all $z\in \bar S_{n,m}$, $i,j=1,\ldots,n$, and $l,k=1,\ldots,m$, we let
\begin{equation}
\label{eq:widehat_D}
\begin{aligned}
\widehat D_{ii}(z) &:=2(1+x_i\widehat a_{ii}(z)),\\
\widehat D_{ij}(z) &:=2\sqrt{x_ix_j}\widehat a_{ii}(z),\quad i\neq j,\\
\widehat D_{n+l,i}(z)=\widehat D_{i,n+l}(z) &:=4\sqrt{x_i}\widehat c_{il}(z),\\
\widehat D_{n+l,n+k}(z) &:=2\widehat d_{lk}(z).
\end{aligned}
\end{equation}
Assumption \ref{assump:Coeff_hat_L} allows us to apply the same analysis used to build a solution $(\sigma(z))$ to equation \eqref{eq:Relation_sigma_D} that is smooth in the $(\sqrt{x},y)$ variables (defined in \eqref{eq:Sqrt_x}), to conclude that there is a solution $(\widehat\sigma(z))$ to the equation $(\widehat\sigma\widehat\sigma^*)(z)=\widehat D(z)$ that is also smooth in the $(\sqrt{x},y)$ variables. We recall that the solution $(\widehat\sigma(z))$ is uniquely determined up to orthogonal transformations, by \cite[Problem 5.4.7]{KaratzasShreve1991}. It follows that the coefficients of the stochastic differential equation \eqref{eq:Kimura_SDE} satisfy \cite[Assumptions 2.1 and 2.4]{Pop_2013a}. We may then apply \cite[Propositions 2.2 and 2.8]{Pop_2013a}, to conclude that the standard Kimura stochastic differential equation \eqref{eq:Kimura_SDE} has a unique weak solution, $(\widehat Z(t))_{t \geq 0}$, on a probability space, $(\widehat\Xi, \widehat\PP^z)$, for any initial condition $\widehat Z(0)=z$, where $z \in\bar S_{n,m}$. 

In order to prove the stochastic representation of solutions to the initial-value problem \eqref{eq:Initial_value_problem_widehat_L}, we first recall the existence, uniqueness and regularity of solutions in anisotropic H\"older spaces to equation \eqref{eq:Initial_value_problem_widehat_L} obtained in \cite{Pop_2013b}. We remark that such results are also established in \cite[Theorem 1.1]{Epstein_Mazzeo_cont_est}, \cite[Proposition 2.1]{Epstein_Mazzeo_cont_est_diag} and \cite[Theorem 10.0.2]{Epstein_Mazzeo_annmathstudies}, but the framework in \cite{Pop_2013b} is closer to our article.

Following \cite[\S 5.2.4]{Epstein_Mazzeo_annmathstudies}, we first introduce suitable anisotropic H\"older spaces. Let $\alpha\in (0,1)$, $k\in\NN$, $T>0$ and $U\subseteq S_{n,m}$.  We let $C^{0,\alpha}_{WF}([0,T]\times \bar U)$ be the H\"older space consisting of continuous functions, $u:[0,T]\times\bar U\rightarrow \RR$, such that the following norm is finite
\begin{equation}
\label{eq:Anisotropic_Holder_space}
\|u\|_{C^{0,\alpha}_{WF}([0,T]\times \bar U)} := \|u\|_{C^0([0,T]\times\bar U)} 
+ \sup_{\stackrel{(t^0,z^0), (t,z)\in [0,T]\times \bar U}{(t^0,z^0)\neq (t,z)}} 
\frac{|u(t^0,z^0)-u(t,z)|}{\rho^{\alpha}((t^0,z^0), (t,z))}. 
\end{equation}
We let $C^{k,\alpha}_{WF}([0,T]\times \bar U)$ denote the H\"older space containing functions, $u\in C^k([0,T]\times\bar U)$, such that the derivatives $D^{\tau}_t D^{\zeta}_z$ belong to the space $C^{0,\alpha}_{WF}([0,T]\times \bar U)$, for all $\tau\in\NN$ and $\zeta\in\NN^{n+m}$, such that $2\tau+|\zeta| \leq k$. We endow the space of functions $C^{k,\alpha}_{WF}([0,T]\times \bar U)$ with the norm,
\begin{align*}
\|u\|_{C^{k,\alpha}_{WF}([0,T]\times \bar U)} &:= \sum_{\stackrel{\tau\in\NN, \zeta\in\NN^{n+m}}{2\tau+|\zeta|\leq k}}
 \|D^{\tau}_t D^{\zeta}_z\|_{C^{0,\alpha}_{WF}([0,T]\times \bar U)}.
\end{align*}
We fix a set of indices, $I\subseteq \{1,\ldots, n\}$. Let $U$ be a set such that $U \subseteq M_{I}$. We let $C^{0,2+\alpha}_{WF}([0,T]\times \bar U)$ denote the H\"older space of functions, $u\in C^{1,\alpha}_{WF}([0,T]\times\bar U)\cap C^2([0,T]\times U)$, such that 
$$
u_t \in C^{0,\alpha}_{WF}([0,T]\times \bar U),
$$
and such that the functions,
\begin{align*}
\sqrt{x_ix_j}u_{x_ix_j}, \sqrt{x_i}u_{x_iy_l}, u_{y_ly_k} &\in C^{0,\alpha}_{WF}([0,T]\times \bar U),
\quad \forall\, i,j\in I,\quad \forall\, l,k=1,\ldots,m,\\
\sqrt{x_i}u_{x_i x_j}, u_{x_jx_k} &\in C^{0,\alpha}_{WF}([0,T]\times \bar U),\quad\forall\, i \in I, \quad\forall\, j,k \in I^c.
\end{align*}
We endow the space $C^{0,2+\alpha}_{WF}([0,T]\times \bar U)$ with the norm,
\begin{align*}
\|u\|_{C^{0,2+\alpha}_{WF}([0,T]\times \bar U)} &:= \|u\|_{C^{1,\alpha}_{WF}([0,T]\times \bar U)} 
+ \sum_{i,j \in I}\|\sqrt{x_ix_j}u_{x_ix_j}\|_{C^{0,\alpha}_{WF}([0,T]\times \bar U)}\\
&\quad+ \sum_{l,k=1}^m\|u_{y_ly_k}\|_{C^{0,\alpha}_{WF}([0,T]\times \bar U)}
+ \sum_{i \in I} \sum_{j \in I^c}\|\sqrt{x_i}u_{x_ix_j}\|_{C^{0,\alpha}_{WF}([0,T]\times \bar U)}\\
&\quad+ \sum_{i \in I} \sum_{l=1}^m\|\sqrt{x_i}u_{x_iy_l}\|_{C^{0,\alpha}_{WF}([0,T]\times \bar U)}
+ \sum_{i, j \in I^c} \|u_{x_ix_j}\|_{C^{0,\alpha}_{WF}([0,T]\times \bar U)}\\
&\quad+ \sum_{i \in I^c} \sum_{l=1}^m \|u_{x_iy_l}\|_{C^{0,\alpha}_{WF}([0,T]\times \bar U)}
+\|u_t\|_{C^{0,\alpha}_{WF}([0,T]\times \bar U)} .
\end{align*}
We now consider the case when $U$ is an arbitrary set in $S_{n,m}$. We let $C^{0,2+\alpha}_{WF}([0,T]\times \bar U)$ denote the H\"older space consisting of functions $u\in C^2([0,T]\times U)$, satisfying the property that 
$$
u\upharpoonright_{\bar U\cap \bar M_{I}}\in C^{0,2+\alpha}_{WF}([0,T]\times(\bar U \cap \bar M_I)),\quad\forall\, I\subseteq\{1,\ldots,n\}.
$$
We endow the H\"older space $C^{0,2+\alpha}_{WF}([0,T]\times \bar U)$ with the norm
$$
\|u\|_{C^{0,2+\alpha}_{WF}([0,T]\times \bar U)} = \sum_{I \subseteq \{1,\ldots,n\}} \|u\|_{C^{0,2+\alpha}_{WF}([0,T]\times(\bar U \cap \bar M_I))}.
$$
We let $C^{k,2+\alpha}_{WF}([0,T]\times \bar U)$ be the space of functions $u\in C^k([0,T]\times U)$, satisfying the property that
$$
D^{\tau}_t D^{\zeta}_z u \in C^{0,2+\alpha}_{WF}([0,T]\times \bar U),\quad\forall\, \tau\in \NN,\forall\, \zeta\in\NN^{n+m}\hbox{ such that } 
2\tau+|\zeta| \leq k,
$$
and we endow it with the norm
\begin{align*}
\|u\|_{C^{k,2+\alpha}_{WF}([0,T]\times \bar U)} &:= \sum_{\stackrel{\tau\in \NN, \zeta\in\NN^{n+m}}{2\tau+|\zeta| \leq k}}
 \|D^{\tau}_t D^{\zeta}_z u\|_{C^{0,2+\alpha}_{WF}([0,T]\times \bar U)}.
\end{align*}
When $k=0$, we write for brevity $C([0,T]\times\bar U)$, $C^{\alpha}([0,T]\times \bar U)$, $C^{\alpha}_{WF}([0,T]\times \bar U)$ and $C^{2+\alpha}_{WF}([0,T]\times \bar U)$, instead of $C^0([0,T]\times\bar U)$, $C^{0,\alpha}([0,T]\times \bar U)$, $C^{0,\alpha}_{WF}([0,T]\times \bar U)$ and $C^{0,2+\alpha}_{WF}([0,T]\times \bar U)$.

The elliptic H\"older spaces $C^{k,\alpha}(\bar U)$, $C^{k,\alpha}_{WF}(\bar U)$ and $C^{k,2+\alpha}_{WF}(\bar U)$ are defined analogously to their parabolic counterparts, and so, we omit their definitions for brevity.

From Assumption \ref{assump:Coeff_hat_L}, it follows that the coefficients of the differential operator \eqref{eq:Generator} satisfy the hypotheses of \cite[Theorem 1.4]{Pop_2013b}. Thus, given $f\in C^{k_0,2+\alpha}_{WF}(\bar S_{n,m})$ and $g\in C^{k_0,\alpha}_{WF}([0,\infty)\times\bar S_{n,m})$, the inhomogeneous initial-value problem \eqref{eq:Initial_value_problem_widehat_L} has a unique solution, $u\in C^{k_0,2+\alpha}([0,\infty)\times\bar S_{n,m})\cap C^{\infty}((0,\infty)\times\bar S_{n,m})$. We can now prove

\begin{lem}[Stochastic representation of solutions to equation \eqref{eq:Initial_value_problem_widehat_L} with respect to $\widehat\PP^z$]
\label{lem:Stochastic_representation_widehat_L_local}
Let $\alpha\in (0,1)$, $k_0\in\NN$, $f\in C^{k_0,2+\alpha}_{WF}(\bar S_{n,m})$ and $g\in C^{k_0,\alpha}_{WF}([0,\infty)\times\bar S_{n,m})$. Let $u\in C^{k_0,2+\alpha}([0,\infty)\times\bar S_{n,m})\cap C^{\infty}((0,\infty)\times\bar S_{n,m})$ be the unique solution to the inhomogeneous initial-value problem \eqref{eq:Initial_value_problem_widehat_L}. Let $\Omega\subset S_{n,m}$ be an open set, $I=(t_1,t_2)\subset \RR_+$ be a bounded interval, and $Q:=I\times\Omega$. Then we have that
\begin{equation}
\label{eq:Stochastic_representation_widehat_L_local}
\begin{aligned}
u(t,z) 
&= \EE_{\widehat \PP^z} \left[u\left(t-(t-t_1)\wedge\widehat\tau_{\Omega},\widehat Z((t-t_1)\wedge\widehat\tau_{\Omega})\right)\right]\\
&\quad+\EE_{\widehat\PP^z}\left[\int_0^{(t-t_1)\wedge\widehat\tau_{\Omega}} g(t-s,\widehat Z(s))\, ds\right],
\quad\forall\, (t,z) \in \bar Q,
\end{aligned}
\end{equation}
where the stopping time $\widehat \tau_{\Omega}$ is defined by
\begin{equation}
\label{eq:hat_tau_Omega}
\widehat \tau_{\Omega} := \inf\{t \geq 0:\, \widehat Z(t) \notin\underline \Omega\},
\end{equation}
and the process $\{\widehat Z(t)\}_{t\geq 0}$ is the unique weak solution to the standard Kimura stochastic differential equation \eqref{eq:Kimura_SDE}, with initial condition $\widehat Z(0)=z$.
\end{lem}

\begin{proof}
Let $(t,z)\in [0,\infty)\times\bar S_{n,m}$, and $\eps\in (0,t)$. Because the function $u$ is smooth on $[\eps,t]\times\bar S_{n,m}$, we can apply It\^o's rule to the process $\{u(t-r,\widehat Z(r))\}_{0\leq r\leq t-\eps}$, and use the fact that the function $u$ solves the inhomogeneous initial-value problem \eqref{eq:Initial_value_problem_widehat_L}, to conclude that the process
$$
M_r:=u(t-r,\widehat Z(r))-u(t,z) +\int_0^r g(t-s,\widehat Z(s))\, ds,\quad\forall\, r\in [0,t-\eps],
$$
is a martingale. Thus, using the fact that $M_0=0$, and stopping the martingale at $(t-t_1\vee \eps)\wedge\widehat\tau_{\Omega}$, we have that 
$$
\EE_{\widehat\PP^z}\left[u(t-(t-t_1\vee \eps)\wedge\widehat\tau_{\Omega},\widehat Z((t-t_1\vee \eps)\wedge\widehat\tau_{\Omega}))-u(t,z) +\int_0^{(t-t_1\vee \eps)\wedge\widehat\tau_{\Omega}} g(t-s,\widehat Z(s))\, ds\right]=0,
$$
which gives us, for all $\eps\in (0,t)$, that
\begin{align*}
u(t,z) &= \EE_{\widehat\PP^z}\left[u(t-(t-t_1\vee \eps)\wedge\widehat\tau_{\Omega},\widehat Z((t-t_1\vee \eps)\wedge\widehat\tau_{\Omega}))\right]\\
&\quad+ \EE_{\widehat\PP^z}\left[\int_0^{(t-t_1\vee \eps)\wedge\widehat\tau_{\Omega}} g(t-s,\widehat Z(s))\, ds\right].
\end{align*}
Using the fact that the functions $u$ and $g$ are continuous on $[0,t]\times\bar S_{n,m}$, and that the process $\{\widehat Z(t)\}_{t\geq 0}$ has continuous paths, we can take the limit as $\eps$ tends to $0$ in the preceding identity to obtain \eqref{eq:Stochastic_representation_widehat_L_local}. This completes the proof of Lemma \ref{lem:Stochastic_representation_widehat_L_local}.
\end{proof}

\subsection{Connection between the differential operators $L$ and $\widehat L$}
\label{subsec:Connection_L__hat_L} 
Lemma \ref{lem:Stochastic_representation_widehat_L_local} shows that the homogeneous initial-value problem \eqref{eq:Initial_value_problem_widehat_L} admits solutions that can be expressed using the probability distribution, $\widehat\PP^z$, of the unique weak solution, $\{\widehat Z(t)\}$, to the standard Kimura stochastic differential equation equation \eqref{eq:Kimura_SDE}. In this section, using Girsanov's Theorem, we prove in Lemma \ref{lem:Stochastic_representation_L_local} that the solutions to the homogeneous initial-value problem \eqref{eq:Initial_value_problem_widehat_L} have a stochastic representation which uses the probability distribution, $\PP^z$, of the unique Markovian solution to a suitable Kimura stochastic differential equation with singular drift of the form \eqref{eq:Kimura_SDE_singular}, as opposed to that of weak solutions to the stochastic differential equation \eqref{eq:Kimura_SDE}. This shows that our Definition \ref{defn:Probabilistic_solution} of probabilistic solutions is not vacuous. In \S \ref{subsec:Proof_of_Harnack}, we use Definition \ref{defn:Probabilistic_solution} to prove in Theorem \ref{thm:Harnack_inequality} that the Harnack inequality holds for nonnegative probabilistic solutions to equation \eqref{eq:Parabolic_hat_L}. 

We make a specific choice of the differential operator $L$ of the form given by \eqref{eq:Singular_operator}. We define the coefficients of the operator $L$ in terms of the coefficients of the operator $\widehat L$ such that for all $i,j=1,\ldots,n$ and all $l,k=1,\ldots,m$, we have
\begin{align*}
a_{ii} (z)&:= 1,\\
\tilde a_{ij} (z) &:= \widehat a_{ij} (z),\\
c_{il} (z) &:= \frac{1}{2}\widehat c_{il} (z),\\
d_{lk} (z) &:= \widehat d_{lk} (z), 
\end{align*}
and we choose the coefficients $b_i(z)$ in \eqref{eq:Singular_operator}, such that 
$$
g_i(z) = \widehat b_i(z),
$$
where the coefficients $g_i(z)$ are defined in \eqref{eq:Coeff_Kimura_SDE_L}. With this choice of the coefficients, the stochastic differential equation \eqref{eq:Kimura_SDE_singular} becomes, for all $i=1,\ldots,n$ and all $l=1,\ldots,m$,
\begin{equation}
\label{eq:Kimura_SDE_singular_new}
\begin{aligned}
d X_i(t) &= \left(\widehat b_i(Z(t))+ X_i(t)\sum_{j=1}^{n}f_{ij}( Z(t))\ln  X_j(t)\right)\, dt+\sqrt{ X_i(t)}\sum_{j=1}^{n+m} \widehat\sigma_{ij}( Z(t))\, d W_j(t),\\
d Y_l(t) &= \left(e_l( Z(t))+\sum_{j=1}^{n}f_{l+n,j}( Z(t))\ln  X_j(t)\right)\, dt+\sum_{j=1}^{n+m} \widehat\sigma_{l+n,j}( Z(t))\, d W_j(t),
\end{aligned}
\end{equation}
Because the coefficients of the differential operator \eqref{eq:Generator} satisfy Assumption \ref{assump:Coeff_hat_L}, it follows that the preceding choice of the coefficients of the stochastic differential equation \eqref{eq:Kimura_SDE} satisfy \cite[Assumption 3.1]{Pop_2013a}. We may then apply \cite[Theorems 3.4 and 3.8]{Pop_2013a}, to conclude that the Kimura stochastic differential equation with logarithmic drift \eqref{eq:Kimura_SDE_singular_new} has a unique Markov solution, $(Z(t))_{t \geq 0}$, on a probability space, $(\Xi,\PP^z)$, for any initial condition $Z(0)=z$, where $z \in\bar S_{n,m}$. 

From the strict ellipticity condition \eqref{eq:Ellipticity_all_second_order_coeff_hat_L}, it follows from definition \eqref{eq:widehat_D} that the diffusion matrix $(\widehat D(z))$ is strictly elliptic, and so, using the fact that $(\widehat\sigma\widehat\sigma^*)(z)=\widehat D(z)$, we obtain that the matrix $\widehat\sigma(z)$ is invertible. Let then $\theta(z):=(\theta_1(z),\ldots,\theta_{n+m}(z))$ be the unique solution to the system of linear equations,
\begin{align*}
\left\{
\begin{array}{rll}
\sum_{k=1}^{n+m} \widehat\sigma_{ik} (z) \theta_k(z) &=\sqrt{x_i}\sum_{j=1}^n f_{ij}(z) \ln x_j,&\quad\forall\, i=1,\ldots,n\\
\sum_{k=1}^{n+m} \widehat\sigma_{n+l,k} (z) \theta_k(z) &=\sum_{j=1}^n f_{n+l,j}(z) \ln x_j+\widehat e_l(z)- e_l(z),&\quad\forall\, l=1,\ldots,m,
\end{array} \right.	
\end{align*}
for all $z \in \bar S_{n,m}$. It follows from \cite[Lemma 3.5]{Pop_2013a} that
\begin{lem}
\label{lem:Exp_Girsanov_theorem}
Suppose that the coefficients of the differential operator \eqref{eq:Generator} satisfy Assumption \ref{assump:Coeff_hat_L}. Then for all $T>0$, there is a positive constant, $\Lambda=\Lambda(\bar b,\delta,K,m,n,T)$, such that 
\begin{equation}
\label{eq:Exp_Girsanov_theorem}
\EE_{\widehat \PP^z} \left[e^{9\int_0^T |\theta(\widehat Z(t))|^2\, dt}\right] \leq \Lambda,\quad\forall\, z \in\bar S_{n,m},
\end{equation}
where $\{\widehat Z(t)\}_{t\geq 0}$ is the unique weak solution that satisfies the Markov property to the stochastic differential equation \eqref{eq:Kimura_SDE_singular_new}, with initial condition $\widehat Z(0)=z$.
\end{lem}
We obtain from Lemma \ref{lem:Exp_Girsanov_theorem} and \cite[Proposition 3.5.12 and Corollary 3.5.13]{KaratzasShreve1991} that the process
\begin{equation*}
\widehat M(t) := e^{\int_0^t \theta(\widehat Z(s))\cdot\, d\widehat W(s) - \frac{1}{2} \int_0^t |\theta(\widehat Z(s))|^2\, ds},\quad\forall\, t\in [0,T],
\end{equation*}
is a $\widehat\PP^z$-martingale. From Girsanov's Theorem (\cite[Theorem 3.5.1]{KaratzasShreve1991}), by letting
\begin{equation}
\label{eq:New_Brownian_motion}
 W(t):= \widehat W(t)-\int_0^t \theta(\widehat Z(s))\, ds,\quad\forall\, t\in [0,T],
\end{equation}
and defining a new probability measure $\PP^z$ by 
\begin{equation}
\label{eq:Change_of_measure}
\frac{d\PP^z}{d\widehat\PP^z} = M(T),
\end{equation}
we obtain that $\{W(t)\}_{t\geq 0}$ is a $\PP^z$-Brownian motion. We also have that $(Z(t):=\widehat Z(t), W(t))$ is a weak solution to the standard Kimura stochastic differential equation \eqref{eq:Kimura_SDE} with initial condition $Z(0)=z$, for all $0\leq t\leq T$. Because $Z(t):=\widehat Z(t)$, for all $0\leq t\leq T$, we also have that $\tau_{\Omega}=\widehat\tau_{\Omega}$, for all open sets $\Omega\subseteq S_{n,m}$, where the preceding two stopping times are defined by \eqref{eq:Exit_time_Omega} and \eqref{eq:hat_tau_Omega}, respectively. From \cite[Lemma 3.9]{Pop_2013a}, it follows that the process
\begin{equation}
\label{eq:Martingale_change_of_measure}
M(t) := e^{-\int_0^t \theta(Z(s))\cdot\, d W(s) - \frac{1}{2} \int_0^t |\theta(Z(s))|^2\, ds},\quad\forall\, t\in [0,T],
\end{equation}
is a $\PP^z$-martingale.

We can now state the stochastic representation of solution to the homogeneous initial-value problem with respect to the probability distribution $\PP^z$. Identity \eqref{eq:Change_of_measure} and Lemma \ref{lem:Stochastic_representation_widehat_L_local} imply

\begin{lem}[Stochastic representation of solutions to equation \eqref{eq:Initial_value_problem_widehat_L} with respect to $\PP^z$]
\label{lem:Stochastic_representation_L_local}
Suppose that the hypotheses of Lemma \ref{lem:Stochastic_representation_widehat_L_local} hold. Then we have that
\begin{equation}
\label{eq:Stochastic_representation_L_local}
\begin{aligned}
u(t,z) &= \EE_{\PP^z} \left[M((t-t_1)\wedge\tau_{\Omega})u\left(t-(t-t_1)\wedge\tau_{\Omega}, Z((t-t_1)\wedge\tau_{\Omega})\right)\right]\\
&\quad+\EE_{\PP^z}\left[M((t-t_1)\wedge\tau_{\Omega})\int_0^{(t-t_1)\wedge\widehat\tau_{\Omega}} g(t-s,\widehat Z(s))\, ds\right],
\quad\forall\, (t,z) \in \bar Q,
\end{aligned}
\end{equation}
where the process $\{Z(t)\}_{t\geq 0}$ is the unique weak solution to the stochastic differential equation \eqref{eq:Kimura_SDE_singular_new}, with initial condition $Z(0)=z$.
\end{lem}
The stochastic representation \eqref{eq:Stochastic_representation_L_local} shows that the parabolic problem $u_t-\widehat Lu=0$ admits probabilistic solutions in the sense of \cite[Definition (2.1)]{Sturm_1994}. Thus the following definition is not vacuous

\begin{defn}[Probabilistic solution]
\label{defn:Probabilistic_solution}
Let $\Omega\subseteq S_{n,m}$ be an open set, and $(t_1,t_2)\subset\RR_+$. Let $Q:=(t_1,t_2)\times\Omega$. We say that a \emph{continuous} function, $u:Q\rightarrow\RR$, is a probabilistic solution to the parabolic equation $u_t-\widehat L u=0$ on $Q$, if for all open sets $Q':=(t'_1,t'_2)\times\Omega'\subseteq Q$, we have that
\begin{equation}
\label{eq:Probabilistic_solution}
u(t,z) = \EE_{\PP^z} \left[M((t-t'_1)\wedge\tau_{\Omega'})u\left(t-(t-t'_1)\wedge\tau_{\Omega'},Z((t-t'_1)\wedge\tau_{\Omega'})\right)\right],\quad\forall\, (t,z) \in \bar Q'.
\end{equation}
\end{defn}

\begin{rmk}
Notice that unlike the definition of the probabilistic solution in \cite[p.~596]{Sturm_1994}, we assume that our probabilistic solution in Definition \ref{defn:Probabilistic_solution}, is continuous, as opposed to only Borel measurable, since the expression inside the expectation in identity \eqref{eq:Probabilistic_solution} is evaluated at points on the boundary of a cylinder, which is a set of measure zero.
\end{rmk}

\subsection{The proof of Harnack's inequality}
\label{subsec:Proof_of_Harnack} 
We use Definition \ref{defn:Probabilistic_solution} to prove that the Harnack inequality holds for nonnegative probabilistic solutions to equation \eqref{eq:Parabolic_hat_L}. We have the following analogue of \cite[Theorem 1]{Sturm_1994}.
\begin{thm}[Harnack inequality]
\label{thm:Harnack_inequality}
Suppose that Assumption \ref{assump:Coeff_hat_L} holds. Let $c\in (\sqrt{2/3},1)$ and $T>0$. Then there is a positive constant, $H=H(\bar b,\delta,K,K_0,\Lambda,T)$, such that for all $(s,z)\in (0,T)\times \bar S_{n,m}$, and all $R \in (0,\sqrt{s})$, if $u$ is a nonnegative, continuous probabilistic solution to the parabolic equation \eqref{eq:Parabolic_hat_L} on $Q_R(s,z)$, we have that
\begin{equation}
\label{eq:Harnack_inequality}
u(t,w) \leq H u(s,z),\quad\forall (t,w) \in (s-c^2R^2,s-2R^2/3)\times B_{cR}(s,z).
\end{equation}
\end{thm}

We have the following consequence of Theorem \ref{thm:Harnack_inequality}.
\begin{thm}[Scale-invariant Harnack inequality]
\label{thm:Harnack_inequality_scale_invariant}
Suppose that Assumption \ref{assump:Coeff_hat_L} holds, and let $c\in (\sqrt{2/3},1)$. Then there are positive constants, $\alpha$, $\beta$, $\gamma$ and $H=H(\bar b,\delta,K,K_0,\Lambda,T)$, such that $\alpha>\beta$ and the following hold. Let $Q\subset (0,\infty)\times S_{n,m}$ be an open set, and assume that $u$ is a nonnegative probabilistic solution to the equation $u_t-\widehat Lu=0$ on $Q$, in the sense of Definition \ref{defn:Probabilistic_solution}. Then for all $(s,z)\in Q$ and all $R>0$ such that
\begin{equation}
\label{eq:Set_for_s_z}
(s-4R^2,s+R^2)\times B_{4R}(z) \subset Q, 
\end{equation}
we have that
\begin{equation}
\label{eq:Harnack_inequality_scale_invariant}
\sup_{Q^-_{\rho}(s,z)} u\leq H \inf_{Q^+_{\rho}(s,z)} u,\quad\forall\, \rho\in (0,cR),
\end{equation}
where we let
\begin{align}
\label{eq:Q_minus}
Q^-_{\rho}(s,z) &:= (s-\alpha\rho^2,s-\beta\rho^2)\times B_{\rho}(z),\\
\label{eq:Q_plus}
Q^+_{\rho}(s,z) &:= (s,s+\gamma \rho^2)\times B_{\rho}(z).
\end{align}
\end{thm}

\begin{proof}
Let $(s,z)\in Q$ and $R>0$ be such that inclusion \eqref{eq:Set_for_s_z} holds. Let $c\in (\sqrt{2/3},1)$. Inequality \eqref{eq:Harnack_inequality} gives us, for all $r\in (0,R)$, that
$$
u(t,w) \leq H u(s,z),\quad\forall (t,w) \in (s-c^2r^2,s-2r^2/3)\times B_{cr}(s,z),
$$
and by denoting $\rho:=cr$, we obtain that for all $\rho\in (0,cR)$, we have that
\begin{equation}
\label{eq:Harnack_inequality_aux}
u(t,w) \leq H u(s,z),\quad\forall (t,w) \in \left(s-\rho^2,s-\frac{2}{3c^2}\rho^2\right)\times B_{\rho}(s,z).
\end{equation}
Let $d$ be a positive constant chosen such that
\begin{equation}
\label{eq:Choice_constant_d}
d^2 < \max\left\{1,4-\frac{8}{3c^2}\right\},
\end{equation}
and let $(s',z')\in (s,s+d^2)\times B_{\rho}(z)$, where we assume that $\rho\in (0,cR)$. Inclusion \eqref{eq:Set_for_s_z} and our choice of the point $(s',z')$ allows us to apply inequality \eqref{eq:Harnack_inequality_aux} with $(s',z')$ replacing $(s,z)$ and $2\rho$ replacing $\rho$, to obtain that
\begin{equation}
\label{eq:Harnack_inequality_aux_aux}
u(t',w') \leq H u(s',z'),\quad\forall (t',w') \in \left(s'-4\rho^2,s'-\frac{8}{3c^2}\rho^2\right)\times B_{2\rho}(z').
\end{equation}
Notice that
\begin{equation}
\label{eq:Inclusion_cylinders}
\left(s-\frac{8}{3c^2}\rho^2,s-(4-d^2)\rho^2\right)\times B_{\rho}(z)\subseteq \bigcap \left(s'-4\rho^2,s'-\frac{8}{3c^2}\rho^2\right)\times B_{2\rho}(z'),
\end{equation}
where the preceding intersection is taken over all points $(s',z')$ in the set $(s,s+d^2\rho^2)\times B_{\rho}(z)$. We now set
$$
\alpha:=\frac{8}{3c^2},\quad\beta:=4-d^2,\quad\hbox{and}\quad\gamma:=d^2.
$$
and notice that our choice of the positive constant $d$ in \eqref{eq:Choice_constant_d} implies that $\alpha>\beta$. Then inequality \eqref{eq:Harnack_inequality_aux_aux} and property \eqref{eq:Inclusion_cylinders} gives us that the scale-invariant Harnack inequality \eqref{eq:Harnack_inequality_scale_invariant} holds. This completes the proof of Theorem \ref{thm:Harnack_inequality_scale_invariant}.
\end{proof}

We prove Theorem \ref{thm:Harnack_inequality} with the aid of a series of lemmas. Our proof follows closely the argument in \cite[\S 2]{Sturm_1994} used to prove \cite[Theorem 1]{Sturm_1994}, but we include the details for clarity and completeness. 

Let $T$ be a positive constant, and $(s,z)\in (0,T)\times\bar S_{n,m}$, and let $r\in (0,\sqrt{s})$. Let $Q_r(s,z):=(s-r^2,s)\times B_r(z)$, where $B_r(z)$ denotes the open ball defined in \eqref{eq:Ball_rho}. We let $\tau_r$ be the stopping time defined by
\begin{equation}
\label{eq:tau_r}
\tau_r := \inf\{t\geq 0:\, Z(t) \notin \bar B_r(z)\}.
\end{equation}
From identity \eqref{eq:Probabilistic_solution}, we obtain that the probabilistic solutions to the parabolic equation \eqref{eq:Parabolic_hat_L} on $Q_R(s,z)$, can be written in the form
\begin{equation}
\label{eq:Stochastic_representation_tau_r}
u(t,w) = \EE_{\PP^w} \left[M((t-s+r^2)\wedge\tau_r)u\left(t-(t-s+r^2)\wedge\tau_r,Z((t-s+r^2)\wedge\tau_r)\right)\right],
\end{equation}
for all $r\in (0,R)$ and all $(t,w)\in Q_r(s,z)$. 

We begin with the analogue of \cite[Lemma 1]{Sturm_1994}.

\begin{lem}[An estimate from below]
\label{lem:Estimate_from_below}
Assume that the hypotheses of Theorem \ref{thm:Harnack_inequality} hold. We then have that
\begin{equation}
\label{eq:Estimate_from_below}
u^{1/3}(s,z) \geq \Lambda^{-1/6} \EE_{\PP^z}\left[\frac{1}{R}\int_0^R u^{1/3}(s-r^2\wedge\tau_r, Z(r^2\wedge\tau_r))\, dr\right].
\end{equation}
\end{lem}

\begin{proof}
For all $r\in (0,R)$, we have that
\begin{align*}
&u^{1/3}(s-r^2\wedge\tau_r, Z(r^2\wedge\tau_r)) \\
&\qquad= \left(M(r^2\wedge\tau_r)u(s-r^2\wedge\tau_r, Z(r^2\wedge\tau_r))\right)^{1/3}	M^{-1/3}(r^2\wedge\tau_r)\\
&\qquad= \left(M(r^2\wedge\tau_r)u(s-r^2\wedge\tau_r, Z(r^2\wedge\tau_r))\right)^{1/3}\\
&\qquad\quad	e^{\frac{1}{3}\int_0^{r^2\wedge\tau_r} \theta(Z(s))\cdot\, dW(s) -\frac{1}{9}\int_0^{r^2\wedge\tau_r} |\theta(Z(s))|^2\, ds} e^{\frac{5}{18}\int_0^{r^2\wedge\tau_r} |\theta(Z(s))|^2\, ds},
\end{align*}
where in the last line we use definition \eqref{eq:Martingale_change_of_measure} of the martingale $\{M(t)\}_{t\geq 0}$. Applying H\"older inequality with $1/3+1/2+1/6=1$, it follows that
\begin{align*}
&\EE_{\PP^z}\left[u^{1/3}(s-r^2\wedge\tau_r, Z(r^2\wedge\tau_r))\right]\\
&\qquad\leq \EE_{\PP^z}\left[(M(r^2\wedge\tau_r)u(s-r^2\wedge\tau_r, Z(r^2\wedge\tau_r))\right]^{1/3}\\
&\qquad\quad \EE_{\PP^z}\left[N(r^2\wedge\tau_r)\right]^{1/2}\EE_{\PP^z}\left[ e^{\frac{5}{3}\int_0^{r^2\wedge\tau_r} |\theta(Z(s))|^2\, ds}\right]^{1/6},
\end{align*}
where we let 
$$
N(t):= e^{\int_0^t \frac{2\theta(Z(s))}{3}\cdot\, dW(s) -\frac{1}{2}\int_0^t\frac{4|\theta(Z(s))|^2}{9}\, ds},\quad\forall\, t\in [0,T].
$$
By Lemma \ref{lem:Exp_Girsanov_theorem} and \cite[Corollary 3.5.13]{KaratzasShreve1991}, we obtain that the process $\{N(t)\}_{0\leq t\leq T}$ is a $\PP^z$-martingale, and so, it follows that $\EE_{\PP^z}[N(r^2\wedge\tau_r)]=1$. Again using Lemma \ref{lem:Exp_Girsanov_theorem}, together with the fact that $u$ satisfies identity \eqref{eq:Stochastic_representation_tau_r}, we obtain that
\begin{align*}
\EE_{\PP^z}\left[u^{1/3}(s-r^2\wedge\tau_r, Z(r^2\wedge\tau_r))\right] &\leq \Lambda^{1/6}u^{1/3}(s,z).
\end{align*}
Integrating the preceding inequality in $r$ from $0$ to $R$, yields estimate \eqref{eq:Estimate_from_below}.
\end{proof}

Let $K_0$ be the positive constant appearing in the statement of Theorem \ref{thm:Stochastic_representation}. For simplicity, given $\rho\in (0,r)$ and $r\in (0,R)$, we denote 
\begin{equation}
\label{eq:Truncated_cylinder}
Q^r_{\rho}(s,z) := Q_{\rho}(s,z) \cap (0,s-2r^2/3) \times S_{n,m}.
\end{equation}
We have the following analogue of \cite[Lemma 3]{Sturm_1994}.

\begin{lem}[Iterated Harnack inequality]
\label{lem:Iterated_Harnack}
There are positive constants, $C=C(K_0)$ and $m=m(K_0)$, such that the following hold. Assume that the hypotheses of Theorem \ref{thm:Harnack_inequality} hold. Let $r\in (\sqrt{2}R/\sqrt{3},R)$ and for all $(t,w) \in Q_r(s,z)$ let
\begin{equation}
\label{eq:v}
v(t,w) := \EE_{\PP^w}\left[u^6(t-(t-s+r^2)\wedge\tau_r, Z((t-s+r^2)\wedge\tau_r)) \mathbf{1}_{\{t-(t-s+r^2)\wedge\tau_r<s-2R^2/3\}}\right].
\end{equation}  
Then, for all $\rho\in (\sqrt{2}R/\sqrt{3},r)$, we have that
\begin{equation}
\label{eq:Iterated_Harnack}
v(t,w) \leq C \left(\frac{r}{r-\rho}\right)^m v(s,z),\quad\forall\, (t,w) \in Q^R_{\rho}(s,z).
\end{equation}
\end{lem}

\begin{proof}
Because the function $u$ is assumed to be continuous by the hypotheses of Theorem \ref{thm:Harnack_inequality}, it follows from Corollary \ref{cor:Stochastic_representation} that the function $v$ defined in \eqref{eq:v} satisfies the Harnack inequality \eqref{eq:Harnack_cont}. We obtain that, for all $\eta\in (0,r)$, we have
\begin{equation}
\label{eq:Harnack_aux}
v(t,w) \leq K_0 v(t',w'), 
\end{equation}
for all $(t,w) \in \left[t'-3\eta^2, t'-2\eta^2\right]\times B_{\eta}(z)$, and for all $(t',w')\in Q_r(s,z)$, chosen such that $Q_{2\eta}(t',w') \subseteq Q_r(s,z)$. Applying inequality \eqref{eq:Harnack_aux} with the choice $(t',w')=(s,z)$ and $\eta=r/2$, we obtain that
\begin{equation}
\label{eq:First_iteration}
v(t,w) \leq K_0 v(s,z),\quad\forall\, (t,w) \in P^1_r(s,w),
\end{equation}
where we let
$$
P^1_r(r,z):=\left[s-3\left(\frac{r}{2}\right)^2, s-2\left(\frac{r}{2}\right)^2\right]\times B_{r/2}(z).
$$
Next, we apply inequality \eqref{eq:Harnack_aux} with the choice $(t',w') \in P^1(s,z)$ and $\eta=r/2^2$, and using inequality \eqref{eq:First_iteration}, we obtain that
$$
v(t,w) \leq K_0^2 v(s,z),\quad\forall\, (t,w) \in P^2_r(s,w),
$$
where we now let
$$
P^2_r(s,z):=\left[s-3\left(\frac{r}{2}\right)^2 -3\left(\frac{r}{2^2}\right)^2, s-2\left(\frac{r}{2}\right)^2-2\left(\frac{r}{2^2}\right)^2\right]\times B_{r/2+r/2^2}(z).
$$
Inductively, we obtain after $k$ steps that
\begin{equation}
\label{eq:kth_Harnack_iteration}
v(t,w) \leq K_0^k v(s,z),\quad\forall\, (t,w) \in P^k_r(s,w),
\end{equation}
where we have that
$$
P^k_r(r,z):=\left[s-\alpha_k, s-\beta_k\right]\times B_{\gamma_k}(z).
$$
The constants $\alpha_k$, $\beta_k$ and $\gamma_k$ are given by
\begin{align*}
\alpha_k&:=	3\left(\frac{r}{2}\right)^2 +3\left(\frac{r}{2^2}\right)^2+\ldots+3\left(\frac{r}{2^k}\right)^2 = (1-4^{-k}) r^2,\\
\beta_k &:=	2\left(\frac{r}{2}\right)^2+2\left(\frac{r}{2^2}\right)^2\ldots+2\left(\frac{r}{2^k}\right)^2= \frac{2}{3}(1-4^{-k}) r^2 ,\\
\gamma_k&:= \frac{r}{2} + \frac{r}{2^2}+ \ldots+ \frac{r}{2^k} = (1-2^{-k}) r.
\end{align*}
Let $\rho\in (0,r)$, and let $k_0\in\NN$ be the smallest integer such that 
$$
(1-2^{-k_0})r \geq \rho,\quad\hbox{and}\quad (1-4^{-k_0})r^2 \geq \rho^2.
$$
We see from \eqref{eq:Truncated_cylinder} that $Q^R_{\rho}(s,z) \subset P^{k_0}_r(s,z)$, and that we can find a positive constant, $c$, such that
$$
\frac{1}{\ln 2}\ln \frac{r}{r-\rho} \leq k_0 \leq \frac{1}{\ln 2}\ln \frac{r}{r-\rho} + c.
$$
The preceding inequality together with estimate \eqref{eq:kth_Harnack_iteration} implies that there are positive constants, $C=C(K_0)$ and $m=\ln K_0/\ln 2$, such that inequality \eqref{eq:Iterated_Harnack} holds.
\end{proof}

We now have the analogue of \cite[Lemma 2]{Sturm_1994}.

\begin{lem}[An intermediary estimate from above]
\label{lem:Estimate_from_above_aux}
There are positive constants, $C=C(K_0)$ and $m=m(K_0)$, such that the following hold. Assume that the hypotheses of Theorem \ref{thm:Harnack_inequality} hold. Then, for all $\rho\in (\sqrt{2}R/\sqrt{3}, R)$ and all $\eta\in(\rho,R)$, we have that
\begin{equation}
\label{eq:Estimate_from_above_aux}
\begin{aligned}
u^6(t,w) &\leq C\Lambda^3\frac{\eta^{m+1}}{(\eta-\rho)^{m+1}}\\
&\quad \EE_{\PP^z}\left[\frac{1}{\eta}\int_0^{\eta} u^6(s-r^2\wedge\tau_r,Z(r^2\wedge\tau_r))\mathbf{1}_{\{s-r^2\wedge\tau_r \leq s-2R^2/3\}}\, dr\right],
\end{aligned}
\end{equation}
for all $(t,w)\in Q^R_{\rho}(s,z)$.
\end{lem}

\begin{proof}
From identity \eqref{eq:Martingale_change_of_measure}, we have for all $r\in (0,R)$ and all $(t,w)\in Q_r(s,z)$, that
\begin{align*}
u(t,w)&= \EE_{\PP^z}\left[e^{-\int_0^{(t-s+r^2)\wedge\tau_r} \theta(Z(s))\cdot\, dW(s) - \frac{3}{2}\int_0^{(t-s+r^2)\wedge\tau_r}|\theta(Z(s))|^2\, ds}\right.\\
&\qquad\left. e^{\int_0^{(t-s+r^2)\wedge\tau_r}|\theta(Z(s))|^2\, ds} u(t-(t-s+r^2)\wedge\tau_r, Z((t-s+r^2)\wedge\tau_r))\right].
\end{align*}
Applying H\"older inequality with $1/3+1/2+1/6=1$, it follows that
\begin{align*}
u(t,w)&\leq 	
\EE_{\PP^w}\left[e^{-\int_0^{(t-s+r^2)\wedge\tau_r} 3\theta(Z(s))\cdot\, dW(s) - \frac{1}{2}\int_0^{(t-s+r^2)\wedge\tau_r}|3\theta(Z(s))|^2\, ds}\right]^{1/3} \\
&\quad\EE_{\PP^w}\left[e^{\int_0^{(t-s+r^2)\wedge\tau_r}2|\theta(Z(s))|^2\, ds}\right]^{1/2} \EE_{\PP^w}\left[u^6(t-(t-s+r^2)\wedge\tau_r, Z((t-s+r^2)\wedge\tau_r))\right]^{1/6}.
\end{align*}
From Lemma \ref{lem:Exp_Girsanov_theorem} and \cite[Corollary 3.5.13]{KaratzasShreve1991}, it follows that the process
$$
N(t):= e^{-\int_0^t 3\theta(Z(s))\, dW(s) - \frac{1}{2}\int_0^t|3\theta(Z(s))|^2\, ds},\quad\forall\, t\in [0,T],
$$
is a $\PP^z$-martingale, and so, using the fact that $\tau_r \leq T$, we have that 
$$
\EE_{\PP^z}[N((t-s+r^2)\wedge\tau_r)]=1.
$$
Using again Lemma \ref{lem:Exp_Girsanov_theorem}, we have that
$$
\EE_{\PP^w}\left[e^{\int_0^{(t-s+r^2)\wedge\tau_r}2|\theta(Z(s))|^2\, ds}\right] \leq \Lambda,
$$ 
and so, we obtain that
$$
u(t,w)\leq  \Lambda^{1/2} \EE_{\PP^w}\left[u^6(t-(t-s+r^2)\wedge\tau_r, Z((t-s+r^2)\wedge\tau_r))\right]^{1/6}.
$$
Let $\rho\in(\sqrt{2}R/\sqrt{3}, R)$, $r\in(\rho, R)$ and $(t,w)\in Q^R_{\rho}(s,z)$. By definition \eqref{eq:Truncated_cylinder} of the set $Q^R_{\rho}(s,z)$, we have that $t \leq s-2R^2/3$, and also that 
$$
t-(t-s+r^2)\wedge\tau_r \leq s-2R^2/3, \quad\forall\, r \in (\rho,R),
$$
which gives us 
$$
u^6(t,w)\leq  \Lambda^3 \EE_{\PP^w}\left[u^6(t-(t-s+r^2)\wedge\tau_r, Z((t-s+r^2)\wedge\tau_r))\mathbf{1}_{\{t-(t-s+r^2)\wedge\tau_r <s-2R^2/3\}}\right],
$$
for all $(t,w) \in Q^R_{\rho}(s,z)$. We can now apply Lemma \ref{lem:Iterated_Harnack} to obtain that, for all $r\in (\rho,R)$ and for all $(t,w) \in Q^R_{\rho}(s,z)$, we have that
$$
u^6(t,w)\leq  C\Lambda^3 \left(\frac{r}{r-\rho}\right)^m\EE_{\PP^z}\left[u^6(s-r^2\wedge\tau_r, Z(r^2\wedge\tau_r))\mathbf{1}_{\{s-r^2\wedge\tau_r <s-2R^2/3\}}\right].
$$
Let now $\eta\in(\rho,R)$. Integrating the preceding inequality in $r$ from $(\rho+\eta)/2$ to $\eta$, we obtain estimate \eqref{eq:Estimate_from_above_aux}.
\end{proof}

We now have the analogue of \cite[Lemma 4]{Sturm_1994}.

\begin{lem}[An estimate from above]
\label{lem:Estimate_from_above}
Let $c\in (\sqrt{2}/\sqrt{3},1)$. Then there is a positive constant, $C=C(c,K_0,\Lambda)$, such that the following hold. Assume that the hypotheses of Theorem \ref{thm:Harnack_inequality} hold. Then we have, for all $(t,w) \in Q^R_{cR}(s,z)$,
\begin{equation}
\label{eq:Estimate_from_above}
u^{1/3} (t,w) \leq C \EE_{\PP^z}\left[\frac{1}{R}\int_0^R u^{1/3}(s-r^2\wedge\tau_r, Z(\tau_r))\mathbf{1}_{\{s-r^2\wedge\tau_r \leq s-2R^2/3\}}\, dr\right] .
\end{equation}
\end{lem}

\begin{proof}
The method of the proof is based on that of \cite[Lemma 4]{Sturm_1994}, which in turn uses ideas of the \cite[Proof of Lemma 3.2]{Fabes_Stroock_1984}. 
For all $\rho\in (\sqrt{2}R/\sqrt{3},R)$, we let
\begin{equation}
\label{eq:Definition_I}
I\left(\frac{\rho}{R}\right) := \EE_{\PP^z} \left[\frac{1}{\rho}\int_0^{\rho} u^6(s-r^2\wedge\tau_r,Z(r^2\wedge\tau_r)) \mathbf{1}_{\{s-r^2\wedge\tau_r \leq s-2R^2/3\}}\, dr\right]^{1/6},
\end{equation}
and we assume, without loss of generality, that
\begin{equation}
\label{eq:Normalize_I}
\EE_{\PP^z} \left[\frac{1}{R}\int_0^R u^{1/3}(s-r^2\wedge\tau_r,Z(r^2\wedge\tau_r)) \mathbf{1}_{\{s-r^2\wedge\tau_r \leq s-2R^2/3\}}\, dr\right]^3=1.
\end{equation}
From \eqref{eq:Definition_I} and \eqref{eq:Normalize_I}, using the fact that we assume that $\rho\in(\sqrt{2}R/\sqrt{3},R)$, we have that
\begin{align*}
I\left(\frac{\rho}{R}\right) 
&= \EE_{\PP^z} \left[\frac{1}{\rho}\int_0^{\rho} u^{6-1/3}(s-r^2\wedge\tau_r,Z(r^2\wedge\tau_r))\right.\\
&\qquad\left. u^{1/3}(s-r^2\wedge\tau_r,Z(r^2\wedge\tau_r))  \mathbf{1}_{\{s-r^2\wedge\tau_r \leq s-2R^2/3\}}\, dr\right]^{1/6}\\
&\leq C\sup_{Q^R_{\rho}(s,z)} |u|^{17/18}.
\end{align*}
Let $\theta:=17/18$. In the preceding inequality, we use \eqref{eq:Estimate_from_above_aux} applied with $r$ replacing $\eta$, and it follows that there is a positive constant, $C=C(K_0,\Lambda)$, such that for all $r\in (\rho,R)$, we have 
$$
I\left(\frac{\rho}{R}\right) \leq C \left(\frac{r}{r-\rho}\right)^{(m+1)\theta/6} I\left(\frac{r}{R}\right)^{\theta}.
$$
Taking the logarithm in the preceding expression, we obtain that
$$
\ln I\left(\frac{\rho}{R}\right) \leq \ln C + \frac{(m+1)\theta}{6} \ln \frac{r}{r-\rho} + \theta \ln I\left(\frac{r}{R}\right),
$$
Let $\alpha:=\rho/R$ and $\beta:=r/R$. Because we assume that $\rho\in(\sqrt{2}R/\sqrt{3},R)$ and $r\in(\rho,R)$, the preceding expression gives us that, for all $\alpha\in (\sqrt{2}/\sqrt{3}, 1)$ and all $\beta\in (\alpha, 1)$, we have 
$$
\ln I(\alpha) \leq C - \frac{(m+1)\theta}{6} \ln (\beta-\alpha) + \theta \ln I(\beta),
$$
where $C=C(K_0,\Lambda)$ is a possibly different positive constant. We choose $p>1$ such that $1/p>\theta$, and $\alpha=\beta^p$ in the preceding expression. Dividing by $\beta$ and integrating from $c^{1/p}$ to $1$, where we choose $c\in(\sqrt{2}/\sqrt{3},1)$, we can find a positive constant, $C=C(K_0,\Lambda)$, such that
\begin{align*}
\int_{c^{1/p}}^1 \frac{\ln I(\beta^p)}{\beta} \, d\beta 
&\leq C + \theta \int_{c^{1/p}}^1 \frac{\ln I(\beta)}{\beta} \, d\beta\\
&\leq C	 + \theta \int_c^1 \frac{\ln I(\beta)}{\beta} \, d\beta\quad\hbox{(since $c^{1/p}>c$.)}
\end{align*}
Applying the change of variable $u:=\beta^p$ on the left-hand side of the preceding inequalities, we obtain that
$$
\frac{1}{p}\int_{c}^1 \frac{\ln I(\beta)}{\beta} \, d\beta \leq C	 + \theta \int_c^1 \frac{\ln I(\beta)}{\beta},
$$
and using the fact that we choose $p>1$ such that $1/p>\theta$, we obtain that
$$
\int_{c}^1 \frac{\ln I(\beta)}{\beta} \, d\beta \leq C.
$$
Since the function $\beta I(\beta)$ is increasing in $\beta$ by construction, the preceding inequality gives us that $I(c) \leq C$. Assumption \eqref{eq:Normalize_I} can now be removed, and the inequality $I(c) \leq C$ translates into
\begin{align*}
&\EE_{\PP^z} \left[\frac{1}{cR}\int_0^{cR} u^6(s-r^2\wedge\tau_r,Z(r^2\wedge\tau_r)) \mathbf{1}_{\{s-r^2\wedge\tau_r \leq s-2R^2/3\}}\, dr\right]^{1/6}\\
&\quad\quad\leq C \EE_{\PP^z} \left[\frac{1}{R}\int_0^R u^{1/3}(s-r^2\wedge\tau_r,Z(r^2\wedge\tau_r)) \mathbf{1}_{\{s-r^2\wedge\tau_r \leq s-2R^2/3\}}\, dr\right]^3.
\end{align*}
Because we assume that $c\in (\sqrt{2}/\sqrt{3},1)$, the preceding inequality together with estimate \eqref{eq:Estimate_from_above_aux} give us that there is a positive constant, $C=C(c,K_0,\Lambda)$, such that inequality \eqref{eq:Estimate_from_above} holds, for all $(t,w)\in Q^R_{cR}(s,z)$.
\end{proof}

We can now give the
\begin{proof}[Proof of Theorem \ref{thm:Harnack_inequality}]
Inequality \eqref{eq:Harnack_inequality} follows from estimates \eqref{eq:Estimate_from_below} and \eqref{eq:Estimate_from_above}.
\end{proof}

%
%

\bibliography{mfpde}

\def\cprime{$'$} \def\polhk#1{\setbox0=\hbox{#1}{\ooalign{\hidewidth
  \lower1.5ex\hbox{`}\hidewidth\crcr\unhbox0}}} \def\cprime{$'$}
  \def\cprime{$'$} \def\cprime{$'$}
  \def\lfhook#1{\setbox0=\hbox{#1}{\ooalign{\hidewidth
  \lower1.5ex\hbox{'}\hidewidth\crcr\unhbox0}}} \def\cprime{$'$}
  \def\cprime{$'$} \def\cprime{$'$} \def\cprime{$'$} \def\cprime{$'$}
\providecommand{\bysame}{\leavevmode\hbox to3em{\hrulefill}\thinspace}
\providecommand{\MR}{\relax\ifhmode\unskip\space\fi MR }
\providecommand{\MRhref}[2]{%
  \href{http://www.ams.org/mathscinet-getitem?mr=#1}{#2}
}
\providecommand{\href}[2]{#2}
\begin{thebibliography}{10}

\bibitem{Adams_1975}
R.~A. Adams, \emph{Sobolev spaces}, Academic Press, Orlando, FL, 1975.

\bibitem{Athreya_Barlow_Bass_Perkins_2002}
S.~R. Athreya, M.~T. Barlow, R.~F. Bass, and E.~A. Perkins, \emph{Degenerate
  stochastic differential equations and super-{M}arkov chains}, Probab. Theory
  Related Fields \textbf{123} (2002), 484--520.

\bibitem{Baudoin}
F.~Baudoin, \emph{{D}iffusion {P}rocesses and {S}tochastic {C}alculus}, EMS
  textbooks in mathematics, 2014.

\bibitem{Bensoussan_Lions}
A.~Bensoussan and J.~L. Lions, \emph{Applications of variational inequalities
  in stochastic control}, North-Holland, New York, 1982.

\bibitem{Brezis}
H.~Brezis, \emph{Functional analysis, {S}obolev spaces and partial differential
  equations}, Universitext, Springer, New York, 2011.

\bibitem{Epstein_Mazzeo_cont_est}
C.~L. Epstein and R.~Mazzeo, \emph{${C}^ 0$-estimates for degenerate diffusion
  operators arising in population biology}, pp. 65, arXiv:1406.1426.

\bibitem{Epstein_Mazzeo_cont_est_diag}
\bysame, \emph{${C}^ 0$-estimates for diagonal degenerate diffusion operators
  arising in population biology}, pp. 19, preprint.

\bibitem{Epstein_Mazzeo_2010}
\bysame, \emph{Wright-{F}isher diffusion in one dimension}, SIAM J. Math. Anal.
  \textbf{42} (2010), 568--608.

\bibitem{Epstein_Mazzeo_annmathstudies}
\bysame, \emph{Degenerate diffusion operators arising in population biology},
  Annals of Mathematics Studies, Princeton University Press, Princeton, NJ,
  2013, arXiv:1110.0032.

\bibitem{Ethier_Kurtz}
S.~N. Ethier and T.~G. Kurtz, \emph{Markov processes: Characterization and
  convergence}, Wiley, 1985.

\bibitem{Evans}
L.~C. Evans, \emph{Partial differential equations}, American Mathematical
  Society, Providence, RI, 1998.

\bibitem{Fabes_Kenig_Serapioni_1982a}
E.~B. Fabes, C.~E. Kenig, and R.~P. Serapioni, \emph{The local regularity of
  solutions of degenerate elliptic equations}, Comm. Partial Differential
  Equations \textbf{7} (1982), 77--116.

\bibitem{Fabes_Stroock_1984}
E.~B. Fabes and D.~W. Stroock, \emph{The {$L^p$}-integrability of {G}reen's
  functions and fundamental solutions for elliptic and parabolic equations},
  Duke Math. J. \textbf{51} (1984), no.~4, 997--1016.

\bibitem{Feehan_Pop_regularityweaksoln_v3}
P.~M.~N. Feehan and C.~A. Pop, \emph{Degenerate elliptic operators in
  mathematical finance and {H\"o}lder continuity for solutions to variational
  equations and inequalities}, submitted, arXiv:1110.5594v3.

\bibitem{Feehan_Pop_stochrepdirichlet}
\bysame, \emph{Stochastic representation of solutions to stationary and
  evolutionary degenerate variational equations and inequalities with
  {D}irichlet boundary conditions}, Transactions of the American Mathematical
  Society, to appear, arXiv:1204.1317v2.

\bibitem{FriedmanSDE}
A.~Friedman, \emph{Stochastic differential equations and applications}, vol. I,
  II, Academic, New York, 1975 and 1976.

\bibitem{Fukushima_Oshima_Takeda}
M.~Fukushima, Y.~Oshima, and M.~Takeda, \emph{Dirichlet forms and symmetric
  {M}arkov processes}, extended ed., de Gruyter Studies in Mathematics,
  vol.~19, Walter de Gruyter \& Co., Berlin, 2011.

\bibitem{GilbargTrudinger}
D.~Gilbarg and N.~Trudinger, \emph{Elliptic partial differential equations of
  second order}, second ed., Springer, New York, 1983.

\bibitem{KaratzasShreve1991}
I.~Karatzas and S.~E. Shreve, \emph{Brownian motion and stochastic calculus},
  second ed., Springer, New York, 1991.

\bibitem{KarlinTaylor2}
S.~Karlin and Taylor, \emph{A second course on stochastic processes}, Academic,
  New York, 1981.

\bibitem{Kimura_1957}
M.~Kimura, \emph{Some problems of stochastic processes in genetics}, Ann. Math.
  Statist. \textbf{28} (1957), 882--901.

\bibitem{Kimura_1964}
\bysame, \emph{Diffusion models in population genetics}, J. Appl. Probability
  \textbf{1} (1964), 177--232.

\bibitem{Koch}
H.~Koch, \emph{Non-{E}uclidean singular integrals and the porous medium
  equation}, Habilitation Thesis, University of Heidelberg, 1999,
  \url{www.mathematik.uni-dortmund.de/lsi/koch/publications.html}.

\bibitem{Krylov_nonlinear_parabolic_pde}
N.~V. Krylov, \emph{Nonlinear elliptic and parabolic equations of the second
  order}, Mathematics and its Applications (Soviet Series), vol.~7, D. Reidel
  Publishing Co., Dordrecht, 1987.

\bibitem{Lierl_Saloff-Coste_2012}
J.~Lierl and L.~Saloff-Coste, \emph{Parabolic harnack inequality for
  time-dependent non-symmetric {D}irichlet forms}, \url{arXiv:1205.6493v2}.

\bibitem{Lions_Magenes1}
J.-L. Lions and E.~Magenes, \emph{Non-homogeneous boundary value problems and
  applications. {V}ol. {I}}, Springer, New York, 1972.

\bibitem{Moser_1961}
J.~Moser, \emph{On {H}arnack's theorem for elliptic differential equations},
  Comm. Pure Appl. Math. \textbf{14} (1961), 577--591.

\bibitem{Moser_1964}
\bysame, \emph{A {H}arnack inequality for parabolic differential equations},
  Comm. Pure Appl. Math. \textbf{17} (1964), 101--134.

\bibitem{Moser_1967}
\bysame, \emph{Correction to: ``{A} {H}arnack inequality for parabolic
  differential equations''}, Comm. Pure Appl. Math. \textbf{20} (1967),
  231--236.

\bibitem{Moser_1971}
\bysame, \emph{On a pointwise estimate for parabolic differential equations},
  Comm. Pure Appl. Math. \textbf{24} (1971), 727--740.

\bibitem{Pop_2013b}
C.~A. Pop, \emph{{$C^0$}-estimates of solutions to the parabolic equation
  associated to {K}imura diffusions}, pp. 28, arXiv:1406.0742.

\bibitem{Pop_2013a}
\bysame, \emph{Existence, uniqueness and the strong {M}arkov property of
  solutions to {K}imura stochastic differential equations with singular drift},
  pp. 25, arXiv:1406.0745.

\bibitem{Portenko}
N.I. Portenko, \emph{Diffusion processes with unbounded drift coefficient},
  Theoret. Probability Appl. \textbf{20} (1976), 27--37.

\bibitem{Shimakura_1981}
N.~Shimakura, \emph{Formulas for diffusion approximations of some gene
  frequency models}, J. Math. Kyoto Univ. \textbf{21} (1981), no.~1, 19--45.

\bibitem{Stroock_Varadhan}
D.~W. Stroock and S.~R.~S. Varadhan, \emph{Multidimensional diffusion
  processes}, Springer, Berlin, 1979.

\bibitem{Sturm_1994}
K.-T. Sturm, \emph{Harnack's inequality for parabolic operators with singular
  low order terms}, Math. Z. \textbf{216} (1994), no.~4, 593--611.

\bibitem{Sturm_1995}
\bysame, \emph{Analysis on local {D}irichlet spaces-{II}. {U}pper {G}aussian
  estimates for the fundamental solutions of parabolic equations}, Osaka J.
  Math. \textbf{32} (1995), no.~2, 275--312.

\bibitem{Sturm_1996}
\bysame, \emph{Analysis on local {D}irichlet spaces. {III}. {T}he parabolic
  {H}arnack inequality}, J. Math. Pures Appl. (9) \textbf{75} (1996), no.~3,
  273--297.

\end{thebibliography}
\bibliographystyle{amsplain}

\end{document}